\documentclass[11pt]{amsart}

\title[Stable maps and branched shadows of 3-manifolds]
{Stable maps and branched shadows of 3-manifolds}

\author[Ishikawa]{Masaharu Ishikawa}
\thanks{The first-named author is 
supported by the Grant-in-Aid for Scientific Research (C),
JSPS KAKENHI Grant Number 25400078}
\address{Tohoku University, Sendai, 980-8578, Japan}
\email{ishikawa@m.tohoku.ac.jp}

\author[Koda]{Yuya Koda}
\thanks{The second-named author is supported by
Japan Society for Promotion of Science (JSPS) Postdoctoral Fellowships for Research Abroad.}

\address{
Mathematical Institute \newline
\indent Tohoku University, Sendai, 980-8578, Japan \newline
\indent and \newline
\indent (Temporary) Dipartimento di Matematica  \newline
\indent Universit\`{a} di Pisa, Largo Bruno Pontecorvo 5, 56127 Pisa, Italy}
\email{koda@math.tohoku.ac.jp}

\usepackage{amsfonts,amsmath,amssymb,amscd}
\usepackage{amsthm}
\usepackage{latexsym}
\usepackage{graphicx}
\usepackage{psfrag}
\usepackage{color}
\usepackage{xypic}
\usepackage[all]{xy}

\theoremstyle{plain}
\newtheorem*{theorem*}{Theorem}
\newtheorem*{lemma*} {Lemma}
\newtheorem*{corollary*} {Corollary}
\newtheorem*{proposition*}{Proposition}
\newtheorem*{conjecture*}{Conjecture}
\newtheorem{theorem}{Theorem}[section]
\newtheorem{lemma}[theorem]{Lemma}
\newtheorem{corollary}[theorem]{Corollary}
\newtheorem{proposition}[theorem]{Proposition}

\newtheorem{remark}[theorem]{Remark}
\newtheorem{claim}{Claim}

\theoremstyle{remark}

\newtheorem*{definition}{Definition}

\newtheorem{example}{Example}

\theoremstyle{definition}

\newtheoremstyle{citing}% name
  {}%      Space above, empty = `usual value'
  {}%      Space below
  {\itshape}% Body font
  {}%         Indent amount (empty = no indent, \parindent = para indent)
  {\bfseries}% Thm head font
  {.}%        Punctuation after thm head
  {.5em}%     Space after thm head: " " = normal interword space;
  {\thmnote{#3}}% Thm head spec

\theoremstyle{citing}
\newtheorem*{citingtheorem}{} 

\newcommand{\Integer}{\mathbb{Z}}
\newcommand{\Real}{\mathbb{R}}
\newcommand{\Complex}{\mathbb{C}}

\newcommand{\Int}{\mathrm{Int}}

\newcommand{\bsc}{\mathrm{bsc}}
\newcommand{\smc}{\mathrm{smc}}
\newcommand{\cross}{\mathrm{cr}}

\newcommand{\Nbd}{\mathrm{Nbd}}
\newcommand{\vol}{\mathrm{vol}}



\begin{document}

\maketitle

\begin{abstract}
Turaev's shadow can be seen locally as the Stein factorization of a
stable map. 
In this paper, we define the notion of stable map complexity 
for a compact orientable 3-manifold bounded by (possibly empty) tori 
counting, with some weights, the minimal number of singular fibers of codimension 2
of stable maps into the real plane, 
and prove that this number equals the minimal number
of vertices of its branched shadows.
In consequence, we give a complete characterization of
hyperbolic links in the 3-sphere whose exteriors have stable map complexity 1 
in terms of Dehn surgeries,
and also give an observation concerning the coincidence
of the stable map complexity and shadow complexity
using estimations of hyperbolic volumes. 
\end{abstract}

\vspace{1em}

\begin{small}
\hspace{2em}  \textbf{2010 Mathematics Subject Classification}: 
57R45; 57M27, 57N70, 58K15

%(候補)
%
%57Rxx  Differential topology For foundational questions of differentiable manifolds 
%
%\hspace{1em} 57R45   Singularities of differentiable mappings 
%
%57Mxx  Low-dimensional topology 
%
%\hspace{1em} 57M20   Two-dimensional complexes 
%
%\hspace{1em} 57M27   Invariants of knots and 3-manifolds 
%
%57Nxx  Topological manifolds 
%
%\hspace{1em} 57N10   Topology of general 3-manifolds
%
%\hspace{1em} 57N70  Cobordism and concordance 
%
%58Kxx   Theory of singularities and catastrophe theory  
%
%\hspace{1em} 58K05   Critical points of functions and mappings 
%
%\hspace{1em} 58K15   Topological properties of mappings 
%
%\hspace{1em} 58K65   Topological invariants 

\hspace{2em} 
\textbf{Keywords}:
stable map; complexity; branched shadow; hyperbolic volume
\end{small}

\section*{Introduction}

%タイトルの候補：
%\begin{description}
%\item[候補 1]
%Branched shadows and stable maps
%\item[候補 2]
%Branched shadows, stable maps and Stein factorizations
%\item[候補 3]
%Singular fibers of stable maps and branched shadows
%\item[候補 4]
%Complexity of branched shadows and singularities of stable maps
%\item[候補 5]
%Branched shadows, stable maps and hyperbolic volumes. 
%\end{description}

The stable maps play an important role in the study of smooth manifolds. 
They are especially used for obtaining topological information of 
the source manifold from the types of their singularities and singular fibers. 
A typical example is the usage of critical points of a Morse function, 
which is a stable map of a manifold into the real line.
In this paper, we deal with stable maps 
from a closed orientable $3$-manifold $M$ to the real plane $\Real^2$. 
As is well-known, the set $S(f)$ of singular points of a stable map 
$f : M \to \Real^2$ 
consists of 
{\it definite fold} points, {\it indefinite fold} points and {\it cusp} points, see \cite{Lev65, Lev85}. 
In \cite{Lev65} Levine showed that 
the cusp points of each stable map can be eliminated by 
a homotopical deformation, which implies that  
every 3-manifold admits a stable map into $\Real^2$ without
cusp points. 
We note that when $f: M \to \Real^2$ is a stable map without cusp points, 
$S(f)$ forms a link in $M$ and $f|_{S(f)}$ is an immersion with only normal crossings. 
A crossing of $f(S(f))$ is said to be {\it non-simple} if 
only one of the connected components of its preimage contains singularities. 
Burlet-de Rham \cite{BR74} showed that 
the 3-manifolds admitting a stable map with only definite fold points are either the 3-sphere or 
connected sums of $S^2 \times S^1$. 
Saeki \cite{Sae96} generalized this result showing that 
the 3-manifolds admitting a stable map with neither 
non-simple crossings nor cusp points 
are graph manifolds, and vice versa. 
It follows from Saeki's work that we need non-simple crossings to 
construct a stable map of a hyperbolic 3-manifold into the plane.  
Hence it is natural to ask how many non-simple crossings 
a hyperbolic 3-manifold needs to have 
and where is the position of the singular fibers.
Costantino-Thurston \cite{CT08} and Gromov \cite{Gro09} gave, independently, 
a linear lower bound 
%and a quadratic upper bound 
of the number of non-simple crossings of stable maps that a closed 3-manifold admits 
in terms of Gromov norm. 
%The definition of stable maps can be generalized as {\it S-maps} for a link 
%in each compact orientable $3$-manifold with $($possibly empty$)$ boundary consisting of tori. 
%See Section \ref{subsec:Stable maps and their Stein factorizations}  for details.  

On the other hand, Turaev \cite{Tur92, Tur94} introduced the notion of {\it shadows} of 4 and 3-manifolds 
in his great deal of study on quantum invariants. 
Roughly speaking, a shadow of a compact oriented 4-manifold $W$ 
is a special kind of 2-dimensional polyhedron ({\it almost-special polyhedron}) 
$P$ embedded in $W$ in a locally flat way so that 
$W$ collapses onto $P$. 
Then $P$ is also called a shadow of the 3-manifold $\partial W$. 
Shadows provide a combinatorial presentation of both the 4 and 3-manifolds. 
In fact, by means of shadows, we can reconstruct a 4 and 3-manifold in a unique way 
from an almost-special polyhedron $P$ equipped with a suitable coloring of the regions of $P$ 
by half integers, which is called a {\it gleam}. 
%Topological properties for 4 and 3-manifolds derived from their shadows are very important as well as 
%a tool for the study of quantum topology. 
Shadows provide important topological properties for 4 and 3-manifolds 
as well as being an interesting notion for the study of quantum topology. 
We refer the reader to Costantino-Thurston \cite{CT08} for triangulations of 4 and 3-manifolds, 
Costantino \cite{Cos06, Cos08} for Stein structure, Spin$^c$ structure and complex structure of 4-manifolds 
and Martelli \cite{Mar11} for a classification of 4-manifolds 
admitting shadows without vertices. 
See also the survey paper \cite{Cos05b} by Costantino. 
%As in the case of stable maps, 
Shadows can be defined also for a link 
in a compact orientable $3$-manifold with $($possibly empty$)$ boundary consisting of tori. 
In general, shadows are defined for a trivalent graph in each compact, orientable 3-manifold, 
however, we consider only in the above case. 

The {\it Stein factorization} of a map is the space of the connected components of its fibers. 
This is a powerful tool for the study of stable maps, because, in this case, the Stein factorization 
has particularly simple local shapes. 
In fact, in each of the papers \cite{BR74}, \cite{Sae96} and \cite{Lev85} mentioned earlier, 
Stein factorization played an important role. 
In \cite{CT08}, Costantino-Thurston revealed a strong relation 
between the Stein factorizations of stable maps of 3-manifolds into the plane and  
the shadows of the manifolds. 
%Using the relation, they obtained a lower bound of the number of 
%non-simple crossings of stable maps that a 3-manifold $M$ admits in terms of 
%Gromov norm. 
In this paper, we study the detailed relation between {\it branched shadows} and stable maps 
of a 3-manifold (and a link in it). 
Here, a branched shadow is a shadow equipped with an orientation of each of their regions 
in a certain admissible way.  

Let $M$ be a compact, orientable $3$-manifold with $($possibly empty$)$ 
boundary consisting of tori 
and $L$ a $($possibly empty$)$ link in $M$.
We call the minimal number of vertices in any of branched shadows of $(M, L)$ 
the {\it branched shadow complexity} of $(M, L)$, and we denote it by $\bsc(M, L)$. 
We can define another kind of complexity of $(M, L)$ 
using the theory of stable maps. 
Here we note that the definition of stable maps can be naturally generalized  for the above pair $(M, L)$. 
In this paper, we call this map an {\it S-map}. 
See Section \ref{subsec:Stable maps and their Stein factorizations}  for details.  
For an S-map of $(M, L)$ into $\Real^2$,  
the singular fibers over non-simple crossings are divided into two types, 
{\it types $\mathrm{II}^2$ and $\mathrm{II}^3$}, 
 according to their shapes as shown in Figure~\ref{fig:codimension_2_singularities}, 
see Saeki \cite{Sae04}. 
\begin{figure}[htbp]
\begin{center}
\includegraphics[width=6cm,clip]{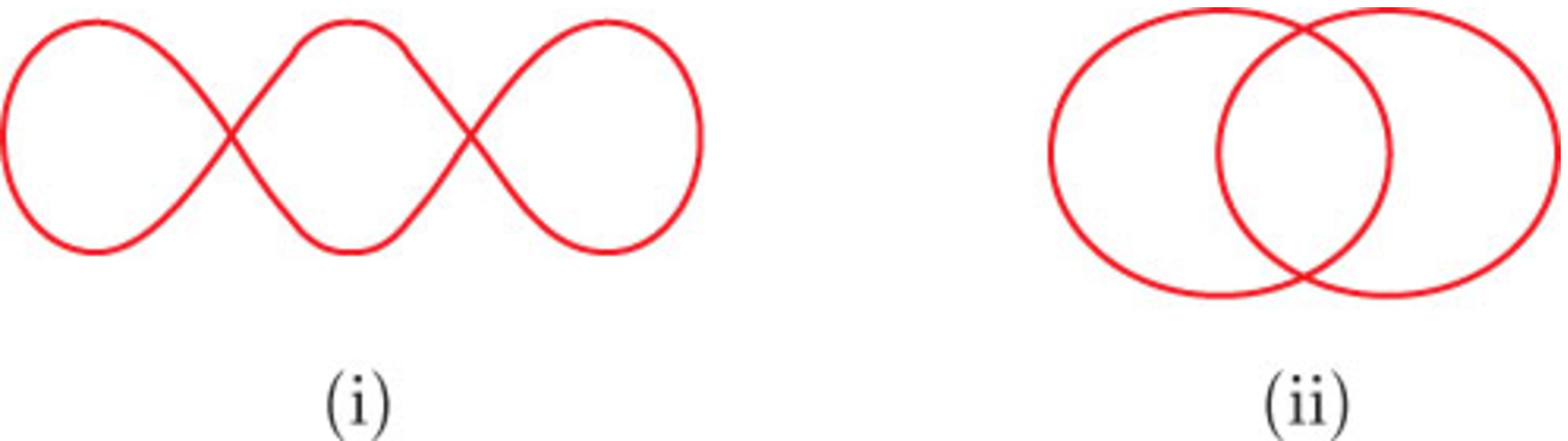}
\caption{(i) A singular fiber of type $\mathrm{II}^2$; (ii) A singular fiber of type $\mathrm{II}^3$.}
\label{fig:codimension_2_singularities}
\end{center}
\end{figure} 
We denote by $\mathrm{II}^2 (f)$ and $\mathrm{II}^3 (f)$ the sets of 
singular fibers of types $\mathrm{II}^2$ and $\mathrm{II}^3$, respectively, 
of an S-map $f:(M,L)\to\Real^2$. 
We call the minimal number of the ``weighted sums" $| \mathrm{II}^2 (f) | + 2 \, |\mathrm{II}^3 (f)| $ for S-maps 
$f: (M, L) \to \Real^2$ 
%where $M$ is a compact orientable $3$-manifold with $($possibly empty$)$ boundary consisting of tori and 
%$L$ is a link in it, 
 the {\it stable map complexity} of $(M, L)$, and we denote it by $\smc(M, L)$. 
The first main result of this paper is the following: 
\begin{citingtheorem}[Theorem \ref{thm:branched shadow complexity and crossing singulatities}]
Let $M$ be a compact, orientable $3$-manifold with $($possibly empty$)$ 
boundary consisting of tori 
and $L$ a $($possibly empty$)$ link in $M$.
Then we have $\bsc (M, L)  = \smc (M, L)$.
\end{citingtheorem}
%\begin{theorem}
%\label{thm:branched shadow complexity and crossing singulatities}
%Let $L$ be a link in a compact orientable $3$-manifold $M$ with $($possibly empty$)$ boundary consisting of tori. 
%Then we have $\bsc (M, L)  = \smc (M, L)$.
%\end{theorem}
By Costantino-Thurston \cite{CT08}, 
the inequality $\mathrm{sc} (M, L)  \leqslant \smc(M, L)$ holds, where 
$\mathrm{sc} (M, L)$ is the {\it shadow complexity} of $(M, L)$, 
that is, the minimal number of vertices in any of its shadows. 
Since a shadow constructed from a stable map in their way is branchable, we have 
the inequality $\bsc (M, L)  \leqslant \smc(M, L)$ immediately. 
To get the other inequality, we construct from a given branched shadow 
a stable map in 2 steps: 
In the first step, we construct an S-map on the submanifold of $M$ corresponding 
to a neighborhood of the singularities of the branched shadow so that 
its Stein factorization is isomorphic to the neighborhood. 
In the second step, we extend this map to whole of $M$ without creating 
non-simple crossings. 
This theorem implies branched shadows of $(M, L)$, which is a purely combinatorial object, 
determine  the minimal number of the weighted sums of non-simple crossings 
of stable maps of $(M, L)$, which is a purely differential value. 
%We note that in \cite{Kod07}, the second-named author showed that 
%the Heegaard genus of a closed orientable 3-manifolds coincides with 
%an invariant defined by branched spines. 
%Here, a spine of a  closed 3-manifold $M$ is an almost-special polyhedron $P$ embedded in $M$ 
%so that $M$ collapsed onto $P$ after removing a small open ball from $M$. 
%In other words, branched spines determine the minimal number of the critical points 
%of Morse functions on $M$. 
%The above theorem can be compared with this fact. 
It is worth noting that in case where we consider a stable map of a closed orientable 
$3$-manifold $M$ into $\Real$, instead of $\Real^2$, 
the minimal number of critical points is determined 
by some specific branched spines of $M$, see Remark \ref{rmk:branched shadow complexity and crossing singulatities} (3). 

Due to Theorem \ref{thm:branched shadow complexity and crossing singulatities}, 
we can use branched shadows to study the behavior of stable map complexities under several operations of 
3-manifolds. 
In fact, the subadditivities of branched shadow complexities under both connected sum and torus sum 
are obtained by easy combinatorial constructions and these provide immediately 
the same properties for stable map complexities 
(cf. Corollaries \ref{cor:S-map of the connected sum}, \ref{cor:S-map of the torus sum}). 
It follows immediately that stable map complexities do not decrease under Dehn filling 
(cf. Corollary \ref{cor:stable map complexities of a link and its exterior}). 
This behavior is similar to that of the hyperbolic volumes
shown by Thurston \cite{Thu80}. 
%The fact that stable map complexities do not decrease under Dehn filling 
%(cf. Corollary \ref{cor:stable map complexities of a link and its exterior}), 
%which is derived directly from the subadditivity of stable map complexities 
%under torus sum, {\color{red} can be compared with} 
%the fact that hyperbolic volumes decrease under Dehn filling shown by Thurston \cite{Thu80}. 

In the remaining part of the paper, we discuss more applications of 
Theorem \ref{thm:branched shadow complexity and crossing singulatities}. 
First, we construct a stable map for a link in the 3-sphere using a branched shadow, 
and introduce a way to determine the configuration of its singular fibers. 
This construction and the result for Dehn filling that mentioned earlier yield the following: 
\begin{citingtheorem}[Corollary \ref{cor:stable map of a surgered manifold}]
Let $M$ be a closed orientable $3$-manifold obtained from $S^3$ by surgery along a link $L$. 
Then there exists a stable map $f: M \to \Real^2$ without cusp points such that 
$| \mathrm{II}^2 (f)  | \leqslant \cross (L) - 2$ and $ \mathrm{II}^3 (f)   = \emptyset$, 
where $\cross (L)$ is the crossing number of $L$. 
%\begin{enumerate}
%\item
%the Stein factorization $W_f$ is homotopy equivalent to the wedge sum of finitely many $2$-spheres; 
%\item
%$\CS_1(f) \leqslant \cross (L) - 2$ and $\CS_2(f) = \emptyset$. 
%\item
%$C(f) = \emptyset$.  
%\end{enumerate} 
\end{citingtheorem}
In \cite{KS12}, Kalm\'{a}r-Stipsicz described the upper bounds of 
the minimal numbers of simple crossings, $\mathrm{II}^2 (f) $, $\mathrm{II}^3 (f) $, etc., for 
stable maps that a closed 3-manifold $M$ admits 
in terms of properties of a surgery diagram of $M$. 
The above corollary provides better upper bounds 
for $| \mathrm{II}^2 (f) |$, $| \mathrm{II}^3 (f) |$ than theirs. 

Next, developing the technique in the above construction, we characterize 
the hyperbolic knots in $S^3$ having branched shadow complexity 1 
as follows: 
%(the links $L_4$, $L_5$, $L_6$ in the theorem are obtained from diagrams of the trivial knot 
%in the above way): 
\begin{citingtheorem}[Theorem  \ref{thm:vertex 1 general}] 
Let $L$ be a hyperbolic link in $S^3$.
Then $\bsc(S^3, L) = 1$ if and only if the exterior of $L$ is diffeomorphic to
a $3$-manifold obtained by Dehn filling the exterior of one of the 
six links $L_1$, $L_2, \ldots, L_6$ in $S^3$ along some of 
$($possibly none of$)$ boundary tori  of its exterior, 
where $L_1$, $L_2, \ldots, L_6$ are 
illustrated in Figure $\ref{fig:vertex_1_general}$.  
\begin{figure}[htbp]
\begin{center}
\includegraphics[width=10cm,clip]{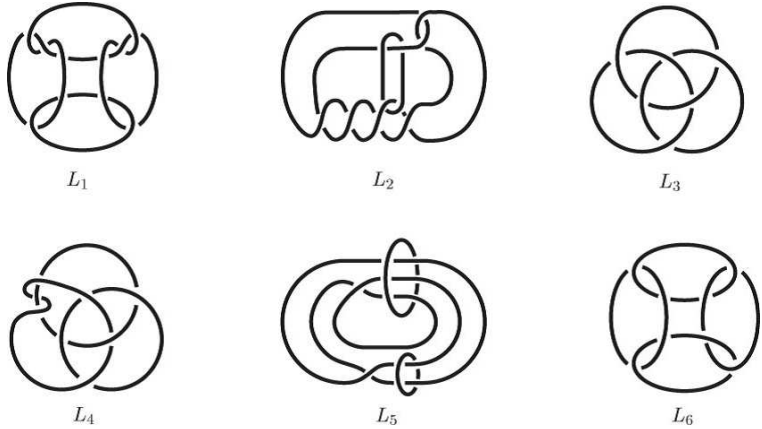}
\caption{The links $L_1, L_2, \ldots, L_6$ in $S^3$.}
\label{fig:vertex_1_general}
\end{center}
\end{figure}
\end{citingtheorem}
Each of the links $L_1$, $L_2, \ldots, L_6$ of this theorem 
is a hyperbolic link having 
the volume $2 V_{\mathrm{oct}}$, 
where $V_{\mathrm{oct}} = 3.66 \dots$ is the volume of the ideal regular octahedron. 
See Costantino-Thurston \cite[Proposition~3.33]{CT08}. 
We note that by Agol-Storm-Thurston \cite[Theorem~9.1]{AST07}, the link $L_1$ 
is a  minimal volume hyperbolic link that contains a meridional incompressible planar surface. 
See also Agol \cite[Example~3.3]{Ago99}. 
In \cite{Yos12} Yoshida proved that the complement of $L_6$ is 
the minimal volume orientable hyperbolic 3-manifold with $4$ cusps.

Theorem \ref{thm:vertex 1 general} can be restated in terms of
singular fibers of stable maps as follows:
\begin{citingtheorem}[Corollary  \ref{cor:Existence of an S-map with smc 1}] 
Let $L$ be a hyperbolic link in $S^3$.
Then there exists a stable map $f : (S^3, L) \to \Real^2$ without cusp points 
such that $| \mathrm{II}^2 (f) | = 1$ and $\mathrm{II}^3 (f) = \emptyset$ 
if and only if 
the exterior of $L$ is diffeomorphic to
a $3$-manifold obtained by Dehn filling the exterior of one of the 
six links $L_1$, $L_2, \ldots, L_6$ in Theorem $\ref{thm:vertex 1 general}$ 
along some of $($possibly none of$)$ boundary tori  of its exterior. 
\end{citingtheorem}
We also describe the configuration of the singular fiber 
of the above S-map for $E(L_i)$, $i=3,4,5,6$ 
in Corollary \ref{cor:configurations of the unique singular fibers}.

Finally, we discuss relation between stable map complexities and hyperbolic volumes. 
Costantino-Thurston \cite{CT08} showed the inequality $\| M\| V_\mathrm{tet} \leqslant 2 \, \text{sc}(M) V_{\mathrm{oct}}$
for any $3$-manifold $M$, where $\| M \|$ is the Gromov norm of $M$
and $V_\mathrm{tet} = 1.01 \ldots$ is the volume of the ideal regular tetrahedron. 
The same argument together with Theorem \ref{thm:branched shadow complexity and crossing singulatities} 
gives the following inequality relating stable map complexities and Gromov norms: 
\[ \| M\| V_{\mathrm{tet}} \leqslant 2 \, \smc (M) V_{\mathrm{oct}}. \]
In particular, if $M$ is hyperbolic then the inequality $\vol (M) \leqslant 2 \, \smc (M) V_{\mathrm{oct}}$ holds. 
The lower bound of $\vol (M)$ by $\smc (M)$ can also be given by using branched, 
{\it special} shadows, where a branched shadow $P$ is special if $P$ is stratified by its singularities 
as a CW-complex. 
We note that every closed orientable 3-manifold admits a branched, special shadow by the moves 
described in \cite{Tur94, Cos05a}. 
By applying Futer-Kalfagianni-Purcell \cite[Theorem~1.1]{FKP08} to our situation, we have
the inequality
\[ 2 \, \smc (M) V_{\mathrm{oct}} \left( 1 - \left( \frac{2\pi}{\mathrm{sl}(P)} \right)^2 \right)^{3/2} \leqslant \vol (M) , \]
where $M$ is a closed orientable hyperbolic $3$-manifold with special shadow $P$, 
and $\mathrm{sl}(P)$ is a certain positive real number determined by the gleam and the topological data of $P$, which 
corresponds to the minimal slope length of Dehn fillings. 
We require $\mathrm{sl}(P)>2\pi$ in the above inequality, 
see Section \ref{sec:Stable maps and hyperbolic volume} for details.
From these inequalities we have the following result that concerns the coincidence of 
shadow complexities, branched shadow complexities and stable map complexities.
\begin{citingtheorem}[Theorem \ref{thm:linear upper and lower bound of sms}]
Let $M$ be a closed orientable $3$-manifold, and let $P$ be a branched, special shadow of $M$.
If $\mathrm{sl}(P) > 2\pi\sqrt{2c(P)}$, then we have
 $\mathrm{sc}(M)=\mathrm{bsc}(M)=\mathrm{smc}(M)=c(P)$.
\end{citingtheorem}
It seems to be hard to compute the shadow complexities, branched shadow complexities and stable map complexities 
by just looking at their definitions, because we need to consider the minimal of infinitely many shadows, branched shadows and 
weighted sums of non-simple crossings, respectively. 
Moreover, it seems to be too optimistic to expect that 
the shadow complexities coincide with the branched shadow complexities, 
since there exist infinitely many shadows that are not branchable. 
However, Theorem \ref{thm:linear upper and lower bound of sms} 
suggests that these three kinds of complexities coincide for many cases. 
\vspace{1em}

This paper is organized as follows. 
In Section \ref{sec:Preliminaries} we review the definitions of branched shadows and stable maps. 
In Section \ref{sec:Branched shadow complexity and stable map complexity}, we provide 
the proof of Theorem \ref{thm:branched shadow complexity and crossing singulatities}, and 
several direct corollaries of it. 
In Section \ref{sec:Stable maps of links}, we introduce a way to construct 
a stable map for each link $L$ in the 3-sphere $S^3$, and 
describe the configuration of the fibers of the map. 
Section \ref{sec:Knots with branched shadow complexity 1} is devoted to a characterization of  
the hyperbolic knots in $S^3$ having the stable map complexity 1 and its applications. 
In Section \ref{sec:Stable maps and hyperbolic volume}, 
we discuss relation between stable map complexities and hyperbolic volumes.

\vspace{1em}

Throughout the paper, we will work in the smooth category unless otherwise mentioned. 

\section{Preliminaries}
\label{sec:Preliminaries}

\subsection{Shadows and branched shadows of 3-manifolds}
\label{subsec:Shadows and branched shadows of 3-manifolds}

A compact topological space $P$ is called 
an {\it almost-special polyhedron} if 
every point of $P$ has a regular neighborhood homeomorphic to 
one of the five local models 
shown in Figure \ref{fig:branched_polyhedron}. 
\begin{figure}[htbp]
\begin{center}
\includegraphics[width=12cm,clip]{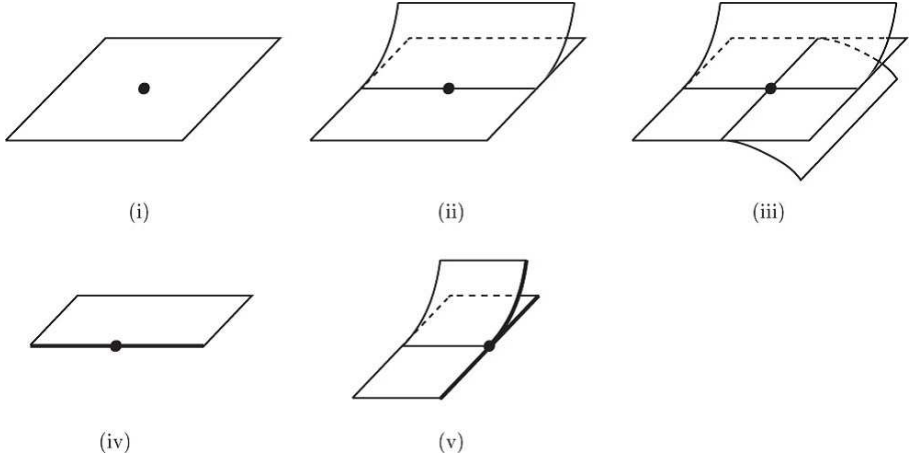}
\caption{The local models of an almost-special polyhedron.}
\label{fig:branched_polyhedron}
\end{center}
\end{figure}
A point whose regular neighborhood is shaped on the model 
(iii) is called a {\it true vertex} of $P$, and we denote the set of true vertices of $P$  by $V(P)$. 
The set of points whose regular neighborhood are shaped on the models (ii), (iii) or (v) is called the {\it singular set} of $P$, 
and we denote it by $S(P)$. 
The set of points whose regular neighborhood are shaped on the models  (iv) or (v) is called the {\it boundary} of $P$ and 
we denote it by $\partial P$. 
A point whose regular neighborhood is shaped on the model 
(v) is called a {\it boundary-vertex} of $P$, and we denote the set of boundary-vertices of $P$  by $\mathit{BV}(P)$. 
Throughout the paper, we set $c(P) = |V(P)| + |\mathit{BV}(P)|$. 
The polyhedron $P$ is said to be {\it closed} if $\partial P = \emptyset$.  
Each component of $P \setminus S(P)$ 
is called a \textit{region}. 
A region is said to be {\it internal} if it does not touch the boundary of $P$. 
Otherwise, it is said to be {\it external}. 

Let $P$ be an almost-special polyhedron. 
A {\it coloring} of $\partial P$  is an arbitrary map 
from the set of components of $\partial P$ to $\{ i, e, f \}$. 
Then with respect to the coloring, $\partial P$ decomposes into three 
peaces $\partial _i P$, $\partial _e P$ and $\partial_f P$. 
An almost-special polyhedron is said to be {\it boundary-decorated} if 
it is equipped with a coloring of $\partial P$. 
If $\partial _f(P) = \emptyset$, $P$ is said to be {\it proper}.

\begin{definition}
Let $M$ be a compact orientable 3-manifold and $L$ a $($possibly empty$)$ link in $M$. 
A boundary-decorated almost-special polyhedron $P$ 
properly embedded in a compact oriented smooth 4-manifold $W$ is called a {\it shadow} of $(M, L)$ if 
\begin{itemize}
\item
$W \setminus P$ is diffeomorphic to $\partial W \times (0, 1]$, or equivalently, 
$W$ collapses onto $P$ after equipping a natural PL structure on $W$; 
\item
$P$ is locally flat, that is, 
each point $p$ of $P$ has a neighborhood $\Nbd(p ; P)$ that lies in a 3-dimensional submanifold of $W$; and 
\item
$(M, L) = ( \partial W \setminus \Nbd (\partial_e P; \partial W),  \partial_i P)$. 
\end{itemize}
When $L = \emptyset$, we say that $P$ is a shadow of $M$ for simplicity. 
\end{definition}
In \cite{Tur92, Tur94}, Turaev proved that any pair 
of a compact orientable 3-manifold with no spherical boundary components
and a $($possibly empty$)$ link in it has a shadow. 
In \cite{Cos05a, CT08}, the {\it shadow complexity} of $(M,L)$, denoted by $\mathrm{sc}(M,L)$, 
was defined to be the minimal number of true and boundary-vertices in any of its shadows.  
\begin{remark}
{\rm
In Costantino-Thurston \cite[Remark~$3.19$]{CT08}, it is proved that, if 
$M$ is a compact orientable 3-manifold with boundary consisting of 
$($possibly empty$)$ tori, then $\mathrm{sc}(M,L)$ 
coincides with the minimal number of true vertices in any of shadows of $(M, L)$. 
Actually, in their paper, $\mathrm{sc} (M, L)$ was defined using only the true vertices. }
\end{remark}

A {\it branching} of an almost-special polyhedron $P$ is an orientation of each
region of $P$ such that 
the orientations on each component of $S(P) \setminus V(P)$ induced by the regions does not coincide. 
%no component of $S(P) \setminus V(P)$ 
%is induced the same orientation by the regions touching it three times.
A branching of $P$ allows us to smoothen $P$ as in the local models in Figure \ref{fig:branched_polyhedron}. 
We note that even though each region of an almost-special polyhedron $P$ is orientable, 
$P$ does not necessarily admit a branching. 
See Ishii \cite{Ish86}, Benedetti-Petronio \cite{BP97} and Petronio \cite{Pet12} for general properties of branched polyhedra. 

\begin{definition}
Let $M$ be a compact orientable 3-manifold and $L$ a $($possibly empty$)$ link in $M$. 
A shadow $P$ of $(M, L)$ equipped with a branching is called a {\it branched shadow} of $(M, L)$. 
\end{definition}
In \cite[Theorem~3.1.7]{Cos05a} and \cite[Proposition~3.4]{Cos08}, Costantino described 
an algorithmic procedure to obtain a branched shadow from an arbitrary shadow through a finite 
sequence of moves. 
As a consequence, any pair 
of a compact orientable 3-manifold with no spherical boundary components
and a $($possibly empty$)$ link in it has a branched shadow. 
\begin{definition}
\label{def:branched shadow complexity}
Let $M$ be a compact orientable 3-manifold and $L$ a $($possibly empty$)$ link in $M$. 
The {\it branched shadow complexity} of $(M,L)$, denoted by $\bsc(M,L)$, 
is the minimal number of true and boundary-vertices in any of its branched shadows.  
%is the minimal number of true and boundary-vertices of any branched shadow of $(M, L)$.  
A branched shadow $P$ of $(M,L)$ is said to be {\it minimal} if it satisfies  
$c (P) = \bsc(M,L)$, that is, it contains the least possible number of 
true and boundary-vertices. 
\end{definition}
\begin{remark}
{\rm
At the moment, we do not know whether $\bsc(M,L)$ coincides with 
the minimal number of only true vertices in any of branched shadows of $(M, L)$ 
if $M$ is a compact orientable 3-manifold with boundary consisting of 
$($possibly empty$)$ tori. 
More generally, we do not know whether $\mathrm{sc} (M, L)$ coincides with $\bsc (M, L)$. 
See Theorem $\ref{thm:linear upper and lower bound of sms}$ in the last section of this paper.} 
\end{remark}

A {\it gleam} on an almost-special polyhedron $P$ is a coloring of all the interior regions of $P$ 
with half integers satisfying a certain condition. 
We call an almost-special polyhedron $P$ equipped with gleams a {\it shadowed polyhedron}. 
In \cite{Tur92, Tur94}, Turaev showed the following: 
\begin{enumerate}
\item
If an almost-special polyhedron $P$ is embedded in a compact oriented smooth 4-manifold $W$ so that 
$P$ is locally-flat, 
there exists a canonical coloring of the interior regions of $P$ with half integers, that is, 
we have the canonical gleam on $P$. 
\item 
({\it Turaev's reconstruction}) 
From a shadowed polyhedron $P$, 
we can reconstruct a compact oriented smooth 4-manifold $W$ and an embedding $P \hookrightarrow W$ 
in a unique way (up to diffeomorphism) so that 
$W$ collapses onto $P$ and the canonical gleam on $P$ given by the embedding $P \hookrightarrow W$ 
coincides with the prefixed gleam on $P$. 
\end{enumerate}
%In Turaev \cite{Tur94}, it is proved that from a shadowed polyhedron $P$, 
%we can reconstruct an oriented 4-manifold $W$ and an embedding $P \hookrightarrow W$ 
%in a unique way (up to PL homeomorphism) so that 
%$W$ collapses onto $P$. 
Here we briefly explain the reconstruction of an oriented 4-manifold from a shadowed polyhedron $P$ 
only in the case where 
$\partial P = \emptyset$, $S(P) \neq \emptyset$ and $S(P)$ is connected. 
%the set of whose regions 
%consists of disks. 
For the detailed and general construction we refer the reader to Costantino \cite{Cos05a} 
(see also Thurston \cite{Thu02} and Costantino-Thurston \cite{CT08}).  
Let $P$ be a shadowed polyhedron without boundary such that 
$S(P)$ is non-empty and connected. 
The polyhedron $\Nbd(S(P); P)$ decomposes into pieces each of which is homeomorphic 
to a compact surface or one of the local models (ii) and (iii) shown in Figure \ref{fig:branched_polyhedron}. 
For each piece homeomorphic to the models (ii) or (iii), we consider its 3-dimensional thickening and then 
glue all these pieces following a natural instruction given by the combinatorial structure of $P$. 
We note that the resulting 3-manifold $X$ is not necessarily orientable. 
Let $W_0$ be the subbundle of the determinant line bundle over $X$ whose fiber is 
$[-1, 1]$ (after giving an Euclidean metric over this bundle). 
This is the unique $[-1, 1]$ bundle over $X$ whose total space is orientable. 
Let $R$ be a component of $P \setminus \Int \, \Nbd(S(P); P)$. 
Let $W_R$ be the subbundle of the determinant line bundle over $R \times [-1,1]$ whose fiber is 
$[-1, 1]$. 
When $R$ is an orientable surface, $W_R$ is nothing else but $R \times [-1, 1] \times [-1, 1]$. 
The combinatorial structure of $P$ again tells us which  
part of $\partial W_R$, which is a solid torus, 
will be glued to which part of $\partial W_0$, which is also a torus, to obtain the required 4-manifold $W$. 

Let $l_1$, $l_2, \ldots, l_k$ be $k$ simple closed curves constituting $\partial \Nbd(S(P); P) \cap R$. 
For each $i \in \{ 1 , 2, \ldots, k \}$, the 3-dimensional thickening of $\Nbd(S(P); P)$ 
provides an annulus or a M\"{o}bius band $B_i$ whose core is $l_i$. 
Similarly, the 3-dimensional thickening $R \times [-1, 1]$ 
provides an annulus $A_i$ whose core is $l_i$. 
Now we have two "framings" $A_i$ and $B_i$ of the same simple closed curve $l_i$, which lie in the 
same solid torus after the earlier-mentioned identification. 
Suppose that $B_i$ is obtained from $A_i$ by twisting $n_i$ times 
($n_i \in (1/2) \Integer$). See Figure \ref{fig:framing_and_gleam}. 
\begin{figure}[htbp]
\begin{center}
\includegraphics[width=6cm,clip]{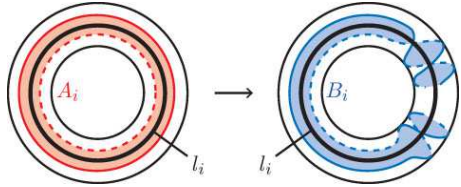}
\caption{The framings $A_i$ and $B_i$.}
\label{fig:framing_and_gleam}
\end{center}
\end{figure} 
The gleam of the region corresponding to $R$ is the sum $\Sigma_{i=1}^k n_i$. 
%{\color{red} (Gleam と region の内点の collapsing に関する逆像での surgery の関係について簡単に言及する．)}

\subsection{Stable maps and their Stein factorizations}
\label{subsec:Stable maps and their Stein factorizations}

Let $M$ be a closed orientable 3-manifold. 
%and let $N$ be a closed or open 2-manifold. 
Let $f$ be a map of $M$ into an orientable 2-manifold $\Sigma$. 
We denote by $S(f)$ the set of singular points of $f$, that is, 
$S(f) = \{ p \in M \mid \mathrm{rank}~ df_p < 2\} $. 
A map $f$ of $M$ into $\Sigma$ is said to be {\it stable} 
if there exists an open neighborhood of $f$ in $C^\infty(M, \Sigma)$
such that for any map $g$ in this neighborhood there exist
diffeomorphisms $\Phi : M \to M$ and $\varphi : \Sigma\to \Sigma$ satisfying
$g=\varphi\circ f\circ \Phi^{-1}$.
Here $C^\infty(M,\Sigma)$ is the set of smooth maps of $M$ into $\Sigma$
with the Whitney $C^\infty$ topology.
If $f$ is stable, there exist local coordinates centered
 at $p$ and $f(p)$ such that 
$f$ is locally described in one of the following way: 
\begin{enumerate}
\item
$(u,x,y) \mapsto (u,x)$; 
\item
$(u,x,y) \mapsto (u,x^2 + y^2)$; 
\item
$(u,x,y) \mapsto (u,x^2 - y^2)$; 
\item
$(u,x,y) \mapsto (u,y^2 + ux -x^3)$. 
\end{enumerate}
(In the cases of (1), (2), (3) and (4), $p$ is called a {\it regular point}, a {\it definite fold point}, 
an {\it indefinite fold point} and a {\it cusp point}, respectively.)	
Further, we require that 
\begin{enumerate}
\setcounter{enumi}{4}
\item
$f^{-1} \circ f(p) \cap S(f) = \{ p \}$ for a cusp point $p$;  
\item
restriction of $f$ to $S(f) \setminus \{\mbox{cusp points}\}$ is an immersion with only normal crossings.
\end{enumerate}
Conversely, if a smooth map satisfies the above conditions, then it is a stable map.
The stable maps form an open dense set in the space $C^\infty (M, \Sigma)$. 
%For general definition and properties of stable maps, see e.g. Levine \cite{Lev85} and Saeki \cite{Sae04}. 

Let $M$ be a compact orientable 3-manifold with $($possibly empty$)$ boundary consisting of tori. 
A smooth map $f$ of $M$ into an orientable 2-manifold $\Sigma$ is called an {\it S-map} if 
\begin{enumerate}
\item 
the restriction of $f$ to $\Int \, M$ is a stable map 
(here a stable map means that, as in the case where $M$ is closed, 
there exists an open neighborhood of $f$ in $C^\infty(\Int \, M, \Sigma)$
such that for any map $g$ in this neighborhood there exist
diffoemorphisms $\Phi:M\to M$ and $\varphi:\Sigma\to\Sigma$
satisfying $g=\varphi \circ f \circ \Phi^{-1}$);
\item
for each $p \in \partial M$ there exist 
a local coordinate $(u,x,y)$ centered at $p$, where 
$\partial M$ corresponds to $\{ y=0 \}$, 
and a local coordinate of $f(p)$ such that 
$f$ is locally described as $(u, x, y) \mapsto (u,x)$.  
\end{enumerate} 
As in Saeki \cite{Sae96}, we denote by $S_0(f)$, $S_1(f)$ and $C(f)$ the sets of definite fold, indefinite fold, cusp 
points, respectively, of the restriction of $f$ to $\Int \, M$. 

Let $f$ be an $S$-map of a compact orientable 3-manifold $M$ with $($possibly empty$)$ boundary consisting of tori into 
an orientable 2-manifold $\Sigma$. 
We say that two points $p_1$ and $p_2$ are {\it equivalent} if they are 
contained in the same component of the fibers of $f$. 
We denote by $W_f$ the quotient space of $M$ with respect to the equivalence relation 
and by $q_f $ the quotient map. 
We define the map $\bar{f} : W_f \to N$ so that $f = \bar{f} \circ q_f$. 
The quotient space $W_f$, or the composition $\bar{f} \circ q_f$ is called the {\it Stein factorization} of $f$. 
The Stein factorization $W_f$ is homeomorphic to a polyhedron, that is, the underlying space of a finite 2-dimensional simplicial complex.

Let $f$ be an $S$-map of a compact orientable 3-manifold $M$ with $($possibly empty$)$ boundary consisting of tori into 
an orientable 2-manifold $\Sigma$. 
By Kushner-Levine-Porto \cite{KLP84} and Levine \cite{Lev85}, the local models of the Stein factorization $W_f$ can be described as follows. 

If $p \in S_0(f)$ ($p\in C(f)$, respectively), 
the projection $q_f(p) \mapsto f(p)$ of the Stein factorization looks locally 
as shown in Figure \ref{fig:definite_fold} (i) (Figure \ref{fig:definite_fold} (ii),
respectively).
\begin{figure}[htbp]
\begin{center}
\includegraphics[width=8cm,clip]{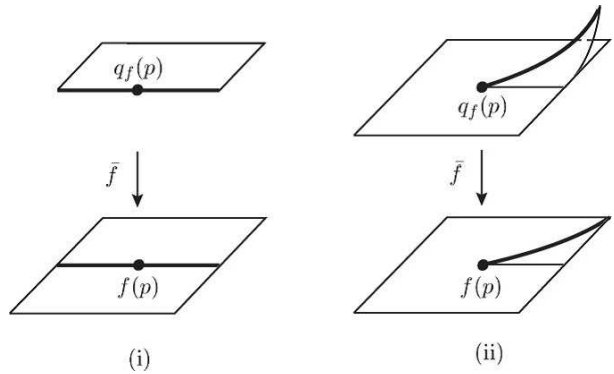}
\caption{The local model of the map $\bar{f} : W_f \to \Sigma$ for: (i) a definite fold point; (ii) a cusp point.}
\label{fig:definite_fold}
\end{center}
\end{figure}
In this case, 
$q_f^{-1} \circ q_f (p)$ 
is called a {\it singular fiber of type $\mathrm{I}^0$}
($\mathrm{II}^a$, respectively). 
See Saeki \cite[Figure~3.4]{Sae04}. 
We remark that, in Saeki's book, these notations are defined  for singular fibers of stable maps of
orientable 4-manifolds into 3-manifolds. However, the classification
of singular fibers of stable maps of orientable 3-manifolds into a
plane coincide with that of singular fibers of codimension 0, 1, 2 of 
stable maps of orientable 4-manifolds into 3-manifolds as mentioned in
\cite[Remark~3.14]{Sae04}. 
For this reason, we consistently use the symbols in \cite[Figure~3.4]{Sae04} 
also for the remaining types of singular fibers described below.

If $p \in S_1(f)$, the projection $q_f(p) \mapsto f(p)$ looks locally as one of the three models 
shown in Figure \ref{fig:indefinite_fold}.  
\begin{figure}[htbp]
\begin{center}
\includegraphics[width=12cm,clip]{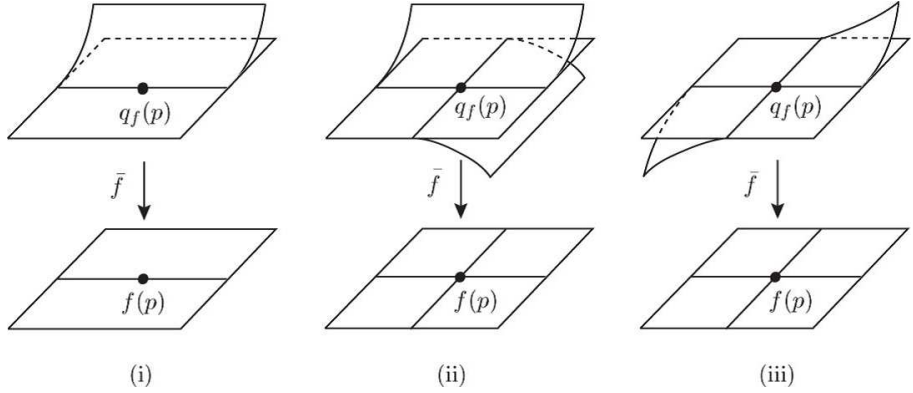}
\caption{The three local models of the map $\bar{f} : W_f \to \Sigma$ for indefinite fold points.}
\label{fig:indefinite_fold}
\end{center}
\end{figure}
Let $U_1, U_2$ and $U_3$ be regular neighborhoods of points $q_f(p)$ of types (i), (ii) and (iii), respectively, shown in the figure. 
Then the preimage $q_f^{-1} (U_1) \subset M$ 
is diffeomorphic to $R \times [0,1]$, where 
$R$ is a pair of pants, and $f$ maps $q_f^{-1} (U_1)$ into $\Sigma$ as depicted in Figure \ref{fig:f-fibration_0}. 
In this case $q_f^{-1} \circ q_f (p)$ is called a 
{\it singular fiber of type $\mathrm{I}^1$}. 
\begin{figure}[htbp]
\begin{center}
\includegraphics[width=13.5cm,clip]{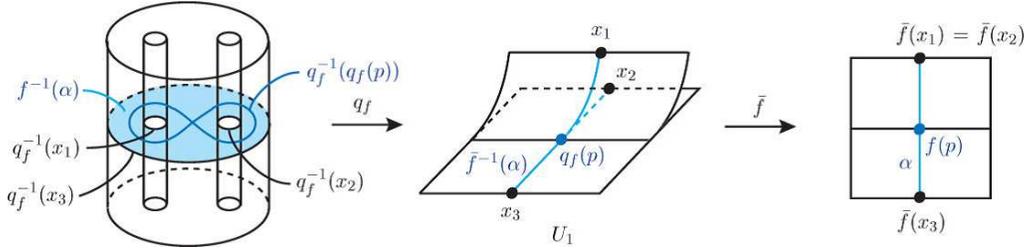}
\caption{The Stein factorization of the restriction of $f$ to $f^{-1}(U_1)$. 
In this figure, $\alpha$ is a transverse arc for $f$ at $f(p)$.  The singular fiber depicted in the figure is of Type $\mathrm{I}^1$.}
\label{fig:f-fibration_0}
\end{center}
\end{figure}
Also, the preimage $f^{-1} (U_2)$ ($f^{-1} (U_3)$, respectively)  
is diffeomorphic to $S \times [0,1]$, where 
$S$ is a disk with three holes, and $f$ maps $f^{-1} (U_2)$ ($f^{-1} (U_3)$, respectively) as drawn in 
Figure \ref{fig:the stein factorization for II3} (Figure \ref{fig:f-fibration_2}, respectively). 
\begin{figure}[htbp]
\begin{center}
\includegraphics[width=14.5cm,clip]{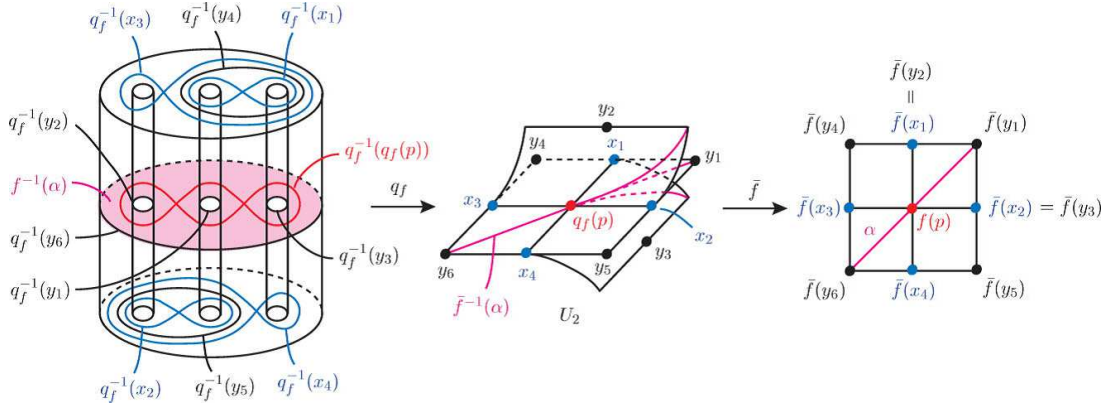}
\caption{The Stein factorization of the restriction of $f$ to $f^{-1}(U_2)$. In this figure, $\alpha$ is a transverse arc for $f$ at $f(p)$.  The singular fiber depicted in the figure is of Type $\mathrm{II}^2$.}
\label{fig:the stein factorization for II3}
\end{center}
\end{figure}
\begin{figure}[htbp]
\begin{center}
\includegraphics[width=14.5cm,clip]{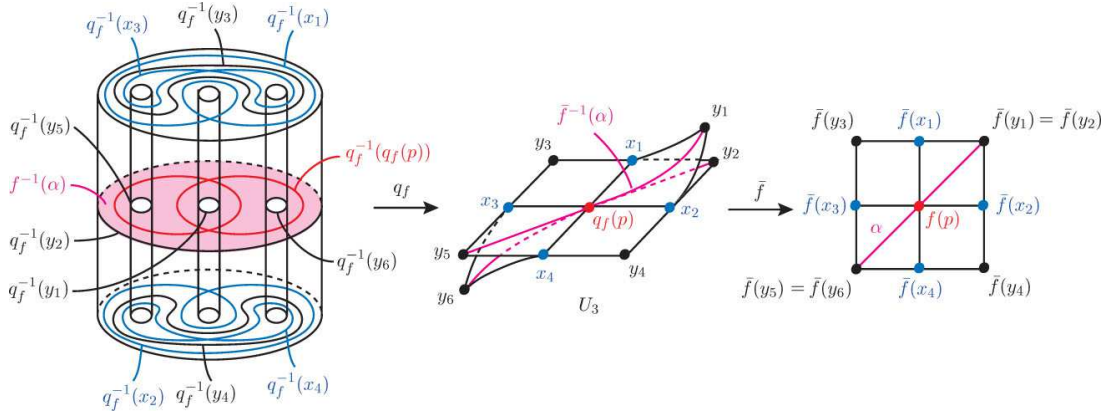}
\caption{The Stein factorization of the restriction of $f$ to $f^{-1}(U_3)$. 
In this figure, $\alpha$ is a transverse arc for $f$ at $f(p)$. The singular fiber depicted in the figure is of Type $\mathrm{II}^3$.}
\label{fig:f-fibration_2}
\end{center}
\end{figure}
In this case $q_f^{-1} \circ q_f (p)$ is called a 
{\it singular fiber of type $\mathrm{II}^2$} ($\mathrm{II}^3$, respectively). 
We call the image in $\Sigma$ of a singular fiber of type $\mathrm{II}^2$ or $\mathrm{II}^3$ 
a {\it non-simple crossing} of $f$. 
We denote by $\mathrm{II}^2 (f)$ and $\mathrm{II}^3 (f)$ the sets of 
singular fibers of types $\mathrm{II}^2$ and $\mathrm{II}^3$, respectively, of $f$. 
%We put $\CS(f) = \CS_1(f) \cup \CS_2(f)$. 

%\begin{figure}[htbp]
%\begin{center}
%\includegraphics[width=8cm,clip]{singular_fiber.eps}
%\caption{}
%\label{fig:singular_fiber}
%\end{center}
%\end{figure}

%Finally, if $p \in C(f)$, the Stein factorization looks locally as shown in Figure \ref{fig:cusp}.
%\begin{figure}[htbp]
%\begin{center}
%\includegraphics[width=3.5cm,clip]{cusp.eps}
%\caption{}
%\label{fig:cusp}
%\end{center}
%\end{figure}

\begin{definition}
\label{def:stable map of a link}
Let $M$ be a compact, orientable $3$-manifold with $($possibly empty$)$ 
boundary consisting of tori 
and $L$ a $($possibly empty$)$ link in $M$.
Let $f: M \to \Sigma$ be an S-map of $M$ into an orientable 2-manifold $\Sigma$. 
We say that $f$ is an {\it S-map of} $(M,L)$ (or simply {\it of} $L$) if $S_0(f) \supset L$. 
An S-map $f$ of $(M, L)$ is said to be {\it proper} if $S_0(f) = L$. 
When $M$ is a closed 3-manifold, we call $f$ a stable map of $(M, L)$. 
\end{definition}

In the following, we give several examples that will be used later. 
All of them appeared essentially in Saeki \cite{Sae96}. 
\begin{example}
\label{ex:stable maps and S-maps}
We identify $S^2$ with $\hat{\Complex} = \Complex \cup \{ \infty \}$. 
%Set $a = -1/2$, $b = 1/2$, $c = \infty$, 
$D_\alpha = \{ z \in \Complex \mid  |z + 1/2| \leqslant 1/4 \}$, 
$D_\beta = \{ z \in \Complex \mid  |z - 1/2| \leqslant 1/4 \}$, 
$D_\gamma = \{ z \in \Complex \mid  |z| \geqslant 1 \}$, 
$\alpha = \partial D_\alpha$, 
$\beta = \partial D_\beta$, 
$\gamma = \partial D_\gamma$. 

Let $h : S^2 \to \Real$ be a height function  
such that 
\begin{itemize}
\item
$h(z) = h(-z)$ for all $z \in \Complex$; 
\item
$h(\pm 1/2) = -1$ and $h( \infty ) = 1$; 
\item
$h(\alpha) = h(\beta) = -1/2$ and $h(\gamma) = 1/2$; and 
\item
$0 \in P$ is the unique critical point of index 1. 
\end{itemize}
For each integer $n$, let $\rho_n : S^2 \times S^1 \to [-1, 1] \times S^1$ be the map defined by 
\[ \rho_n (z , \theta) = (h(z \exp (\sqrt{-1} (-n \theta / 2))), \theta)  , \] 
where we identify $S^1$ with $ \Integer / 2 \pi \Integer $. 
%For each integer $n$, let $\tilde{\rho}_n : S^2 \times [0, 2 \pi] \to [-1, 1] \times [0, 2 \pi]$ be the map defined by 
%\[  \tilde{\rho}_n (z , \theta) = (h(z \exp (\sqrt{-1} (-n \theta / 2))), \theta)  . \] 
%The quotient space 
%\[ (S^2 \times [0, 2 \pi]) / (z \exp (-n \pi) , 2 \pi) \sim (z , 0)\] 
%s diffeomorphic to $S^2 \times S^1$. 
%Then $\tilde{\rho}_n$ induces a map 
%$\rho_n : S^2 \times S^1 \to [-1,1] \times S^1$ by 
%$\rho_n ([ z,  \theta ] ) = (\id \times \pi) \circ \tilde{\rho}_n (z, \theta)$, 
%where 
%$\pi : \Integer \to \Integer / 2 \pi \Integer$ is the projection and 
%we identify $S^1$ with $ \Integer / 2 \pi \Integer $. 
Let $\psi : [-1,1] \times S^1 \to \Real^2$ be an immersion. 
\begin{enumerate}
\item
\label{ex:S2 times S1}
%Fix a local coordinate ($\cong \Real^2$) of a point on $S^2$. 
The map $f_n : S^2 \times S^1 \to \Real^2$ defined by 
$f_n(p) = \psi \circ \rho_n(p)$ is  a stable map without cusp points. 
For this stable map, the set $S_1(f_n)$ of indefinite fold points consists of a single circle, while 
the set $S_0(f_n)$ of definite fold points consists of three (two, respectively) circles when $n$ is an even (odd, respectively) integer. 
\item
\label{ex:R times S1}
Let $f_n : S^2 \times S^1 \to \Real^2$ be the stable map defined in (\ref{ex:S2 times S1}). 
For each integer $n$, we set 
\[
V_n = ( S^2 \times S^1)  \setminus ( \Int \, ( D_\gamma \times S^1 )  \cup 
\Int \, \{ ( z \exp (\sqrt{-1} (n \theta / 2)) , \theta ) \mid z \in D_\alpha \cup D_\beta \}) . 
\]
We note that $V_n$ can be identified with $(S^2 \times S^1) \setminus \Int \,  \Nbd (S_0(f_n))$. 
In particular, when $n$ is an even integer, $V_n$ is diffeomorphic to $R \times S^1$, where $R$ is a pair of pants. 
The restriction of $\rho_n$ to $V_n$ is a map from $V_n$ to $[-1/2, 1/2] \times S^1$. 
Then $f_n|_{V_n} : V_n \to \Real^2$ is an S-map without cusp or definite fold points while 
the set $S_1 (f_n|_{V_n})$ of indefinite fold points consists of a single circle. 
This map will be used in the proof of Theorem 
\ref{thm:branched shadow complexity and crossing singulatities}. 
\item
\label{ex:A times S1}
Set $A = \hat{\Complex} \setminus \Int \, ( D_{\alpha} \cup D_{\beta} )$ and 
$\rho_A = \rho_0 | _{A \times S^1} $. 
The map $ \psi \circ \rho_A : A \times S^1 \to \Real^2$ is an S-map without cusp points. 
Each of $S_0(\psi \circ \rho_A)$ and $S_1(\psi \circ \rho_A)$ consists of a single circle. 
This map will also be used in the proof of Theorem 
\ref{thm:branched shadow complexity and crossing singulatities}. 
%The restriction of $\rho_0$ to $A \times S^1$ is a map from $A \times S^1$ to $[-1/2, 1] \times S^1$. 
%Let $\psi_2 : [-1/2 , 1] \times S^1 \hookrightarrow \Real^2$ is an embedding. 
%Then the map $h_0 : \psi_2 \circ \rho_0 | _{A \times S^1} : A \times S^1 \to \Real^2$ is an S-map without cusp points. 
\item
\label{ex:S3}
Set $D = \hat{\Complex} \setminus \Int \, D_{\gamma}$. 
The restriction of $\rho_n$ to $D \times S^1$ is a map from $D \times S^1$ to $[-1, 1/2] \times S^1$. 
We identify $S^3$ with 
$(D \times S^1) \cup_\varphi (D^2 \times S^1)$, where 
$\varphi :\partial D^2 \times S^1 \to \partial D \times S^1 $ is defined by 
$\varphi (\theta , \tau) = (\tau , \theta)$.  
Define the map $g_n : S^3 \to  \Real^2$ by 
\begin{eqnarray*}
g_n (p)  = \left\{ 
\begin{array}{ll}
\psi_1 \circ \rho_n (p)  & \mbox{ for } p \in D \times S^1 \\
x  & \mbox{ for } p = (x, \theta)  \in D^2 \times S^1 \\
\end{array}
\right. ,
\end{eqnarray*}
where $\psi_1 : [-1 , 1/2] \times S^1 \hookrightarrow \Real^2 = \Complex$ is defined by 
$\psi_1 (r, \theta) = (3/2 - r) \exp (\sqrt{-1} \theta)$ and 
$D^2$ is identified with the unit disk on $\Real^2$ centered at (0,0). 
This is a stable map without cusp points. 
The set $S_0(g_n)$ of definite fold points is the $(2, n)$-torus link, in other words, 
$g_n$ is a proper stable map of the $(2, n)$-torus link. 
The Stein factorization $W_{g_n}$ of this map is shown in Figure \ref{fig:torus_link}. 
\begin{figure}[htbp]
\begin{center}
\includegraphics[width=8cm,clip]{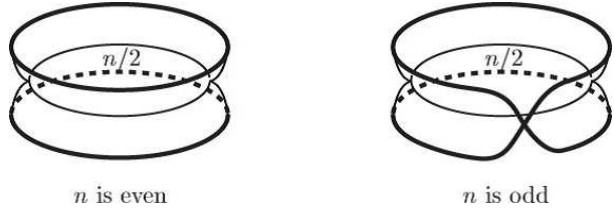}
\caption{The Stein Factorization $W_{g_n}$.}
\label{fig:torus_link}
\end{center}
\end{figure}
We note that the almost-special polyhedron $W_{g_n}$ equipped with 
the gleam $n/2$ on its disk region is a (branched) shadow of 
the $(2, n)$-torus link. 
See Costantino-Thurston \cite[Example~3.16]{CT08}. 
\end{enumerate}
 \end{example}

\begin{definition}
\label{def:stable map complexity}
Let $M$ be a compact, orientable $3$-manifold with $($possibly empty$)$ 
boundary consisting of tori 
and $L$ a $($possibly empty$)$ link in $M$.
Let $f: (M,L) \to \Real^2$ be an S-map. 
The {\it complexity} of $f$, denoted by $c(f)$, is defined by 
$c(f) = | \mathrm{II}^2 (f) |
+ 2 \, | \mathrm{II}^3 (f) | $. 
The {\it stable map complexity of $(M,L)$} (or simply {\it of $L$}), denoted by 
$\smc (M, L)$, is defined by $\smc(M, L) = \min_f \{  c(f)  \}$, 
%where $f$ runs over all stable maps of $(M, L)$ to $S^2$. 
where $f$ runs over all S-maps of $(M, L)$ into $\Real^2$ without cusp points.
\end{definition}

\begin{example}
\label{ex:stable map complexity of (2,n)-torus liinks}
In Example \ref{ex:stable maps and S-maps} (\ref{ex:S3}), we have seen that the 
stable map complexity of a $(2,n)$-torus link is zero. 
\end{example}

\section{Branched shadow complexity and stable map complexity}
\label{sec:Branched shadow complexity and stable map complexity}

\begin{lemma}
\label{lem:stable maps to R2 and S2}
Let $M$ be a compact orientable $3$-manifold with $($possibly empty$)$ boundary consisting of tori. 
Let $f: M \to S^2$ be an S-map. 
Then there exists an S-map $g: M \to \Real^2$ satisfying 
$ c(g) = | \mathrm{II}^2 (f)  | + 2 \, | \mathrm{II}^3 (f) |$ and $S_0(f) \subset S_0(g)$. 
\end{lemma}
\begin{proof}
The assertion follows from the construction of a stable map in Saeki \cite[Lemma~3.6]{Sae96}. 
For completeness, we recall the argument.

Let $x \in S^2$ be a regular value of $f$. 
Choose two small disks $D_0$ and $D_1$ in $S^2$ so that 
$x \in D_0 \subset \Int \, D_1$ and $f(S(f))\cap D_1=\emptyset$. 
%We identify $S^2$ with $\hat{\Complex} \cup \{ \infty \}$, and we choose 
%the two disks $D_\alpha$ and $D_\beta$ as in Example \ref{ex:stable maps and S-maps}.

Let $A$ and $\rho_A$ be as in Example \ref{ex:stable maps and S-maps} (\ref{ex:A times S1}). 
Let $N_1, N_2,\cdots,N_n$ be the connected components
of $ f^{-1} ( D_1 \setminus \Int \, D_0 ) $. 
For each $1 \leqslant i \leqslant n$, there exist 
diffeomorphisms $t : N_i \to A \times S^1$ and 
$u: [-1/2 , 1] \times S^1 \to D_1 \setminus \Int \, D_0$ 
such that 
$u \circ \rho_A \circ t = f|_{N_i}$. 

Now we define the required map $g:M\to S^2\setminus\{x\}\cong\Real^2$ by
\begin{eqnarray*}
g (p) = \left\{ 
\begin{array}{ll}
f (p) & \mbox{ for } p \in f^{-1}(S^2 \setminus \Int \, D_1) \\
u \circ \rho_A \circ t (p) & \mbox{ for } p \in N_i, 1 \leqslant i \leqslant n \\
u_0 \circ f (p) & \mbox{ for } p \in f^{-1}(D_0), \\
\end{array}
\right. .
\end{eqnarray*}
where 
$u_0:D_0\to S^2\setminus\{x\}$ is a smooth embedding of 
$D_0$ such that 
the restrictions of $u_0\circ f$ and $u\circ\rho_A\circ t$
to $f^{-1}(D_0)\cap N_i$ coincide. 
Since $\rho_A$ has
no singular fibers of types $\mathrm{II}^2$ or 
$\mathrm{II}^3$, 
$g$ satisfies the requirement of the lemma.
\end{proof}

\begin{theorem}
\label{thm:branched shadow complexity and crossing singulatities}
Let $M$ be a compact, orientable $3$-manifold with $($possibly empty$)$ 
boundary consisting of tori and $L$ a $($possibly empty$)$ link in $M$.
Then we have $\bsc(M, L)  = \smc(M, L)$.
\end{theorem}
\begin{proof}
We fix an orientation of $\Real^2$. 

We first prove the inequality  
$\bsc(M, L)  \leqslant \smc(M, L)$. 
The argument essentially bases on the construction in the proof of 
Costantino-Thurston \cite[Theorem~4.2]{CT08}. 
%Let $f: (M, L) \to S^2$ be a stable map without cusp points. 
Let $f: E(L) \to \Real^2$ be an S-map without cusp points. 
Let $M \overset{q_f}{\longrightarrow} W_f  \overset{\bar{f}}{\longrightarrow} \Real^2$ be the 
Stein factorization of $f$. 
Then $Q = W_f \setminus \Int \, \Nbd ( q_f( \mathrm{II}^3 (f) ) ; W_f)$ is an almost-special polyhedron. 
We note that the number of true vertices of $Q$ equals $| \mathrm{II}^2 (f) |$. 
Put $N = {q_f}^{-1} (Q)$. 
We note that 
%$\bar{f}|_{Q} \circ q_{f|_{N}} : N \to \Real^2$ is an S-map and 
$N \overset{q_{f|_{N}}}{\longrightarrow} Q  \overset{\bar{f}|_{Q}}{\longrightarrow} \Real^2$
is the Stein factorization of the map $\bar{f}|_Q \circ q_{f|_N} : N \to \Real^2$. 
By \cite{CT08}, $Q$ is a shadow of $(N , L)$. 
We may equip an orientation on each region of $Q$ in such a way that 
$\bar{f}|_{Q}$ is locally an orientation-preserving diffeomorphism 
at each interior point of the region. 
Apparently, these orientations of the regions of $Q$ give a branching $b_0$ of $Q$. 
We denote by $H_1, H_2, \ldots, H_n$ the components of $\partial Q \cap \partial \Nbd ( q_f ( \mathrm{II}^3 (f)  ) ; W_f)$, 
each of which is a trivalent graph with 4 vertices. 
Let $N_1, N_2, \ldots , N_{|\mathrm{II}^3 (f) |}$  be the components of 
${q_f}^{-1} \Nbd ( q_f( \mathrm{II}^3 (f) ) ; W_f)$, each of which is a genus 3 handlebody containing a singular fiber of type 
$\mathrm{II}^3$. 
By \cite{CT08}, the shadowed polyhedron $Q_i$ shown 
on the left-hand side in Figure \ref{fig:branching_of_shadow} 
is a shadow of each $N_i$ and each component 
$H_i$ of $\partial Q$ 
can be capped off by $Q_i$ so that the resulting polyhedron $P$ 
is a shadow of $(M, L)$. 
\begin{figure}[htbp]
\begin{center}
\includegraphics[width=8cm,clip]{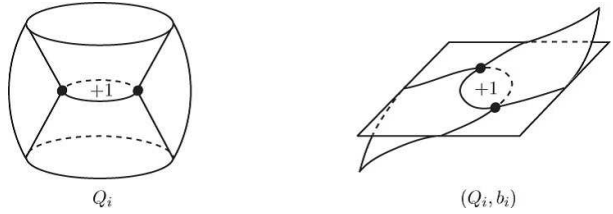}
\caption{A branching of $Q_i$.}
\label{fig:branching_of_shadow}
\end{center}
\end{figure}
We may equip a branching $b_i$ of $Q_i$ in such a way that 
the orientations on each edge of $H_i$ induced by 
$b_0$ and $b_i$ do not coincide, as shown 
on the right-hand side in Figure \ref{fig:branching_of_shadow}. 
Then the branchings $b_0$ and $b_{i} $ give a branching of $P$. 
It is straightforward from the above construction that the number of the true vertices of $P$ is 
$|\mathrm{II}^2 (f)|  + 2 \, |\mathrm{II}^3 (f) |$ while $P$ has no boundary-vertices, so we are done.

In the following we will show the other inequality:  
$\bsc (M, L)  \geqslant \smc(M, L) $. 
Let $P \subset W$ be a minimal branched shadow of $(M, L)$. 
By Lemma \ref{lem:stable maps to R2 and S2}, it suffices to show that 
there exists an S-map $f: E(L) \to S^2$ satisfying 
$| \mathrm{II}^2 (f) | + 2 \, | \mathrm{II}^3  (f)  | = \bsc (M, L)$. 
Suppose that $P$ contains non-empty boundary-vertices. 
Let $G_1, G_2, \ldots, G_n$ be the components of $\partial_f P$ containing at least one boundary vertex. 
For each $1 \leqslant i \leqslant n$, let $F_i$ be a compact orientable surface such that 
there exists an embedding $\iota_i: G_i \to \Int \, F_i$ and 
$F_i$ collapses onto  $\iota_i(G_i)$. 
Fix an orientation of $F_i$ in an arbitrary way. 
We attach $P$ to each $F_i$ by the map $\iota_i$ and we denote the resulting polyhedron by $P^*$. 
The branching of $P$ and the orientations of $F_1$, $F_2 , \ldots, F_n$  give a branching of $P^*$. 
See Figure \ref{fig:wall}. 
\begin{figure}[htbp]
\begin{center}
\includegraphics[width=14cm,clip]{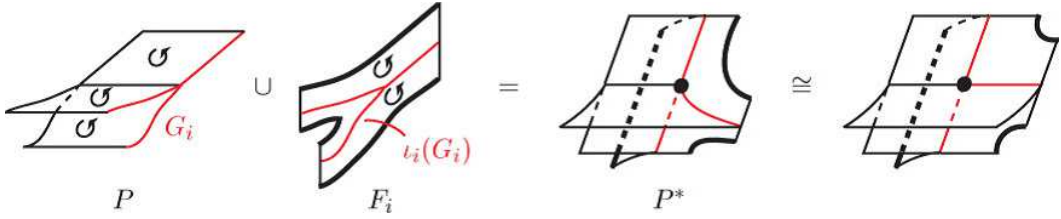}
\caption{Elimination of boundary-vertices.}
\label{fig:wall}
\end{center}
\end{figure}
We note that by this operation, each boundary-vertex of $P$ gives rise to a true vertex of $P^*$ and 
no other true or boundary-vertices are produced. 
Thus $P^*$ has no boundary-vertices and we have 
$|V(P^*)|  = c(P) = \bsc (M, L)$. 
Therefore we may assume without loss of generality that 
$P$ has no boundary-vertices. 

We choose a collapsing $W \searrow P$ so that the induced projection  $\pi : M \to P$ satisfies the following: 
\begin{itemize}
\item
for each $x \in \partial P$, $\pi^{-1}(p)$ is a single point; 
\item
for each $x \in V(P)$, $\pi^{-1} (x) \cong T_4$, where $T_4$ is the suspension of four points; 
\item
for each $x \in S(P) \setminus V(P)$, $\pi^{-1} (x) \cong T_3$, where $T_3$ is the suspension of three points; and 
\item
the restriction of $\pi$ to the preimage $\pi^{-1} (P \setminus ( \partial P \cup  S(P)))$ is a trivial, smooth $S^1$-bundle. 
%for each $x \in P \setminus S(P)$, $\pi^{-1} (x) \cong S^1$;
%\item 
%the restriction of $\pi$ to $\pi^{-1} (P \setminus (S(P)))$ is a smooth map 
%to $P \setminus (S(P))$. 
\end{itemize}
Note that for the construction of the projection $\pi$, we do not need a branching of $P$. 
For a point $x$ in $S(P) \setminus V(P)$, the preimage of $\Nbd(x)$ under $\pi$ is diffeomorphic to 
$R \times [0,1]$, where $R$ is a pair of pants, and $\pi$ is described as 
on the left-hand side in Figure \ref{fig:pi-fibration_0}. 
\begin{figure}[htbp]	
\begin{center}
\includegraphics[width=13.5cm,clip]{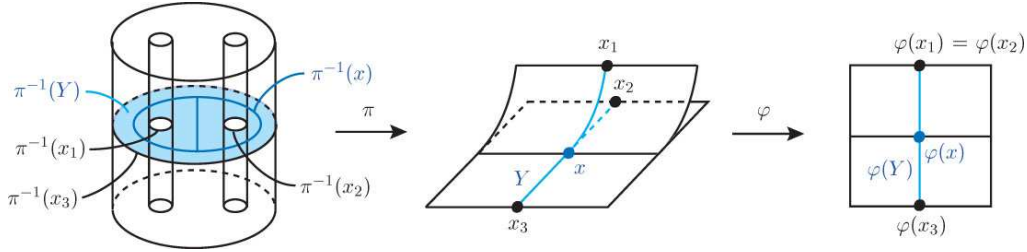}
\caption{The local model of $\pi$ and $\varphi$ for $\Nbd (x; P)$, where $x \in S(P) \setminus V(P)$.}
\label{fig:pi-fibration_0}
\end{center}
\end{figure}
Also, for a point $x$ in $V(P)$, the preimage of $\Nbd(x)$ under $\pi$ is diffeomorphic to 
$S \times [0,1]$, where $S$ is a disk with 3 holes, and $\pi$ is described as on the left-hand side in Figure \ref{fig:pi-fibration_1}. 
\begin{figure}[htbp]
\begin{center}
\includegraphics[width=14.5cm,clip]{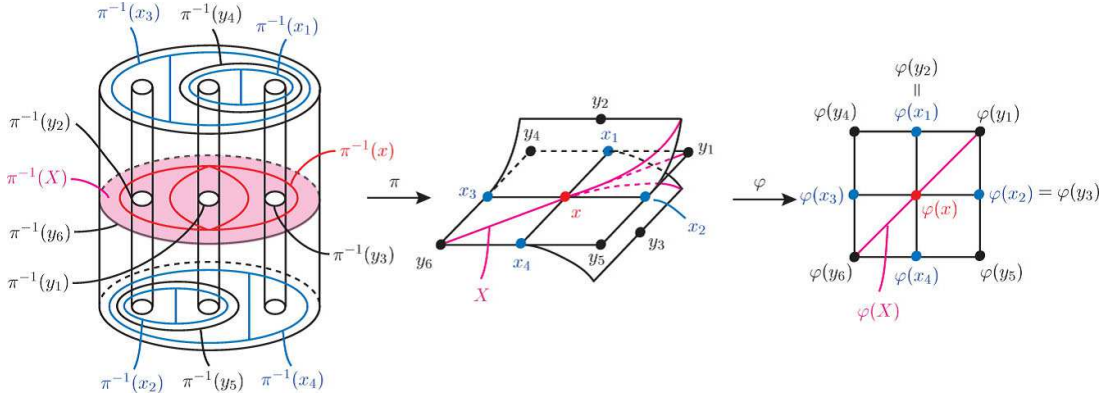}
\caption{The local model of $\pi$ and $\varphi$ for $\Nbd (x; P)$, where $x \in V(P)$.}
\label{fig:pi-fibration_1}
\end{center}
\end{figure}

Let $v_1, v_2, \ldots , v_{n_0}$ be the true vertices of $P$. 
We put 
\[P' = \overline{ P \setminus \left( \Nbd( \partial P ; P) \cup \left( \bigcup_{i=1}^{n_0} \Nbd (v_i; P) \right) \right) }.\]  
The set of components of $S(P')$ consists of 
intervals $e_1, e_2, \ldots, e_{m_1}$ and circles 
$l_1, l_2, \ldots, l_{n_1}$. 
We note that $\Nbd (e_i; P')$ is homeomorphic to $Y \times [0, 1]$,  and 
$\Nbd (l_j; P')$ is homeomorphic to $Y \times S^1$ or $Y \times [0, 2 \pi] / (z , 2 \pi ) \sim (\rho (z) , 0) $, 
where $Y$ is a Y-shaped graph that we identify with 
$Y = \{ r \exp(\sqrt{-1} \theta) \in \Complex \mid  0 \leqslant r \leqslant 1, \theta = 0 \mbox{ or } \pm 2 \pi / 3 \}$ 
and $\rho : Y \to Y$ is defined by the complex conjugation. 
With this homeomorphisms, we equip each of $\Nbd (e_i; P')$ and $\Nbd (l_i; P')$ with a $Y$-bundle structure. 
In the following, we put the three boundary points of $Y$ as $a = \exp((\sqrt{-1}) 2/3 \pi)$, $b = \exp((\sqrt{-1}) (- 2/3 \pi))$ and 
$c = 1$. 
The set of components of 
\[
%P'' = 
\overline{ P' \setminus \left( \left( \bigcup_{i=1}^{m_1} \Nbd (e_i; P')\right) \cup \left( \bigcup_{i=1}^{n_1} \Nbd (l_i; P') 	\right) \right) } 
=  P \setminus \Int \, \Nbd (\partial P \cup S(P); P) \] 
consists of oriented surfaces $R_1, R_2, \ldots, R_{n_2}$. 

We set $N = \pi^{-1} (\Nbd (\partial P \cup S(P);  P))$. 
Since $\Nbd (\partial P \cup S(P);  P)$ is a branched polyhedron, there exists a map 
$\varphi : \Nbd (\partial P \cup S(P);  P) \to S^2$ satisfying the following 
(see the right-hand side in Figures \ref{fig:pi-fibration_0} and \ref{fig:pi-fibration_1}): 
\begin{itemize}
\item
$\varphi$ is generic on $S(P)$, that is, 
$\varphi$ is an injection on $S(P)$  
except at finitely many double points, which are the projections of only two points of $S(P) \setminus V(P)$ and 
at the crossing points the two threads project to two different directions on $S^2$; 
\item
For $1 \leqslant i < j \leqslant n_0$, $\varphi (\partial \Nbd (v_i)) \cong S^1$ and 
$\varphi  (\Nbd (v_i)) \cap \varphi (\Nbd (v_j)) = \emptyset$; 
\item
$\varphi$ is locally an orientation-preserving diffeomorphism at each point of $\Nbd (\partial P \cup S(P);  P) \setminus S(P)$; 
\item
$\varphi (\Nbd (S(P);  P))$ is a regular neighborhood of $\varphi (S(P))$ in $S^2$; 
\item
for each $x \in e_i$ ($x \in l_i$, respectively), 
the fiber $Y_x$ of the 
$Y$-bundle structure of $\Nbd(e_i; P')$ ($\Nbd(l_i; P')$, respectively)
is mapped into a segment by $re^{i\theta}\mapsto r\cos\theta$.
%\item
%for each $x \in e_i$, the three end points of the fiber $Y_x$ of the $Y$-bundle structure of $\Nbd(e_i; P')$ are 
%mapped into two points of $Y$; and 
%\item
%for each $x \in l_i$, the three end points of the fiber $Y_x$ of the $Y$-bundle structure of $\Nbd(l_i; P')$ are 
%mapped into two points of $Y$. 
\end{itemize}

\begin{claim}
\label{claim:S-map}
There exists an S-map $g : N \to S^2$ with the Stein factorization 
$N \overset{q_g}{\longrightarrow} W_g  \overset{\bar{g}}{\longrightarrow} S^2$ 
such that 
\begin{itemize}
\item
$g(N) = \varphi \circ \pi (N)$; 
\item
$W_g \cong \Nbd ( \partial P \cup S(P); P)$; and 
\item
the following diagram commutes: 
\[\xymatrix{
\partial N \ar[d]_{\pi} \ar[r]^{q_g} \ar@{}[dr]|\circlearrowleft & W_g \ar[d]^{\bar{g}} \\
\Nbd ( \partial P \cup S(P); P) \ar[r]^{\hspace{3.5em} \varphi} & S^2 \\
}  \]
\end{itemize}
\end{claim}
\begin{proof}[Proof of Claim $\ref{claim:S-map}$]

It is clear from the local model of 
the definite fold points (recall Figure \ref{fig:definite_fold} (i)) and the definition of $\varphi$ that 
$\pi^{-1} (\Nbd ( \partial P ; P))$ consists of finitely many solid tori and 
there exists an S-map $g_{\partial P} : \pi^{-1} (\Nbd ( \partial P ; P)) \to S^2$ 
with the Stein factorization 
$\pi^{-1} (\Nbd ( \partial P ; P)) \overset{q_{g_{\partial P}}}{\longrightarrow} W_{g_{\partial P}}  \overset{\bar{g}_{\partial P}}{\longrightarrow} S^2$ 
such that 
\begin{itemize}
\item
$g_{\partial P} (\pi^{-1} (\Nbd ( \partial P ; P))) = \varphi (\Nbd ( \partial P ; P))$; 
\item
$W_{g_{\partial P}} \cong \Nbd ( \partial P ; P)$; and 
\item
$\bar{g}_{\partial P} \circ q_{g_{\partial P}} = \varphi \circ \pi$ on $\partial \pi^{-1} (\Nbd ( \partial P ; P))$.  
\end{itemize}

For each $1 \leqslant i \leqslant n_1$, we identify 
$\Nbd(l_i; P)$ with $Y \times S^1$ or 
$Y \times [0, 2 \pi] / (z , 2 \pi ) \sim ( \rho (z) , 0 )$ as mentioned earlier. 
Here we identify them so that 
the points $(a , 0) \in Y \times S^1$ and $(b, 0) \in Y \times S^1$ lie on the 
regions of $P$ that induce the same orientation on $l_i$.  
The map $\varphi |_{\Nbd(l_i; P)} : \Nbd(l_i; P) \to S^2$ factors through 
$\Nbd(l_i; P) \overset{\varphi_1}{\longrightarrow} [-1/2, 1/2] \times S^1 \overset{\varphi_2}{\longrightarrow} S^2$, 
where $\varphi_1$ is defined by 
$\varphi_1 (r \exp(\sqrt{-1} (\pm 2 \pi / 3 ))) = -r/2$ and $\varphi_1 (r) = r/2$ for 
$0 \leqslant r \leqslant 1$. 
For $n \in \Integer$, let $V_n$ and $\rho_n$ be as defined in Example \ref{ex:stable maps and S-maps} (\ref{ex:R times S1}). 
Then by construction, we may identify $\pi^{-1} (Y \times S^1)$ with $V_0$ or $V_1$, 
depending on whether 
$\Nbd(l_i; P)$ is $Y \times S^1$ or 
$Y \times [0, 2 \pi] / (\rho(z), 2 \pi) \sim (z, 0)$, so that 
$\varphi_2 \circ \rho_n$ coincides with $\varphi \circ \pi |_{V_n}$ on $\partial V_n$.  
We denote the map $\varphi_2 \circ \rho_n : \pi^{-1}(\Nbd(l_i; P)) \to S^2$ by $g_{l_i}$. 
%We identify the preimage $\pi^{-1} (Y \times \{ 0 \}) $ with the pair of pants 
%$R$ modeled in Example \ref{ex:the pair of pants bundle over the circle}. 
%For this identification, we may require that 
%$\pi$ maps the components $\alpha$, $\beta$, $\gamma$ of $\partial R$ 
%to $(a, 0)$, $(b, 0)$, $(1, 0)$, respectively. 
%Then $\pi^{-1}(\Nbd(l_i; P))$ can be identifies with 
%$V_0$ or $V_1$ depending on whether 
%$\Nbd(l_i; P)$ with $Y \times S^1$ or 
%$Y \times [-1, 1] / (\rho(z), 1) \sim (z, -1)$. 
%Moreover, we may assume that $\pi$ preserves the 
%bundle structures over $S^1$. 
%Let $\psi : [-1,1] \times S^1 \to S^2$ be an immersion such that 
%\begin{itemize}
%\item
%$\psi  ( [-1,1] \times S^1 ) = \varphi ( \Nbd (l_i; P') )$; 
%\item
%$\psi \circ \rho_k$ coincides with $\varphi \circ \pi $ on $\partial V_k$, where $k=0$ or $1$ and 
%$\rho_k$ is the map defined in Example \ref{ex:the pair of pants bundle over the circle}. 
%\end{itemize} 
%We denote the above map $\psi \circ \rho_k : \pi^{-1}(\Nbd(l_i; P)) \to S^2$ by $g_{l_i}$. 

For each $1 \leqslant j \leqslant m_1$, we identify 
$\Nbd(e_j; P')$ with $Y \times [0, 1]$. 
As in the above argument, we may assume that 
the points $(a , 0) \in Y \times S^1$ and $(b, 0) \in Y \times S^1$ lie on the 
regions of $P$ that induce the same orientation on $e_j$.  
The map $\varphi |_{\Nbd(e_j; P')} : \Nbd(e_j; P') \to S^2$ factors through 
$\Nbd(e_j; P') \overset{\varphi_1}{\longrightarrow} [-1/2, 1/2] \times [0,1] \overset{\varphi_2}{\longrightarrow} S^2$, 
where $\varphi_1$ is defined by 
$\varphi_1 (r \exp(\sqrt{-1} (\pm 2 \pi / 3 ))) = -r/2$ and $\varphi_1 (r) = r/2$ for 
$0 \leqslant r \leqslant 1$. 
Let $R = \hat{\Complex} \setminus \Int  \, (D_\alpha \cup D_\beta \cup D_\gamma)$  
and $\rho_0 : R \times S^1 \to [-1/2, 1/2]$ be as in Example \ref{ex:stable maps and S-maps} (\ref{ex:R times S1}). 
By construction, we may identify $\pi^{-1} (Y \times [0,1])$ with $R \times [0,1]$ so that 
$\varphi_2 \circ \rho_0 |_{R \times [0,1]}$ coincides with $\varphi \circ \pi |_{R \times [0,1]}$ on $\partial R \times [0,1]$.  
We denote the map $\varphi_2 \circ \rho_0 |_{R \times [0,1]} : \pi^{-1}(\Nbd(e_j; P')) \to S^2$ by $g_{e_j}$. 

%Also, we identify the preimage $\pi^{-1} (Y \times \{ 0 \}) $ with the pair of pants $R$ so that 
%$\pi$ maps the components $\alpha$, $\beta$, $\gamma$ of $\partial R$ 
%to $(a , 0)$, $(b , 0)$, $(1, 0)$, respectively. 
%Then $\pi^{-1}(\Nbd(e_j; P'))$ can be identifies with $R \times [0, 2 \pi]$ and 
%we may assume that $\pi$ preserves the 
%bundle structures over $[0, 2 \pi]$. 
%Let $\tilde{\psi} : [-1,1] \times [0, 2 \pi] \to S^2$ be an immersion such that 
%\begin{itemize}
%\item
%$\psi  ( [-1,1] \times [0, 2 \pi] ) = \varphi ( \Nbd (e_j; P') )$; 
%\item
%$\tilde{\psi} \circ \tilde{\rho}_0$ coincides with $\varphi \circ \pi $ on $\partial R \times [0, 2 \pi]$, where 
%$\tilde{\rho}_0$ is the map defined in Example \ref{ex:the pair of pants bundle over the circle}. 
%\end{itemize} 
%We note that the fibrations of $R \times [0, 2 \pi]$ by the preimages of the maps 
%$\tilde{\psi} \circ \tilde{\rho}_0$ and $\varphi \circ \pi$ are 
%as in described in Figures \ref{fig:f-fibration_0} and \ref{fig:pi-fibration_0}, respectively. 
%We denote the above map $\tilde{\psi} \circ \tilde{\rho}_0  : \pi^{-1}(\Nbd(e_j; P')) \to S^2$ by $g_{e_j}$. 

Finally, for each $1 \leqslant k \leqslant n_0$, we identify $\pi^{-1} (\Nbd(v_k; P))$ with $S \times [0,1]$ 
as shown in Figure \ref{fig:pi-fibration_1}, where we recall that $S$ is a disk with 3 holes. 
Let $g_{v_k} : S \times [0,1] \to S^2$ be the local model 
of an S-map in a small neighborhood of a singular fiber of type 
$\mathrm{II}^2$
shown in Figure \ref{fig:the stein factorization for II3} with the Stein factorization 
$S \times [0,1] \overset{q_{g_{v_k}}}{\longrightarrow} W_{g_{v_k}}  \overset{\bar{g}_{v_k}}{\longrightarrow} S^2$ 
such that 
\begin{itemize}
\item
$g_{v_k} |_{\partial S \times [0,1]} = \varphi \circ \pi |_{\partial S \times [0,1]}$; 
\item
$g_{v_k} (S \times \{t\}) = \varphi \circ \pi  (S \times \{t\}))$; and 
%$g_{v_k} (\partial (S \times [0,1])) = \varphi \circ \pi  (\partial (S \times [0,1]))$; and 
\item
the Stein factorization $W_{g_{v_k}}$ is isomorphic to $\Nbd(v_k; P)$ as branched polyhedron, 
where we equip an orientation of each region of $W_{g_{v_k}}$ in such a way that 
map $\bar{g}_{v_k}$ is locally an orientation-preserving diffeomorphism on the region. 
\end{itemize}
Since $P$ is equipped with a branching, a small $C^{\infty}$-perturbation allows us to assume that, 
on each component of $\pi^{-1} (\Nbd (v_k; P) \cap \Nbd (e_j; P') )$, 
that is a pair of pants, $g_{v_k}$ and $g_{e_j}$ coincide. 

Now, we define a map $g : N \to S^2$ by 
\begin{eqnarray*}
g (p)  = \left\{ 
\begin{array}{ll}
g_{\partial P (p)} & \mbox{ for } p \in \pi^{-1} ( \Nbd(\partial P; P) ) ; \\
g_{l_i} (p) & \mbox{ for } p \in \pi^{-1} (\Nbd(l_i; P')), 1 \leqslant i \leqslant n_1 ; \\
g_{e_j} (p) & \mbox{ for } p \in \pi^{-1} (\Nbd(e_j; P')), 1 \leqslant j \leqslant m_1 ; \\
g_{v_k} (p) & \mbox{ for } p \in \pi^{-1} (\Nbd(v_k; P)), 1 \leqslant k \leqslant n_0 . \\
\end{array}
\right. 
\end{eqnarray*}
This completes the proof of the claim. 
\end{proof}

By the above claim, it suffices to show that, for each $1 \leqslant i \leqslant n_2$, 
there exists an S-map $g_{R_i} : \pi^{-1} (R_i) \to S^2$ such that 
\begin{itemize}
\item
$g_{R_i} |_{\partial (\pi^{-1} (R_i))} = g |_{\partial (\pi^{-1} (R_i))}$; and 
\item
%$g_{R_i}$ admits neither indefinite fold points nor cusp points. 
$g_{R_i}$ admits no singular fibers of types $\mathrm{II}^2$ or $\mathrm{II}^3$.
\end{itemize}

%\begin{figure}[htbp]	
%\begin{center}
%\includegraphics[width=4cm,clip]{extending_a_stable_map.eps}
%\caption{}
%\label{fig:extending_a_stable_map}
%\end{center}
%\end{figure}

First we prove the assertion in the case where each surface $R_i$ is a disk 
for each $i \in \{ 1, 2, \ldots, n_2 \}$.
We fix the index $i$ and prove the existence of such an S-map of $\pi^{-1}(R_i)$. 
Hereafter we identify $\pi^{-1}(R_i)$ with $R_i\times S^1$.
By small perturbation of the map $g$, we may assume that $\bar g(\partial R_i)$ has only normal double points.
Choose a smooth arc in $\bar g(\partial R_i)$ which starts and ends at
the same point, i.e., a double point of $\bar g(\partial R_i)$, and has no other double points.
Such a smooth arc can be found by choosing a suitable double point.
Choose a band $b_1=[0,1]\times [0,1]$ in $R_i$ with $\Nbd (S(P);P)\cap b_1=\{0,1\}\times [0,1]$
and a map $\bar g_{b_1} : \Nbd(S(P);P)\cup b_1\to S^2$ such that
$\bar g_{b_1}|_{\Nbd(S(P);P)}=\bar g$ and the image $\bar g_{b_1}(b_1)$ of $b_1$ is as shown in Figure \ref{fig:band}. 
\begin{figure}[htbp]	
\begin{center}
\includegraphics[width=8cm,clip]{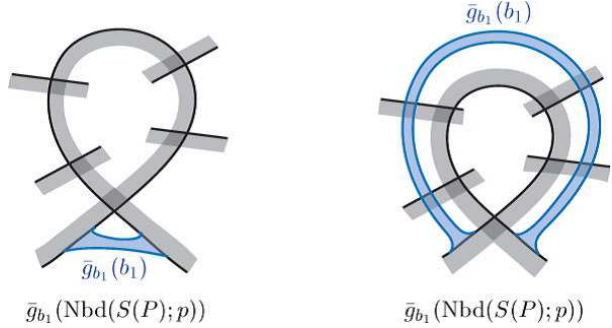}
\caption{Extension of $\bar g$ over a band $b_1$.}
\label{fig:band}
\end{center}
\end{figure}
This map induces an S-map $g_{b_1}:N\cup(b_1\times S^1)\to S^2$ satisfying
\begin{itemize}
\item $g_{b_1}|_{N}=g$; and
\item the Stein factorization of $g_{b_1}$ is given as $g_{b_1} = \bar g_{b_1} \circ q_{g_{b_1}}$,
where $q_{g_{b_1}} |_{b_1\times S^1} :b_1 \times S^1\to b_1$ is the first projection.
\end{itemize}
Note that $b_1 \times S^1 \subset R_i \times S^1$ and $b_1\times S^1$ is attached to $N$ according to the identification
of $\pi^{-1}(R_i)$ with $R_i\times S^1$.
Since $R_i$ is a disk, the closure of $R_i \setminus b_1$ consists of two disjoint disks, say $R_i'$ and $R_i''$,
such that 
\begin{itemize}
\item $\bar g_{b_1}(\partial R'_i)$ has just one double point; and
\item the number of double points of $\bar g_{b_1}(\partial R''_i)$ is less than
that of $\bar g(\partial R_i)$. 
\end{itemize}
Next, we choose the surface $R_i''$ and apply the same argument, i.e., attach a band $b_2$ and make an S-map 
$g_{b_2}:N\cup(b_1\times S^1)\cup(b_2\times S^1)\to S^2$.
We continue this process successively until 
$\bar g_{b_j}(c)$ has at most one double point for each boundary component $c$ of the closure of $R_i \setminus (b_1 \cup \cdots \cup b_j)$. 
Since the number of double point of $\bar g(\partial R_i)$ is finite 
%andthe number of double points of $\bar g_{b_j}(\partial R''_i)$ decreases as the index $j$ increases,
this process ends at a finite number of steps, say $n_3$ steps.
As a consequence, we obtain an S-map $g_B:N\cup(B\times S^1)\to S^2$, where $B=\cup_{j=1}^{n_3} b_j$,
such that
\begin{itemize}
\item $g_{B}|_{N}=g$; 
\item the Stein factorization of $g_B$ is given as $g_{B}=\bar g_{B}\circ q_{g_{B}}$,
where $\bar g_B=\bar g_{b_{n_3}}$ and $q_{g_{B}} |_{B \times S^1} : B \times S^1\to B$ is the first projection; and
\item for each boundary component $c$ of the closure of $R_i\setminus B$,
$\bar g_B|_{c}$ is an immersion of a circle into $S^2$ such that it has at most one double point.
\end{itemize}

Suppose that $c_0$ is a boundary component of the closure of $R_i\setminus B$ 
such that $\bar g_B(c_0)$ has no double point and let $D_{c_0}$ denote the disk in $R_i$ bounded by $c_0$.
In this case, the S-map $g_{c_0}:N\cup(B\times S^1)\cup (D_{c_0}\times S^1)\to S^2$
is easily obtained by choosing the image $\bar g_{c_0}(D_{c_0})$ to be the disk bounded by $\bar g_B(c_0)$,
where $N\cup(B\times S^1)\cup (D_{c_0}\times S^1)\subset M$ due to the identification of $\pi^{-1}(R_i)$ with $R_i\times S^1$.

Suppose that $c_1$ is a boundary component of the closure of $R_i\setminus B$ 
such that $\bar g_B(c_1)$ has exactly one double point.
Let $V_1$ and $f_1|_{V_1}:V_1\to \mathbb R^2$ be as in Example \ref{ex:stable maps and S-maps} (\ref{ex:R times S1}). 
%Recall that $V_1$ is diffeomorphic to
%the complement of an open tubular neighborhood of a $(2,1)$-cable knot in the solid torus.
We regard $V_1$ as the exterior of an open tubular neighborhood of a $(2,1)$-cable knot in the solid torus $D^2\times S^1$ 
and set $T=\partial (D^2\times S^1)$ and $T_1 = \partial V_1 \setminus T$. 
Let $\phi:V_1\to V_1$ be an orientation-preserving diffeomorphism such that the positions of two boundary tori exchange,
and set $f_\phi=f_1|_{V_1}\circ \phi$. 
The image $f_\phi(V_1)$ is an annulus in $\mathbb R^2$.
If the preimage of a point on the boundary of $f_\phi(V_1)$ is on $T$
then it is $\{a_1, a_2\}\times S^1$, where $a_1, a_2$ are distinct two points on $\partial D^2$.
If the preimage is on $T_1$, 
then it is a longitude of the $(2,1)$-cable knot. 
%If the preimage is on the boundary of the compact tubular neighborhood of the $(2,1)$-cable knot
%then it is a longitude.

Consider the Stein factorization of $f_\phi$, i.e., 
$V_1 \overset{q_{f_\phi}}{\longrightarrow} W_{f_\phi}  \overset{\bar{f}_\phi}{\longrightarrow} S^2$.
Attach $W_{f_\phi}$ to $\text{Nbd}(S(P);P)\cup B$ smoothly so that
$(\text{Nbd}(S(P);P)\cup B) \cap W_{f_\phi}=q_{f_\phi}(T)=c_1$,
and then attach $V_1$ to $N\cup(B\times S^1)$ so that $q_{f_\phi}|_T=g_B|_T$.
Here the solid torus $D^2\times S^1$ including $V_1$ is identified with $\pi^{-1}(D_{c_1})\subset \pi^{-1}(R_i)$,
where $D_{c_1}$ is the disk in $R_i$ bounded by $c_1$.
We need to modify the map $f_\phi$ to $f'_\phi$, with keeping $q_{f_\phi}=q_{f'_\phi}$, so that 
$\bar f'_{\phi}|_{q_{f_{\phi}}(T)}=\bar g_B|_{c_1}$. In other words,
modify it in such a way that the S-map $g_{c_1}:N\cup(B\times S^1)\cup V_1\to S^2$ defined by
\[
g_{c_1}=
\begin{cases}
g_B(p)& \text{ for } p\in N\cup(B\times S^1) \\
f'_\phi(p) & \text{ for } p\in V_1
\end{cases}
\]
is well-defined. 
%Moreover, we require that $f'_\phi$ satisfies the following.  
The existence of such a modification of $f_\phi$ can be checked as follows: 
We set $N_{c_1}=\Nbd(c_1;\Nbd(S(P);P)\cup B)$, which is an annulus.
The image $\bar g_B(N_{c_1})$ is an immersed annulus such that 
the image of one of the immersed boundary components is $\bar g_B(c_1)$. 
There are two regions in $S^2\setminus \Int \, \bar g_B(N_{c_1})$ which have just one corner, 
and one of these regions is bounded by $\bar g_B(c_1)$.
Taking the image of $f_\phi$ in this region, 
we see that $f_\phi$ can be modified so that it coincides with
$\bar g_B|_{c_1}$ on $q_{f_{\phi}}(T)$. See Figure \ref{fig:modification}. 
\begin{figure}[htbp]	
\begin{center}
\includegraphics[width=8.5cm,clip]{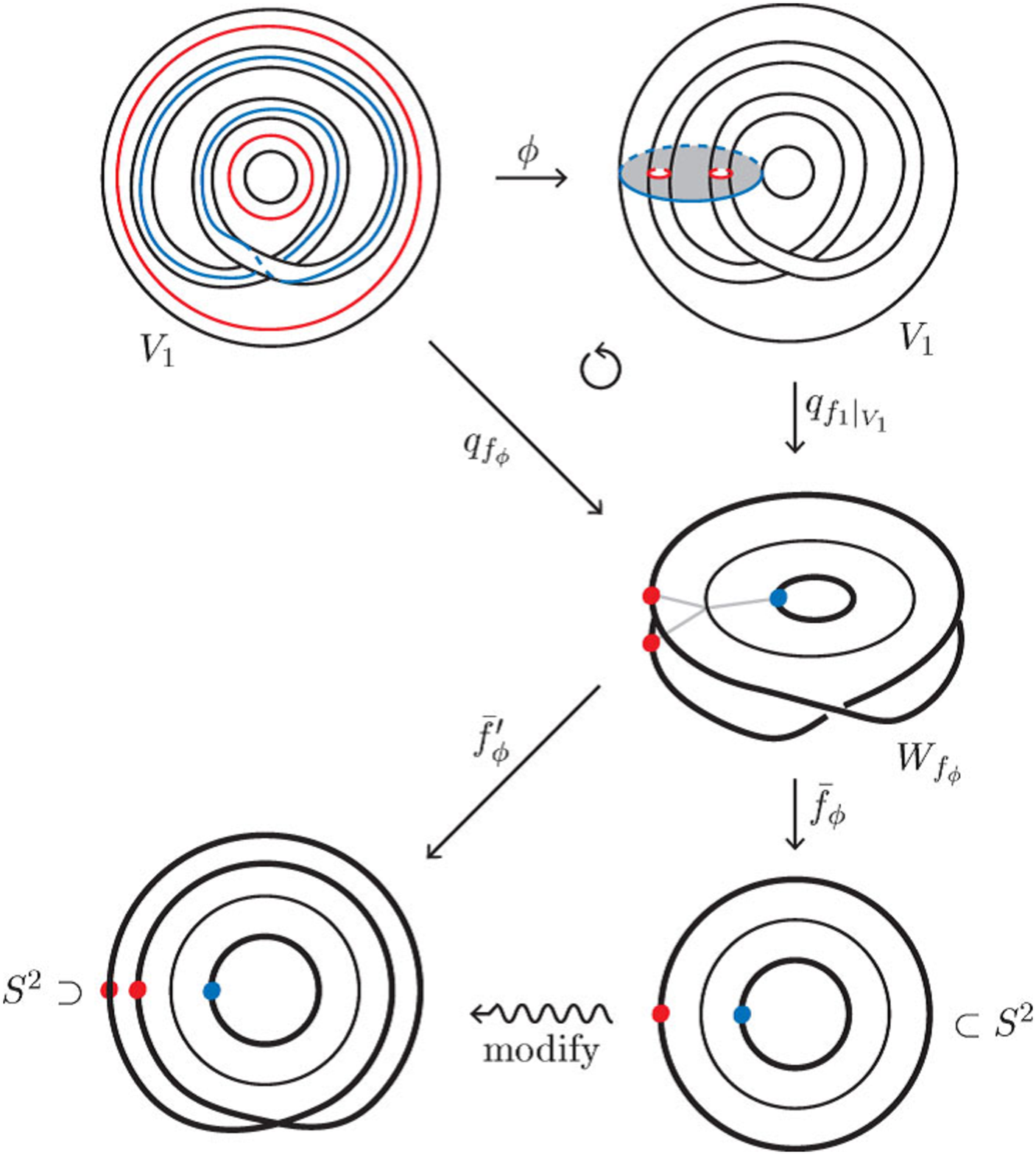}
\caption{Extension of $g_B$ over $D_{c_1}$.}
\label{fig:modification}
\end{center}
\end{figure}

We attach a solid torus to $V_1$ along the boundary torus $T_1$ as follows: 
%We fill the boundary of the compact tubular neighborhood of the $(2,1)$-cable knot by a solid torus as follows:
We first recall the construction of the map $h_2:T^2\times [0,1] \to \Real^2$ given in Saeki \cite[p.~680]{Sae96},
where $T^2$ is a torus.
Let $R$ be a pair of pants and $a$ be one of its boundary components. 
Let $\chi : \partial D^2\times S^1\to a \times S^1$ be the diffeomorphism defined by $\chi(\theta,\tau)=(\theta+\tau,\tau)$
and set $X=(R \times S^1)\cup_{\chi}(D^2\times S^1)$, which is diffeomorphic to $T^2\times [0,1]$.
Then define $h_2:T^2\times [0,1] \to \Real^2$ by
\[
h_2(p)=
\begin{cases}
(h(p_1),p_2)\in\Real\times S^1\subset\Real^2 & \text{for $p=(p_1,p_2)\in R \times S^1$} \\
q_1\in D^2\subset\Real^2 & \text{for $p=(q_1,q_2)\in D^2\times S^1$},
\end{cases}
\]
where $h:R \to \Real$ is the Morse function as in Example \ref{ex:stable maps and S-maps} 
with regarding that $R$ is the pair of pants
bounded by $\alpha, \beta$ and $\gamma$. 
The map $h_2$ has the following important property:
For $i=0,1$, let $\mu_i\subset T^2\times\{i\}$ be a section of $h\times \text{id}_{S^1} : R \times S^1\to\Real\times S^1$
and $\lambda_i$ be a fiber of $h_2$ on $T^2\times\{i\}$.
Then $h_2$ satisfies $[\lambda_0]=[\mu_1]$ and $[\mu_0]=[\lambda_1]$ in $H_1(T^2 \times  [0,1] )$.
%In other words, $h_2$ changes the position of a fiber into a position of a section and vice-versa.

Set $U=\Nbd(T_1;V_1)$ and $T_0=\partial U\setminus T_1$.
Let $D$ denote the disk in $S^2$ bounded by the simple closed curve $g_{c_1}(T_0)$.
There exist diffeomorphisms $t:U\to T^2\times [0,1]$ and 
$u:D^2\to D\subset S^2$ such that $u\circ h_2\circ t|_{T_0}=g_{c_1}|_{T_0}$.
Then we define an S-map $g_{c_1}':N\cup(B\times S^1)\cup V_1\to S^2$ by
\[
g_{c_1}'=
\begin{cases}
g_{c_1}(p)& \text{ for } p\in N\cup(B\times S^1)\cup V_1\setminus (\Int \, U \cup T_1)\\
u\circ h_2\circ t(p) & \text{ for } p\in U.
\end{cases}
\]
Since each fiber of the map $g_{c_1}'$ on $T_0$ is a longitude, the fiber of $g_{c_1}'$ on $T_1$ is the meridian of the $(2,1)$-cable knot.
The compact tubular neighborhood of an $S^1$-family of definite fold points
is a solid torus whose fiber on the boundary is the meridian.
Attaching this solid torus to $N\cup(B\times S^1)\cup V_1$ along $T_1$,
we obtain the S-map from $N\cup(B\times S^1)\cup (D^2\times S^1)$ to $S^2$,
where $D^2\times S^1$ is the solid torus obtained from $V_1$ by filling the boundary $T_1$ canonically.

We apply the same construction for all boundary components of the closure of $R_i\setminus B$
and obtain an S-map $g_{C_i}:N\cup (R_i\times S^1)\to S^2$. 
%The gleam on the region $R_i$ can be adjusted by removing and re-gluing 
%a compact tubular neighborhood of the fiber at a point on $\partial R_i$.

Applying this construction to all surfaces $R_1 , \ldots , R_{n_2}$, 
we obtain the S-map $g_C:M\to S^2$ with the required properties. 
This completes the proof in the case where each surface $R_i$ is a disk.

Suppose that some of $R_1, \ldots, R_{n_2}$ is not a disk. 
Let $\Gamma$ be the union of mutually disjoint, essential simple closed curves in 
$R_1 \cup \cdots \cup R_{n_2}$ that cuts $R_1 \cup \cdots \cup R_{n_2}$ 
into planar surfaces each of which contains at most one 
boundary component of $\Nbd ( \partial P \cup S(P); P)$. 
We remark that the polyhedron $P$ cut off by $\Gamma$ may not be connected. 
We denote by $\hat{P}$ the almost-special branched polyhedron 
obtained by cutting $P$ along $\Gamma$ and then 
capping off the new boundary components by disks. 
Let $\hat{M}$ be the $3$-manifold represented by $\hat P$ after giving suitable 
gleams to the new disk regions. 
Since each region of $\hat P$ is a disk, by applying the above argument,
we obtain an S-map $g_C:\hat M \to S^2$. 
To recover $P$, we first remove suitable number of disks from the regions of $\hat P$
and also remove their preimages in $\hat M$.
We assume that the removed disks are near the boundary of $\Nbd(S(P);P)$ and
it is sufficiently small so that their images in $S^2$ have no intersection.

Let $D_0$ and $D_1$ be disks in $\hat P$ whose boundaries will be identified.
Set $A_j=\Nbd(\partial D_j;\hat P\setminus\Int \, D_j)$ for $j=0,1$.
The preimage $q_{g_C}^{-1}(A_j)$ has two boundary components, one of which is $q_{g_C}^{-1}(\partial D_j)$ named $T_{j,1}$.
We denote the other boundary component by $T_{j,0}$.
Let $A$ and $\rho_0$ be as in Example \ref{ex:stable maps and S-maps} (\ref{ex:A times S1}).
For each $j=0,1$,
there exist diffeomorphisms $t_j:q_{g_C}^{-1}(A_j)\to A\times S^1$ and
$u_j:[-1/2,1]\times S^1\to \bar g_C(A_j)$ such that $u_j\circ\rho_0\circ t_j|_{T_{j,0}}=g_C|_{T_{j,0}}$.
Then we define an S-map $g_D:\hat M\setminus q_{g_C}^{-1}(\Int \, (D_0\cup D_1))\to S^2$
by
\[
g_D(p)=
\begin{cases}
g_C(p) & \text{for $p\not\in q_{g_C}^{-1}(A_j\cup D_j)$} \\
u_j\circ\rho_0\circ t_j(p) & \text{for $p\in q_{g_C}^{-1}(A_j)$}.
\end{cases}
\]
Remark that $g_D(T_{0,1})$ and $g_D(T_{1,1})$ are simple closed curves in $S^2$ and
the interior of the annulus bounded by these curves does not intersect $g_D(A_j)$ for $j=0,1$.
We attach an annulus $A_2$ between $\partial D_0$ and $\partial D_1$ 
and define the S-map $g_D': M' \to S^2$ 
with Stein factorization 
$M'\overset{q_{g_D'}}{\longrightarrow} W_{g_D'}  \overset{\bar g_D'}{\longrightarrow} S^2$
such that $\bar g_D'(A_2)$ is the annulus bounded by $g_D(T_{0,1})$ and $g_D(T_{1,1})$,
where $M'=(\hat M\setminus q_{g_C}^{-1}(\Int \, (D_0\cup D_1)))\cup(A_2\times S^1)$. 
Applying this construction for each pair of disks in $\hat P$ whose boundary should be identified,
and adjusting the gleam by removing and re-gluing a compact neighborhood 
of the fiber of a point in $\partial R_i$ with keeping the existence of an S-map,
we obtain the S-map $g_E:M\to S^2$ with the required properties. This completes the proof.
\end{proof}

\begin{remark}
\label{rmk:branched shadow complexity and crossing singulatities}
{\rm
\begin{enumerate}
\item
Let $(M, L)$ be as in Theorem 
$\ref{thm:branched shadow complexity and crossing singulatities}$, 
and let $f : (M , L) \to \Real^2$ be an S-map. 
If $\mathrm{II}^3 (f)  = \emptyset$, the Stein factorization 
$W_f$ is exactly a branched shadow of $(M, L)$. 
Suppose that $\mathrm{II}^3 (f)  \neq \emptyset$, and 
let $N_1$, $N_2 , \ldots, N_n$ be closed neighborhoods 
of the singular fibers of type $\mathrm{II}^3$. 
We recall that each $N_i$ is a genus 3 handlebody. 
By the latter part of the 
proof of Theorem 
$\ref{thm:branched shadow complexity and crossing singulatities}$, we may construct 
a map $g_i : N_i \to \Real^2$ using 
the shadowed branched polyhedron depicted in Figure $\ref{fig:branching_of_shadow}$.  
Exchanging $f$ with $g_i$ inside $N_i$ for each $i \in \{ 1 , 2 , \ldots , n\}$, 
we get a new S-map $(M , L) \to \Real^2$ having 
no singular fibers of type $\mathrm{II}^3$. 
\item
Let $(M, L)$ be as in Theorem 
$\ref{thm:branched shadow complexity and crossing singulatities}$ and 
let $P$ be its branched shadow having no boundary-vertices.  
By the latter part of the proof of Theorem 
$\ref{thm:branched shadow complexity and crossing singulatities}$, 
we may construct an S-map $f: E(L) \to \Real^2$ 
with $|\mathrm{II}^2 (f) | = c (P)$ and $  \mathrm{II}^3 (f)    = \emptyset$. 
Then the Stein factorization $W_f$ is also an almost-special polyhedron and 
we have a natural inclusion $S(P) \subset S(W_f)$. 
Apparently, each ``region" of $W_f$ is a planar surface.  
However, even if all the regions of $P$ are planar surfaces, 
$W_f$ is not homeomorphic to $P$ in general. 
\item
In \cite{Kod07}, the second-named author showed that 
the Heegaard genus of a closed orientable $3$-manifold coincides with 
an invariant defined by branched spines. 
(See Endoh-Ishii \cite{EI05} for details of the invariant.) 
Here, a spine of a  closed $3$-manifold $M$ is an almost-special polyhedron $P$ embedded in $M$ 
so that $M$ collapsed onto $P$ after removing a small open ball from $M$. 
In other words, branched spines determine the minimal number of the critical points 
of Morse functions on $M$. 
Theorem $\ref{thm:branched shadow complexity and crossing singulatities}$ can be compared with this fact. 
\end{enumerate}}
\end{remark}

%\begin{remark}
%Since every closed orientable 3-manifold $M$ has a stable map to $S^2$ without 
%cusp points, the argument in the proof of 
%Theorem \ref{thm:branched shadow complexity and crossing singulatities} 
%provide an alternate proof of the existence of a branched shadow of $M$. 
%\end{remark}

The following is a straightforward consequence of Theorem 
\ref{thm:branched shadow complexity and crossing singulatities} 
(The case where $M$ is closed and $L = \emptyset$ is proved by Kalm\'{a}r-Stipsicz \cite[Theorem~1.4]{KS12}). 
See also (1) of the above remark. 
\begin{corollary}
\label{cor:existence of S-maps without CS2}
Let $M$ be a compact, orientable $3$-manifold with $($possibly empty$)$ 
boundary consisting of tori and $L$ a $($possibly empty$)$ link in $M$.
Then there exists an S-map $f: (M , L) \to \Real^2$ without cusp points 
satisfying  
$| \mathrm{II}^2 (f) | = \smc (M, L)$ 
and $\mathrm{II}^3 (f)  = \emptyset$. 
\end{corollary}

\begin{lemma}
\label{lem:connected sum}
Let $M_1$ and $M_2$ be compact orientable $3$-manifolds. 
Then $\bsc(M_1 \# M_2) \leqslant \bsc(M_1) + \bsc(M_2)$. 
\end{lemma}
\begin{proof}
Let $P_1$ and $P_2$ be branched shadows of $M_1$ and $M_2$, respectively. 
Choose an interior point $p_i$ of a region of $P_i$ for each $i \in \{ 1,2 \}$. 
By Costantino-Thurston \cite[Lemma~3.2]{CT08}, 
the polyhedron $P$ obtained by gluing small regular neighborhoods of 
$p_1$ and $p_2$ is 
a shadow of $M$ (after putting the gleam 0 on the new disk region.) 
The polyhedron $P$ apparently admits a branching. 
\end{proof}

By Theorem \ref{thm:branched shadow complexity and crossing singulatities}
and Lemma \ref{lem:connected sum}, the following holds. 
\begin{corollary}
\label{cor:S-map of the connected sum}
Let $M_1$ and $M_2$ be compact orientable $3$-manifolds with $($possibly empty$)$ boundary consisting of tori. 
Then we have $\smc(M_1 \# M_2) \leqslant \smc(M_1) + \smc(M_2)$. 
\end{corollary}

\begin{lemma}
\label{lem:torus sum}
\begin{enumerate}
\item
\label{lem:torus sum_case_1}
Let $M_1$ and $M_2$ be compact orientable $3$-manifolds such that 
both of $\partial M_1$ and $\partial M_2$ contain torus components $T_1$ and $T_2$, respectively. 
Let $\varphi: T_1 \to T_2$ be a diffeomorphism. 
Then $\bsc(M_1 \cup_{\varphi} M_2) \leqslant \bsc(M_1) + \bsc(M_2)$. 
\item
\label{lem:torus sum_case_2}
Let $M$ be a compact oriented $3$-manifold such that 
$\partial M$ contains at least two torus components $T_1$ and $T_2$, which are equipped with 
the orientations induced from that of $M$. 
Let $\varphi: T_1 \to T_2$ be an orientation-reversing diffeomorphism. 
Then $\bsc(M / \varphi) \leqslant \bsc(M)$. 
\end{enumerate}
\end{lemma}
\begin{proof}
\noindent (\ref{lem:torus sum_case_1}) 
Let $P_1$ and $P_2$ be branched shadows of $M_1$ and $M_2$, respectively. 
Let $l_1$ and $l_2$ be the boundary circles of $P_1$ and $P_2$ 
that correspond to the boundary tori $T_1$ and $T_2$, respectively. 
By Costantino-Thurston \cite[Proposition~3.27]{CT08}, 
a shadow $P$ of $M_1 \cup_{\varphi} M_2$ is obtained by 
\[P_1  \cup_{l_1 = \alpha_1} Q \cup_{l_2 = \alpha_2} P_2, \] 
where $Q$ is one of the two oriented polyhedra shown in Figure \ref{fig:torus_sum} equipped with certain gleams. 
\begin{figure}[htbp]
\begin{center}
\includegraphics[width=9cm,clip]{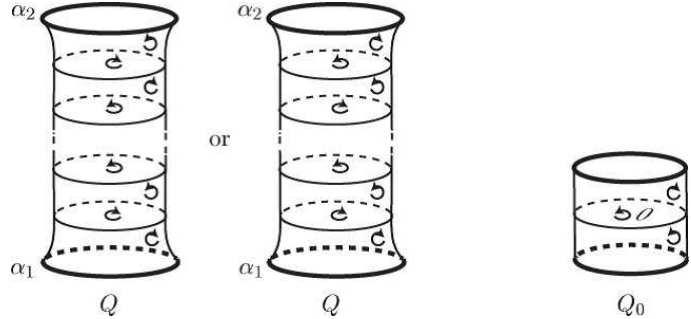}
\caption{Branched shadows in the thickened torus.}
\label{fig:torus_sum}
\end{center}
\end{figure}
Then the orientations of the regions of $Q$ shown in the figure 
extends to the whole of $P$, after reversing the orientations of the regions of $P_1$ or 
$P_2$, if necessary. 

\noindent (\ref{lem:torus sum_case_2})
Let $P_1$ be a branched shadow of $M$. 
Let $l_1$ and $l_2$ be the boundary circles of $P_1$ 
that correspond to the boundary tori $T_1$ and $T_2$, respectively. 
As in (\ref{lem:torus sum_case_1}), 
the polyhedron $P$ obtained by 
attaching $Q$ along the loops $l_1$ and $l_2$ is a shadow of $\bsc(M / \varphi)$. 
If the orientations of the regions of $Q$ shown in the figure 
extends to the whole of $P$ we are done. 
Otherwise, we first attach a shadow polyhedron $Q_0$ shown on the right-hand side in Figure \ref{fig:torus_sum}, 
to $Q$ along $\alpha_1$ and 
then we apply the former operation using $Q_0 \cup Q$ instead of $Q$. 
This does not affect the topology of the resulting manifold, and now 
we may extend the orientation of the regions of $Q_0 \cup Q$ to the whole of $P$.    
\end{proof}

Theorem \ref{thm:branched shadow complexity and crossing singulatities}
and Lemma \ref{lem:torus sum} yield the following: 
\begin{corollary}
\label{cor:S-map of the torus sum}
\begin{enumerate}
\item
\label{lem:torus sum_case_1}
Let $M_1$ and $M_2$ be compact orientable $3$-manifolds with boundary consisting of tori. 
Choose a component $T_i$ of $\partial M_i$ for each $i \in \{1, 2 \}$. 
Let $\varphi: T_1 \to T_2$ be a diffeomorphism. 
Then $\smc (M_1 \cup_{\varphi} M_2) \leqslant \smc (M_1) + \smc (M_2)$. 
\item
\label{lem:torus sum_case_2}
Let $M$ be a compact oriented $3$-manifold with boundary consisting of at least two tori. 
Choose two components $T_1$ and $T_2$ of $\partial M$, and equip them with 
the orientations induced from that of $M$. 
Let $\varphi: T_1 \to T_2$ be an orientation-reversing diffeomorphism. 
Then $\smc (M / \varphi) \leqslant \smc(M)$. 
\end{enumerate}
\end{corollary}

For a link $L$ in a $3$-manifold $M$, we denote by 
$E(L)$ its exterior, that is, 
the complement of an open tubular neighborhood of $L$.

\begin{lemma}
\label{lem:branched shadow complexities of a link and its exterior}
Let $M$ be a compact orientable $3$-manifold with $($possibly empty$)$ boundary 
consisting of tori, and let $L$ be a link in $M$. 
Then we have $\bsc (M) \leqslant \bsc (M, L) = \bsc (E(L))$. 
\end{lemma}
\begin{proof}
The left equality follows immediately from the definition. 

We recall that a disk equipped with the gleam $0$ is a branched shadow of $D^2 \times S^1$. 
Hence by Lemma \ref{lem:torus sum} we have 
\[ \bsc(M,L) \leqslant \bsc (E(L)) + \bsc (D^2 \times S^1) = \bsc (E(L))
. \]
The other inequality follows from the fact that 
by changing 
the colors of the boundary circles of a given branched shadow of $(M, L)$ corresponding to 
$L$, which are all ``$i$", to ``$e$"'s, we get a branched shadow of $E(L)$.  
\end{proof}

By Theorem 
\ref{thm:branched shadow complexity and crossing singulatities}
and Lemma \ref{lem:branched shadow complexities of a link and its exterior}
we have the following:

\begin{corollary}
\label{cor:stable map complexities of a link and its exterior}
Let $M$ be a compact, orientable $3$-manifold with $($possibly empty$)$ 
boundary consisting of tori and $L$ a ink in $M$. 
Then we have $\smc (M) \leqslant \smc (M,L) = \smc (E(L))$. 
\end{corollary}
%\begin{proof}
%By Theorem 
%\ref{thm:branched shadow complexity and crossing singulatities}
%and Lemma \ref{lem:branched shadow complexities of a link and its exterior}
%we have $\smc (M) \leqslant \smc (M,L)$ and 
%\[ 
%\smc (M,L) = \bsc(M,L) = \bsc(E(L)) = \smc (E(L)) , 
%\]
%whence a conclusion.  
%\end{proof}

We recall that a compact orientable 3-manifold is called a {\it graph manifold} if 
we can cut it off by embedded tori into $S^1$-bundles over surfaces. 
A $($possibly empty$)$ link in a compact orientable 3-manifold is called a {\it graph link} if 
its exterior is a graph manifold. 

\begin{proposition}
\label{prop:graph manifolds}
Let $M$ be a compact, orientable $3$-manifold with $($possibly empty$)$ 
boundary consisting of tori and $L$ a $($possibly empty$)$ link in $M$.
Then the following conditions are equivalent: 
\begin{enumerate}
\item
\label{prop:graph manifolds item 1}
$\mathrm{sc}(M, L) = 0$
\item
\label{prop:graph manifolds item 2}
$\bsc(M, L) = 0$
\item
\label{prop:graph manifolds item 3}
$L$ is a graph link. 
\end{enumerate}
%Then $\bsc(M, L) = 0$ if and only if 
%$E(L)$ is a graph manifold. 
\end{proposition}
\begin{proof}
By definition, (\ref{prop:graph manifolds item 1}) follows from (\ref{prop:graph manifolds item 2}). 
Also, by the proof of Costantino-Thurston \cite[Proposition 3.31]{CT08}, 
(\ref{prop:graph manifolds item 3}) follows from (\ref{prop:graph manifolds item 1}).
Below we show that (\ref{prop:graph manifolds item 2}) follows from (\ref{prop:graph manifolds item 3}). 
Suppose that $L$ is a graph $($possibly empty$)$ link. 
By Lemma \ref{lem:branched shadow complexities of a link and its exterior}, it suffices to show that 
$\bsc(E(L)) = 0$. 
By Saeki \cite{Sae96}, there exists an S-map $f: E(L) \to S^2$ such that 
$f |_{S(f)}$ is a smooth embedding. 
In particular, $f$ has no singular fibers of types 
$\mathrm{II}^2$ or $\mathrm{II}^3$, 
and hence $\smc (E(L)) = 0$. 
Thus by Corollary \ref{cor:stable map complexities of a link and its exterior} 
we have $\smc(M,L)=0$. 
Now we get $\bsc (M, L) = 0$ by Theorem \ref{thm:branched shadow complexity and crossing singulatities}. 
%If $\bsc (M, L) = 0$, the shadow complexity of $(M, L)$ is also zero. 
%Hence by the proof of Proposition 3.31 of Costantino-Thurston \cite{CT08}, $E(L)$ is a graph manifold. 
%Suppose that $E(L)$ is a graph manifold. 
%By Lemma \ref{lem:branched shadow complexities of a link and its exterior}, it suffices to show that 
%$\bsc(E(L)) = 0$. 
%By \cite{Sae96}, there exists an S-map $f: E(L) \to S^2$ such that 
%$f |S(f)$ is a smooth embedding. 
%In particular, $f$ has no non-simple crossing singularities and hence $\smc (M, L) = 0$. 
%Thus we have $\bsc (M, L) = 0$ by Theorem \ref{thm:branched shadow complexity and crossing singulatities}. 
\end{proof}

\section{Stable maps of links}
\label{sec:Stable maps of links}

In this section we will always identify $S^1$ with $\Real / 2 \pi \Integer$. 

Let $D$ be a flat oriented disk properly embedded in the standard 4-ball $B^4$, 
and let $D'$ be the closure of $D \setminus \Nbd (\partial D; D) $. 
Let $\pi : S^3 \to D$ be the projection induced by a collapsing $B^4 \searrow D$. 
We may assume without loss of generality that 
the preimage $V_h = \pi^{-1} (D')$ is a solid torus, that we identify with $D' \times S^1$, and 
$\pi |_{V_h}$ is an S-map without singular points. 
We put $V_v = S^3 \setminus \Int \, V_h$. 
Let $L = L_1 \sqcup L_2 \sqcup \cdots \sqcup L_n$ be an $n$-component link in $S^3$. 
We can push $L$ by isotopy so that $L$ lies in $D' \times [[-\varepsilon], [+ \varepsilon]]$, where 
$\varepsilon$ is a sufficiently small positive real number, and $\pi |_{V_h}$ is generic with respect to 
$L$. 
The image $\pi (L)$, together with over/under crossing information at each double point, gives a 
diagram $D_L$ for $L$ on $D'$. 
Consider the mapping cylinder 
\[ P_{D_L}^* = ( (L \times [0,1]) \sqcup D' ) / (x, 0)  \sim \pi (x) . \] 
Then $P_{D_L}^*$, together with a color ``$i$" of each $L_j \times \{ 1 \} \subset \partial P_L^*$， 
is a shadow of $(S^3, L)$. 
Refer to Costantino-Thurston \cite[Examples~3.15-17]{CT08}. 
We fix an orientation of each component of $L$. 
The set of the regions of $P_{D_L}^*$ consists of 
subsurfaces of $D'$ and $A_j = (L_j \times [0,1]) / \sim$, $j \in \{ 1,2, \ldots, n\}$. 
To these subsurfaces of $D'$, we give the orientations induced by the prefixed orientation of $D'$.  
Also, we give an orientation to each $A_j$ in such a way that 
the induced orientation of $L_j \times \{ 1 \}$ coincides with that of the prefixed orientation of $L_j$. 
See Figure \ref{fig:orientations_of_annuli}. 
\begin{figure}[htbp]
\begin{center}
\includegraphics[width=9cm,clip]{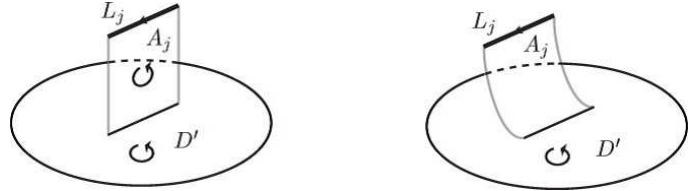}
\caption{The branching on $A_j$.}
\label{fig:orientations_of_annuli}
\end{center}
\end{figure}
With these orientations of the regions, $P_{D_L}^*$ becomes a branched shadow of $(S^3, L)$.  
Let $R$ be the region of $P_{D_L}^*$ touching $\partial D'$. 
 
We say that the oriented link diagram $D_L$ is {\it admissible} if $R$ satisfies the following: 
\begin{itemize}
\item
the closure of $R$ is diffeomorphic to an annulus; and 
\item
on each component of $S(P_{D_L}^*) \setminus V(P_{D_L}^*)$ touched by $R$ and $A_j$ ($j \in \{ 1, 2, \ldots, n\}$), 
these two regions induce the same orientation. 
%on each component of $S(P_{D_L}^*) \setminus V(P_{D_L}^*)$ touched by $R$, 
%the orientations induced by the remaining two regions touching it are opposite. 
\end{itemize}
We note that if we choose, for example, $D_L$ to be a closed braid presentation of $L$ and 
give each component of $L$ a suitable orientation, then $R$ satisfies these conditions. 
If $D_L$ is admissible, $P_{D_L} = P_{D_L}^* \setminus R$ is still a branched shadow of $(S^3, L)$. 
In the following we will discuss this branched shadow. 

Let $f: (S^3, L) \to \Real^2$ be a stable map constructed from $P_L$ 
by the argument in Theorem \ref{thm:branched shadow complexity and crossing singulatities}. 
Let $S^3 \overset{q_f}{\longrightarrow} W_f  \overset{\bar{f}}{\longrightarrow} S^2$ be the 
Stein factorization of $f$. 
From the construction of the stable map, we can identify the quotient space 
$W_f$ with $P_{D_L}$ in a natural way. 
The preimage of a point of $P_{D_L}$ under the map $q_f$ can be described as follows: 
\begin{enumerate}
\item
Let $x_1$ be a point in $D' \setminus P_{D_L} $. 
Then the preimage $q_f^{-1} (x_1)$, which is a regular fiber of $f$, is $\{ x_1 \} \times S^1 \subset V_h$. 
See Figure \ref{fig:interior_region}. 
\begin{figure}[htbp]
\begin{center}
\includegraphics[width=10cm,clip]{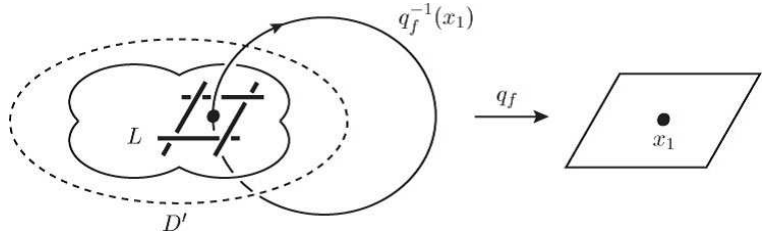}
\caption{The configuration of a regular fiber.}
\label{fig:interior_region}
\end{center}
\end{figure}
\item
Let $x_2$ be a point in $L_i \times (0,1)$. 
Then the preimage $q_f^{-1} (x_2) $, which is also a regular fiber, is a meridian of $L_i$. 
See Figure \ref{fig:exterior_region}. 
\begin{figure}[htbp]
\begin{center}
\includegraphics[width=8cm,clip]{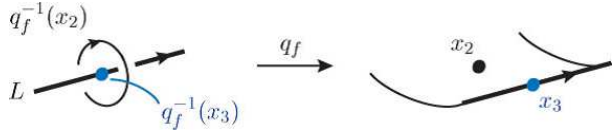}
\caption{The configuration of singular fiber of type $\mathrm{I}^0$.}
\label{fig:exterior_region}
\end{center}
\end{figure}
\item
Let $x_3$ be a point in $L_i \times \{ 1 \}$. 
Then the preimage $q_f^{-1} (x_3)$, which is a singular fiber of type $\mathrm{I}^0$, is a point in $L_i$. 
See Figure \ref{fig:exterior_region}. 
\item
Let $x_4$ be a point in $S(P_{D_L}) \setminus V(P_{D_L})$. 
Let $\alpha$ be a transverse arc for $f$ at $\bar{f} (x_4)$. 
Then the preimage $Y = \bar{f}^{-1} (\alpha)$ is a Y-shaped graph. 
One endpoint of $Y$, say $y_1$, is mapped to one endpoint of $\alpha$ by $\bar{f}$ while 
the remaining two endpoints, say $y_2$ and $y_3$, are 
mapped to the other endpoint of $\alpha$. 
The preimage $q_f^{-1} (Y)$ is a pair of pants spanned by 
the three circles $q_f^{-1} (y_1)$, $q_f^{-1} (y_2)$ and $q_f^{-1} (y_3)$. 
Note that the configuration of these circles has already explained in Cases (1) and (2). 
Then the preimage $q_f^{-1} (x_4)$, which is a singular fiber of type $\mathrm{I}^1$, 
lies in $q_f^{-1} (Y)$ as a spine as illustrated in Figure \ref{fig:edge}. 
%Then the preimage $q_f^{-1} (x_4)$ is as illustrated in Figure \ref{fig:exterior_region}. 
%In the figure, the the boundary points of $Y \times \{0\}$ is denoted by $y_1, y_2, y_3$. 
%The preimage $q_f^{-1} (Y)$ is a pair of pants spanned by 
%the three circles $q_f^{-1} (y_1)$, $q_f^{-1} (y_2)$ and $q_f^{-1} (y_3)$. 
\begin{figure}[htbp]
\begin{center}
\includegraphics[width=14cm,clip]{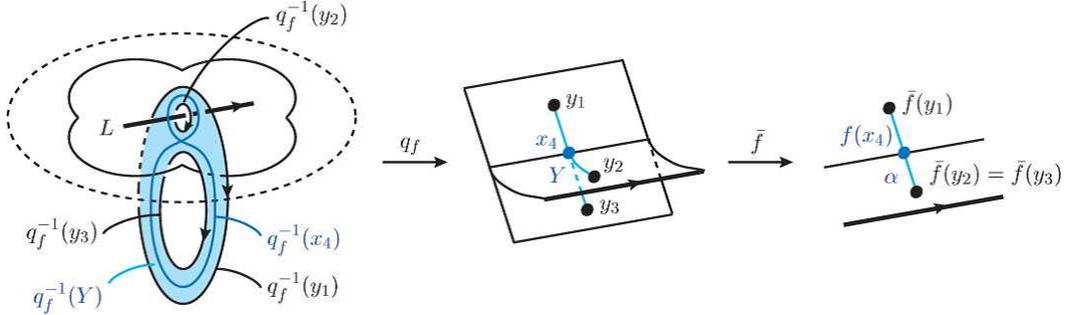}
\caption{The configuration of singular fiber of type $\mathrm{I}^1$.}
\label{fig:edge}
\end{center}
\end{figure}
\item
Finally let $x_5$ be a true vertex of $P_{D_L}$. 
Let $\alpha$ be a transverse arc for $f$ at $\bar{f} (x_5)$ as shown in Figure \ref{fig:vertex}. 
Then the preimage $X = \bar{f}^{-1} (\alpha)$ is the $X$-shaped graph, i.e., the cone on four points. 
One endpoint of $X$, say $y_1$, is mapped to one endpoint of $\alpha$ by $\bar{f}$ while 
the remaining three endpoints, say $y_2$, $y_3$ and $y_4$, are 
mapped to the other endpoint of $\alpha$. 
The preimage $q_f^{-1} (X)$ is a disk with three holes spanned by 
the four circles $q_f^{-1} (y_1)$, $q_f^{-1} (y_2)$, $q_f^{-1} (y_3)$ and $q_f^{-1} (y_4)$. 
Note, again, that the configuration of these circles has already explained in Cases (1) and (2). 
Then the preimage $q_f^{-1} (x_5)$, which is a singular fiber of type $\mathrm{II}^2$, 
lies in $q_f^{-1} (X)$ as a spine as illustrated in Figure \ref{fig:vertex}. 
\begin{figure}[htbp]
\begin{center}
\includegraphics[width=14cm,clip]{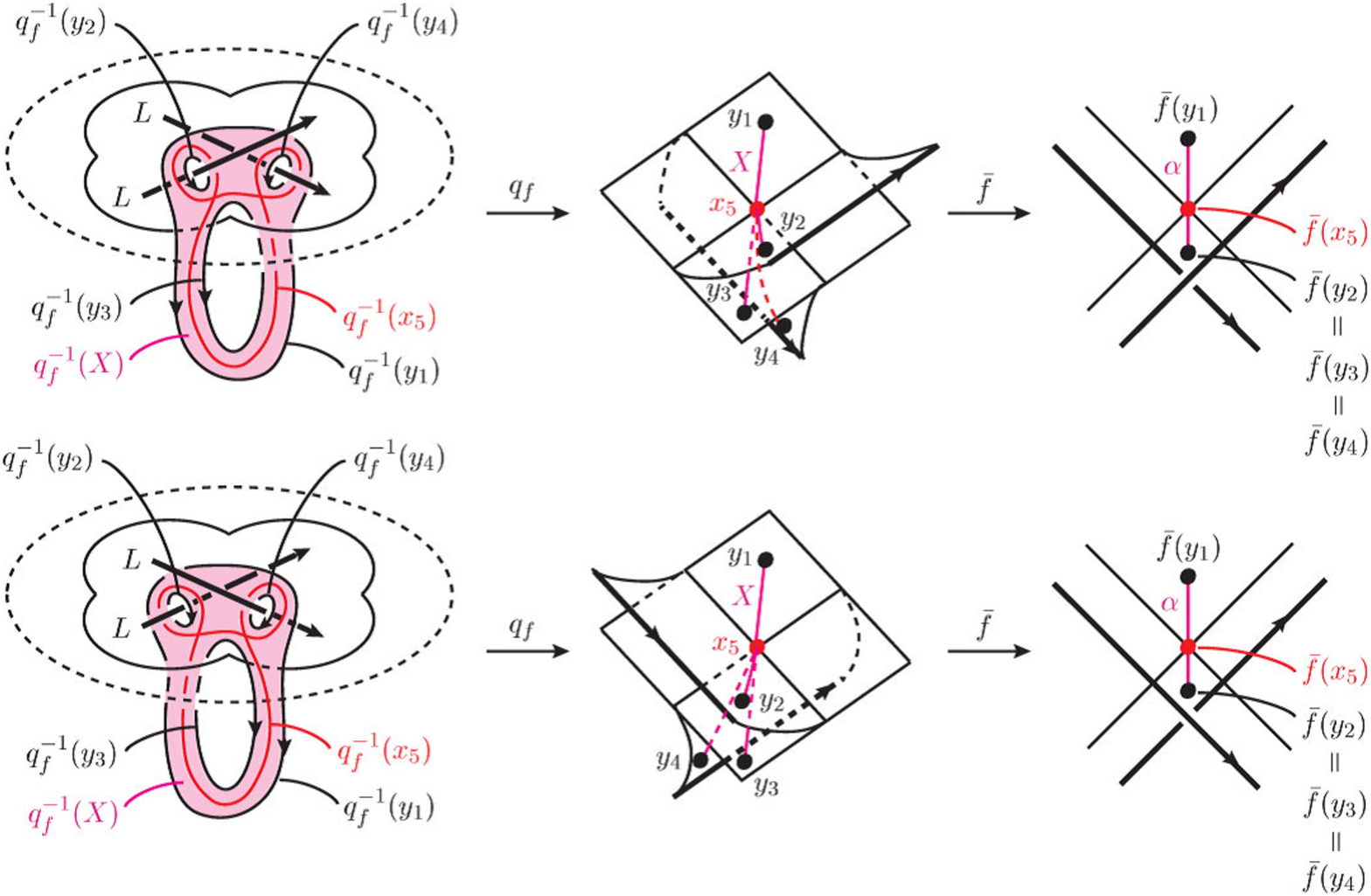}
\caption{The configuration of singular fiber of type $\mathrm{II}^2$.}
\label{fig:vertex}
\end{center}
\end{figure}
\end{enumerate}

Let $L$ be a link in $S^3$. 
Recall that the {\it crossing number} of $L$, denoted by $\cross (L)$, is the minimal number of crossings in any of its diagrams on 
$\Real^2$. 

\begin{lemma}
\label{lem:admissible link diagram}
Let $L$ be an oriented non-split link in $S^3$.
There exists a diagram $D_L$ for $L$ on an oriented disk with $\cross (L)$ 
crossing points such that $D_L$ becomes admissible after reversing 
the orientations of all link components if necessary. 
%Let $L$ be a non-split link in $S^3$. 
%There exists a diagram $D_L$ for $L$ on an oriented disk with $\cross (L)$ crossing points such that 
%$D_L$ becomes admissible after giving suitable orientations to each component of $L$. 
\end{lemma}
\begin{proof}
We start with a diagram $D_L$ for $L$ on the oriented 2-sphere $S^2$ with minimal crossing number. 
%Fix an orientation of $L$. 
Using Seifert's algorithm for $D_L$ we get a disjoint collection $\{ c_1, c_2, \ldots, c_m \} $ of circles on $S^2$. 
We refer to, for instance, Rolfsen \cite{Rol76} for Seifert's algorithm. 
Let $d$ be an open disk component of $S^2 \setminus \bigcup_{i=1}^m c_i$. 
Let $D$ be a disk obtained by removing a small open neighborhood of an interior point of $d$. 
We equip $D$ with the orientation induced by that of $S^2$. 
Then the diagram $D_L$ on $D$ with the orientation of $L$ or its inverse, is the required diagram for $L$. 
\end{proof}

The next theorem immediately follows from the above argument and Lemma \ref{lem:admissible link diagram}.
%By the above argument and Lemma \ref{lem:admissible link diagram}, we have the following theorem: 

\begin{theorem}
\label{thm:Stable maps of links}
For a link $L$ in $S^3$, there exists a proper stable map $f: (S^3, L) \to \Real^2$ without cusp points such that 
\begin{enumerate}
\item
the Stein factorization $W_f$ is contractible; and 
%\item
%$S_0(f) = L$; 
\item
$| \mathrm{II}^2 (f) | \leqslant \cross (L) - 2$ and $\mathrm{II}^3 (f) = \emptyset$.  
%, where $\cro (L)$ is 
%the minimal number of crossing in any of the closed braid presentations of $L$; 
%\item
%$C(f) = \emptyset$.  
\end{enumerate} 
Further, the configuration of both the regular and singular fibers 
%preimages of non-simple crossings 
of $f$ can be described using a diagram of $L$. 
\end{theorem}

\begin{example}
The left-hand side in Figure \ref{fig:figure_eight} shows an admissible diagram $D_K$ of the figure-eight knot $K$. 
The right-hand side in the figure illustrates the contractible branched polyhedron $P_{D_K}$ constructed from $D_K$. 
Then we have a proper stable map $f: ( S^3 , K ) \to \Real^2$ with Stein factorization 
$(S^3, K) \overset{q_f}{\longrightarrow} P_{D_K}  \overset{\bar{f}}{\longrightarrow} \Real^2$ 
such that 
$| \mathrm{II}^2 (f) | = 2$, $\mathrm{II}^3 (f) = \emptyset$ and $C(f) = \emptyset$.  
The configuration of the preimages of the points $x_1, x_2, \ldots, x_6$ in $P_{D_K}$ is shown in the figure. 
\begin{figure}[htbp]
\begin{center}
\includegraphics[width=11cm,clip]{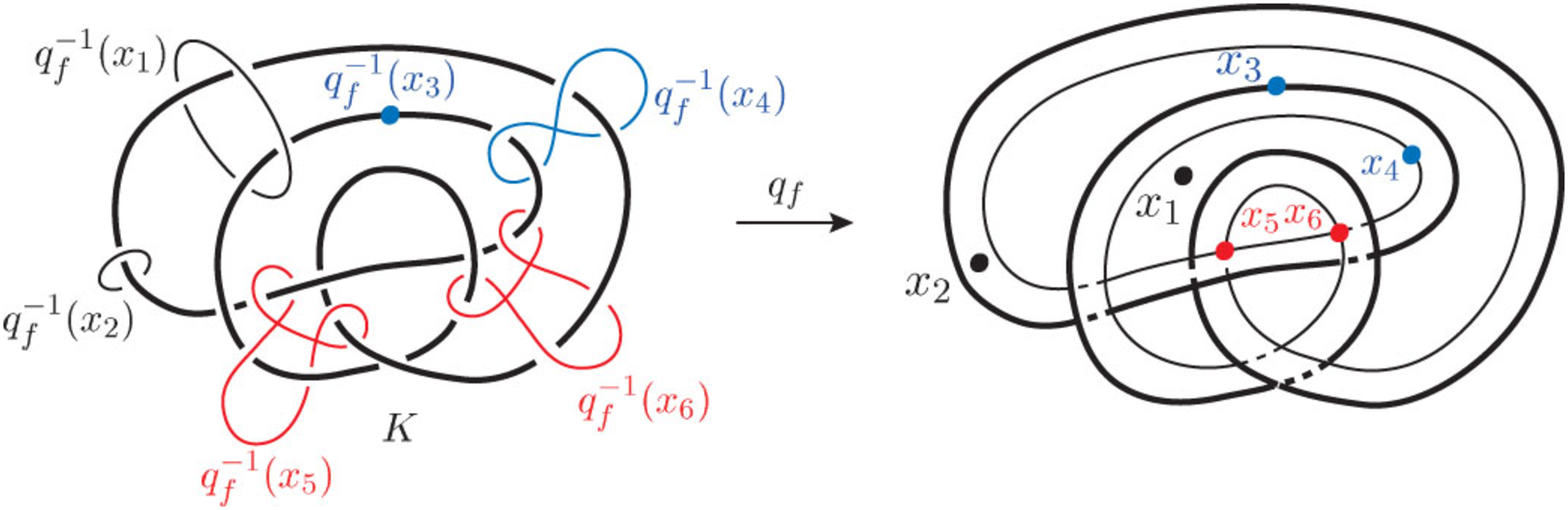}
\caption{Regular and singular fibers of a stable map of the figure-eight knot. 
$q_f^{-1} (x_1)$ and $q_f^{-1} (x_2)$ are regular fibers of $f$. 
$q_f^{-1} (x_3)$ is a singular fiber of type $\mathrm{I}^0$. 
$q_f^{-1} (x_4)$ is a singular fiber of type $\mathrm{I}^1$. 
$q_f^{-1} (x_5)$ and $q_f^{-1} (x_6)$ are singular fibers of type $\mathrm{II}^2$.}
\label{fig:figure_eight}
\end{center}
\end{figure}
\end{example}

From Corollary \ref{cor:stable map complexities of a link and its exterior}
and Theorem \ref{thm:Stable maps of links}, 
the following holds: 
\begin{corollary}
\label{cor:stable map of a surgered manifold}
Let $M$ be a closed orientable $3$-manifold obtained from $S^3$ by surgery along a link $L \subset S^3$. 
Then there exists a stable map $f: M \to \Real^2$ without cusp points such that 
$| \mathrm{II}^2 (f) | \leqslant \cross (L) - 2$ and $\mathrm{II}^3 (f) = \emptyset$.  
%where $\cross (L)$ is the crossing number of $L$. 
%Let $M$ be a closed orientable $3$-manifold obtained from $S^3$ by a surgery along a link $L$. 
%Then there exists a stable map $f: M \to \Real^2$ without cusp points such that 
%\begin{enumerate}
%\item
%the Stein factorization $W_f$ is homotopy equivalent to the wedge sum of finitely many $2$-spheres; 
%\item
%$\CS_1(f) \leqslant \cross (L) - 2$ and $\CS_2(f) = \emptyset$.  
%\end{enumerate} 
\end{corollary}

\begin{remark}
{\rm
A result similar to 
Corollary $\ref{cor:stable map of a surgered manifold}$ 
is obtained in Kalm\'{a}r-Stipsicz \cite[Theorem~1.2]{KS12}. 
The numbers of singular fibers of types $\mathrm{II}^2$ and $\mathrm{II}^3$, 
and cusp points in our Corollary $\ref{cor:stable map of a surgered manifold}$ 
are less than theirs (in particular, $C(f) = \emptyset$ in our result).} 
\end{remark}

\section{Knots with branched shadow complexity 1}
\label{sec:Knots with branched shadow complexity 1}

In this section, we use the arguments in the previous section 
to characterize hyperbolic links in $S^3$ whose exteriors have stable map
complexity $1$.
. 

\begin{lemma}
\label{lem:epimorphism between the fundamental groups}
Let $P$ be a shadow of a compact orientable $3$-manifold $M$. 
Suppose that $P$ is embedded in a compact orientable smooth $4$-manifold $W$ in a locally flat way so that 
$W \setminus P \cong \partial W \times (0,1]$. 
Let $\pi: M \to P$ be a projection given by a collapsing $W \searrow P$. 
Then the homomorphism $\pi_1(M) \to \pi_1(P)$ induced by $\pi$ is an epimorphism. 
\end{lemma}
\begin{proof}
The proof follows immediately from the fact that every closed curve $l$ in $P$ 
can be lifted to a closed curve $\tilde{l}$ in $M$. 
\end{proof}

In the following, let $U_1$, $U_2$, $U_3$ and $U_4$ be the diagrams of the 
trivial knot illustrated in Figure \ref{fig:unknot_diagrams}. 
\begin{figure}[htbp]
\begin{center}
\includegraphics[width=12cm,clip]{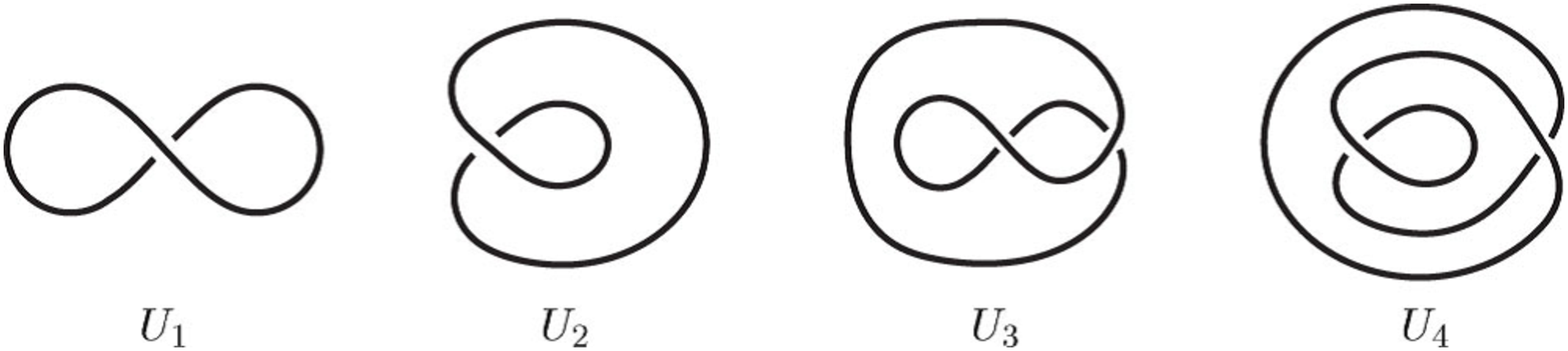}
\caption{The diagrams $U_1, U_2, \ldots, U_4$ of the trivial knot.}
\label{fig:unknot_diagrams}
\end{center}
\end{figure}

\begin{lemma}
\label{lem:non-separating torus}
Let $M$ be a closed orientable $3$-manifold that does not contain non-separating tori, and let $P$ be a shadow of $M$. 
Then any simple closed curve in a region of $P$ separates $P$.
\end{lemma}
\begin{proof}
Let $\pi: M \to P$ be a projection as in Lemma \ref{lem:epimorphism between the fundamental groups}. 
If a simple closed curve $l$ in a region of $P$ does not separate $P$, then $\pi^{-1} (l)$ is a non-separating torus in $M$, 
which is a contradiction. 
\end{proof}

\begin{theorem}
\label{thm:vertex 1 connected}
Let $K$ be a hyperbolic knot in $S^3$.
Then $(S^3, K)$ admits a branched shadow $P$ such that $c (P) = 1$ 
%$($so $| V (P) | =1$ and $| \mathit{BV} (P) | = 0)$ 
and $S(P)$ is connected 
if and only if $K$ is one of the knots shown in Figure $\ref{fig:vertex_1_connected}$. 
\begin{figure}[htbp]
\begin{center}
\includegraphics[width= \textwidth,clip]{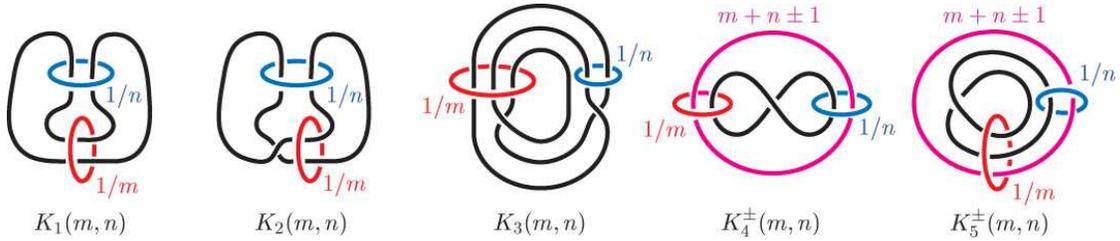}
\caption{The knots admitting a branched shadow with one vertex and connected singular set.}
\label{fig:vertex_1_connected}
\end{center}
\end{figure}
\end{theorem}
\begin{remark}
{\rm
Using {\it Conway's normal form},  
the knots $K_1(m,n)$ and $K_2 (m,n)$ can be written as  
$C(2n, 2m)$ and $C(2n, 2m+1)$, respectively. 
Also, $K^{\pm}_4(m,n)$ is a knot obtained by  performing Dehn surgery on one component of 
$C(2n+1, 2m+1)$ with the coefficients $m + n \pm 1$. 
The knot $K_3 (m,n)$ is the {\it twisted torus knot} 
$T (3, 3m+2; 2, n)$, see 
Callahan-Dean-Weeks \cite{CDW99} for the definition.} 
\end{remark}
\begin{proof}
Suppose that $(S^3, K)$ admits a branched shadow $P$ such that 
$c(P) = 1$ $($so $| V (P) | =1$ and $| \mathit{BV} (P) | = 0)$ and $S(P)$ is connected. 
Then $S(P)$ is one of the two graphs shown in 
Figure \ref{fig:singularities_of_branched_shadows}. 
\begin{figure}[htbp]
\begin{center}
\includegraphics[width=7cm,clip]{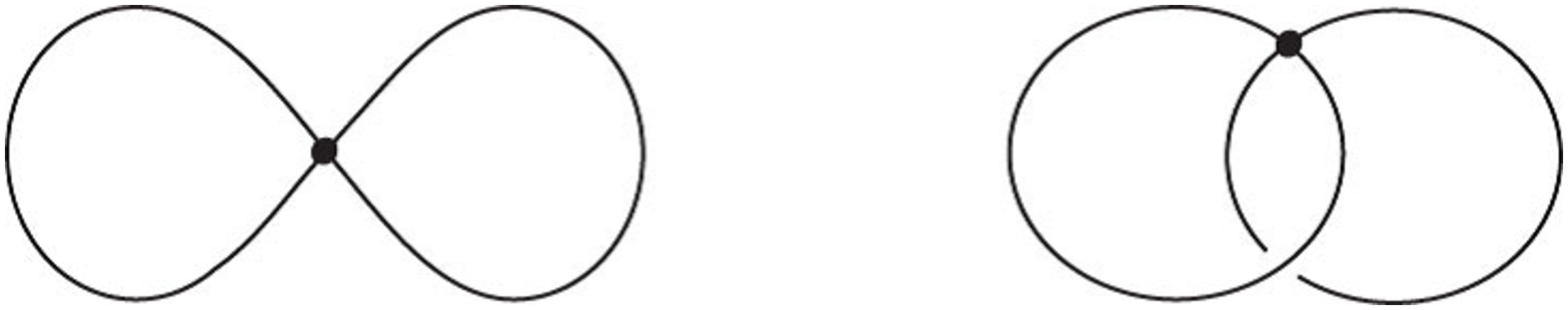}
\caption{Two types of $S(P)$.}
\label{fig:singularities_of_branched_shadows}
\end{center}
\end{figure}
Suppose that $S(P)$ is the one shown on the left-hand side in the figure. 
Then its neighborhood $\Nbd (S(P); P)$ in $P$ is isomorphic to 
one of the four models illustrated in Figure \ref{fig:vertex_1_case_1}. 
\begin{figure}[htbp]
\begin{center}
\includegraphics[width=14cm,clip]{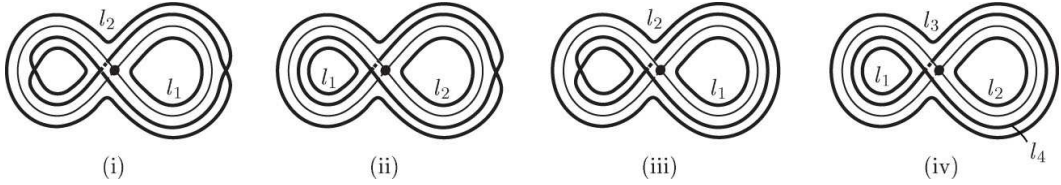}
\caption{Four possible shapes of $\Nbd(S(P);P)$.}
\label{fig:vertex_1_case_1}
\end{center}
\end{figure}
By Lemmas \ref{lem:epimorphism between the fundamental groups} and 
\ref{lem:non-separating torus}, each loop of the boundary of $\Nbd (S(P); P)$ 
separates $P$, and $P$ is a simply connected polyhedron obtained from $\Nbd (S(P); P)$ 
by capping off some of the boundary components by disks.  
Case (i) is impossible because we can not obtain a simply connected polyhedron 
no matter how we cap off boundary components of $\Nbd (S(P); P)$ by disks. 

In Case (ii), $P$ is simply connected if and only if $P$ is obtained from $\Nbd (S(P); P)$ by capping off 
the boundary components $l_1$ and $l_2$ by disks. 
We denote the two disk regions corresponding to $l_1$ and $l_2$ by $D_1$ and $D_2$, respectively. 
Then as a branched polyhedron $P$ is isomorphic to $P_{U_3}$, see the left-hand side in Figure \ref{fig:vertex_1_example_1}. 
\begin{figure}[htbp]
\begin{center}
\includegraphics[width=13cm,clip]{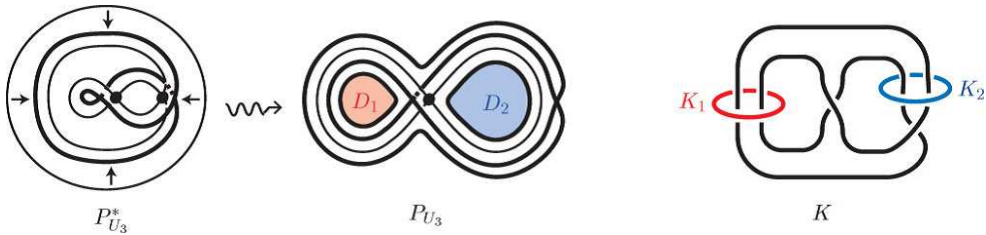}
\caption{The link in Case (ii) in Figure \ref{fig:vertex_1_case_1}.}
\label{fig:vertex_1_example_1}
\end{center}
\end{figure}
Note that in this case, the regions $D_1$ and $D_2$ are equipped with the gleams $1/2$ and $1$, respectively. 
Moreover, we note that for each $i \in \{ 1,2\}$, the preimage of an interior point of $D_i$ under the projection $\pi : S^3 \to P$ is a trivial loop 
$K_i$ illustrated on the right-hand side in Figure \ref{fig:vertex_1_example_1}.   
In general, the gleams of the regions $D_1$ and $D_2$ are $1/2 + m$ and $1 + n$, respectively, where $m$, $n$ are integers. 
Then the shadowed branched polyhedron $P$ is a branched shadow of a knot $K$ shown 
on the right-hand side in Figure \ref{fig:vertex_1_example_1}, 
where we perform the Dehn surgeries on $K_1$ and $K_2$ with the coefficients $1/m$ and $1/n$, respectively. 
Hence $K$ is the knot $K_2 (m,n)$ shown in Figure $\ref{fig:vertex_1_connected}$. 

In Case (iii) $P$ is simply connected if and only if $P$ is obtained from $\Nbd (S(P); P)$ by capping off 
the boundary components $l_1$ and $l_2$ by disks. 
Then as a branched polyhedron $P$ is isomorphic to $P_{U_4}$. 
In the same manner as in Case (ii), we can see that $K$ is the knot $K_3 (m,n)$ 
shown in Figure $\ref{fig:vertex_1_connected}$. 

In Case (iv), $P$ is simply connected if and only if $P$ is obtained from $\Nbd (S(P); P)$ 
in one of the following ways: 
\begin{enumerate}
\item
Cap off the boundary components $l_i$ and $l_j$ by disks, where 
$\{ i , j \}=  \{ 1,2 \}$, $\{ 1,3 \}$, $\{ 1, 4 \}$, $\{ 2 , 3 \}$, or $\{ 2,4 \}$. 
\item
Cap off the boundary components $l_i$, $l_j$  and $l_k$ by disks, where 
$\{ i , j , k\}$ = $\{ 1,2,3 \}$, $\{ 1, 2, 4 \}$, $\{ 1, 3, 4 \}$ or $\{ 2 , 3, 4 \}$.
\end{enumerate}
In Case (1), $P$ is isomorphic to $P^*_{U_1}$ or $P^*_{U_2}$ and hence it is easily seen that 
$K$ is one of the knots shown in Figure \ref{fig:vertex_1_example_2}. 
\begin{figure}[htbp]
\begin{center}
\includegraphics[width=12cm,clip]{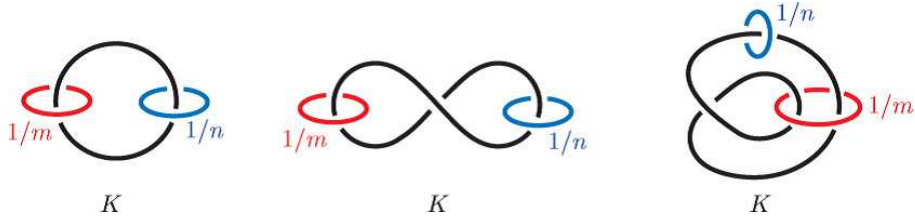}
\caption{The links in Case(iv)-(1).}
\label{fig:vertex_1_example_2}
\end{center}
\end{figure}
It follows that $K$ is the trivial knot or the $(2,m+1)$-torus knot, where $m$ is an integer. 
This contradicts the assumption that $K$ is hyperbolic. 
Next we consider Case (2). 
If $\{ i , j , k\}$ = $\{ 1,2,3 \}$, 
then $P$ is obtained by capping off the boundary component of $P^*_{U_1}$  
that does not correspond to the trivial knot $U_1$ by a disk.  
We denote the two disk regions of $P^*_{U_1}$ by $D_1$ and $D_2$. 
We note that the both gleams on $D_1$ and $D_2$ defined by the construction of $P^*_{U_1}$ are $+1/2$. 
The two boundary circles $K$ and $K'$ of $P^*_{U_1}$ together with 
the preimages $K_1$ and $K_2$ of interior points of $D_1$ and $D_2$, respectively, 
under the projection $\pi : S^3 \to P$ form  
the minimally twisted 4-chain link. 
%Here we assume that $K'$ is the component corresponding to $l_3$. 
The changing of the gleams on $D_1$ and $D_2$ to $1/2 + m$ and $1/2 + n$, respectively, where $m$, $n$ are integers, 
corresponds to Dehn surgeries on $K_1$ and $K_2$ with the coefficients $1/m$ and $1/n$, respectively. 
Moreover, the capping off of $K' \subset \partial P^*_{U_1}$ by a disk 
corresponds to a Dehn surgery on $K'$ with the coefficients $p$, where $p$ is an integer, 
see Figure \ref{fig:vertex_1_example_3}. 
\begin{figure}[htbp]	
\begin{center}
\includegraphics[width=13cm,clip]{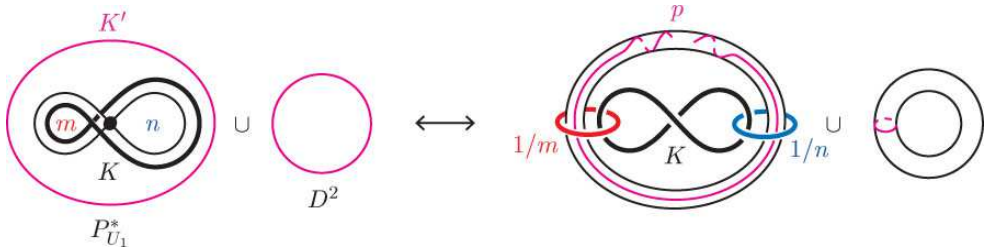}
\caption{The link in the case $\{i,j,k\}=\{1,2,3\}$ or $\{1,2,4\}$.}
\label{fig:vertex_1_example_3}
\end{center}
\end{figure}
Hence $P$ is a branched shadow of a knot in $L(p-m-n,1)$, thus $p=m+n\pm 1$ by assumption. 
In this case, $K$ is the knot $K_4^{\pm} (m,n)$ shown in Figure $\ref{fig:vertex_1_connected}$. 
The same arguments apply to the case $\{ i , j , k\}$ = $\{ 1,2,4 \}$. 
If  $\{ i , j , k\}$ = $\{ 1,3,4 \}$ or $\{ 2,3,4 \}$, 
then $P$ is obtained by capping off the boundary component of $P^*_{U_2}$ 
that does not correspond to the trivial knot $U_2$ by a disk.  
By a similar argument as above, $K$ is a knot described in Figure \ref{fig:vertex_1_example_4}, 
where $m$, $n$, $p$ are integers. 
\begin{figure}[htbp]	
\begin{center}
\includegraphics[width=3cm,clip]{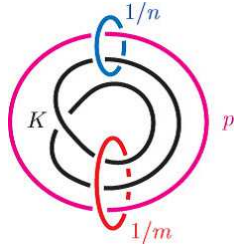}
\caption{The link in the case $\{i,j,k\}=\{1,3,4\}$ or $\{2,3,4\}$.}
\label{fig:vertex_1_example_4}
\end{center}
\end{figure}
Again, $P$ is a branched shadow of a knot in $L(p-m-n,1)$, thus $p=m+n\pm 1$ by assumption. 
In this case, $K$ is the knot $K_5^{\pm} (m,n)$ shown in Figure $\ref{fig:vertex_1_connected}$. 
%Hence it follows from Gordon-Luecke \cite{GL89} that 
%$K$ is a knot in $S^3$ only if $\{m,n,n\} = \{ 0, 0, 1\}$, $\{ 0, 0, -1\}$, $\{ 0, 2, -1 \}$ 
%$\{ -1, 0, 1\}$, $\{ -1, 0, -1\}$, $\{ -1, 2, -1 \}$, and in each case, 
%$K$ is the trivial knot or a torus knot. 
%This is a contradiction. 

Next, suppose that $S(P)$ is the one shown on the right-hand side in Figure \ref{fig:singularities_of_branched_shadows}. 
Then its neighborhood $\Nbd (S(P); P)$ in $P$ is homeomorphic to 
one of the four models illustrated in Figure \ref{fig:vertex_1_case_2}. 
\begin{figure}[htbp]
\begin{center}
\includegraphics[width=14cm,clip]{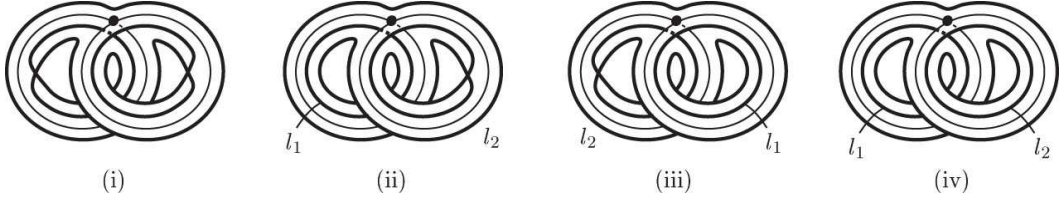}
\caption{The other four shapes of $\Nbd(S(P);P)$.}
\label{fig:vertex_1_case_2}
\end{center}
\end{figure}
None of Cases (i)-(iii) is possible because we can not obtain a simply connected polyhedron 
no matter how we cap off boundary components of $\Nbd (S(P); P)$ by disks. 

In Case (iv), $P$ is simply connected if and only if $P$ is obtained from $\Nbd (S(P); P)$ by capping off 
the boundary components $l_1$ and $l_2$ by disks. 
We denote the two disk regions corresponding to $l_1$ and $l_2$ by $D_1$ and $D_2$, respectively. 
We may embed $P$ into $S^3$ in such a way that $P$ is a spine of $S^3$, that is, 
$S^3$ collapses onto $P$ after removing a small open 3-ball. 
Then the 3-thickening of $P$ can be realized as a regular neighborhood $\Nbd(P; S^3)$, which is 
homeomorphic to a closed 3-ball. 
We first consider the case where the both gleams of $D_1$ and $D_2$ are 0. 
As observed in Costantino-Thurston \cite[Lemma~3.24]{CT08}, 
the 4-manifold reconstructed from $P$ is $\Nbd(P; S^3) \times [-1,1] \cong B^4$ and 
hence $P$ is a branched shadow of $S^3$. 
Moreover, the boundary of $B^4$, which is $S^3$, 
can be identified with the union of $\Nbd(P;S^3)$ and its mirror image
glued along their boundaries by the canonical identity.
The collapsing $B^4 \to P$ is obtained by combining two collapsings $\Nbd(P; S^3) \searrow P$ and $[0,1] \searrow \{ 0 \}$.  
Let $k_1$ and $k_2$ be the preimages of interior points of $D_1$ and $D_2$, respectively, 
under the collapsing $\Nbd(P; S^3) \searrow P$. 
Then on the boundary of $B^4$, the two copies of each $k_i$ give rise to a simple closed curve $K_i$ as described in 
Figure \ref{fig:vertex_1_case_2_torus}. 
\begin{figure}[htbp]
\begin{center}
\includegraphics[width=13cm,clip]{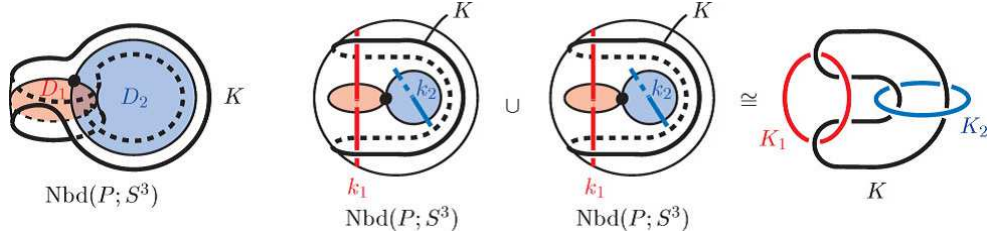}
\caption{The link in Case (iv) in Figure \ref{fig:vertex_1_case_2}.}
\label{fig:vertex_1_case_2_torus}
\end{center}
\end{figure}
In particular, we see that $K \sqcup K_1 \sqcup K_2$ is the Borromean link. 
In the case where the gleams of $D_1$ and $D_2$ are $m$ and $n$, respectively, where 
$m$, $n$ are integers, $K$ is the knot in $S^3$ obtained by the Dehn surgeries on 
$K_1$ and $K_2$ with the coefficients $1/m$ and $1/n$, respectively. 
Hence $K$ is the knot $K_1 (m,n)$ shown in Figure $\ref{fig:vertex_1_connected}$. 

Conversely, suppose that $K \subset S^3$ is a hyperbolic knot shown in Figure \ref{fig:vertex_1_connected}. 
%By Proposition \ref{prop:graph manifolds}, $\bsc(S^3, K) \geqslant 1$. 
By the above arguments, $(S^3 , K)$ admits a branched shadow $P$ such that $c(P) = 1$ and 
$S(P)$ is connected. 
This completes the proof. 
\end{proof}

We call the almost-special polyhedron of the shape depicted in 
Figure \ref{fig:tower} a {\it tower}. 
%We call almost-special polyhedra depicted in 
%Figure \ref{fig:tower} (i) and (ii) a {\it closed tower} and an {\it open tower}, respectively. 
\begin{figure}[htbp]
\begin{center}
\includegraphics[height=3cm,clip]{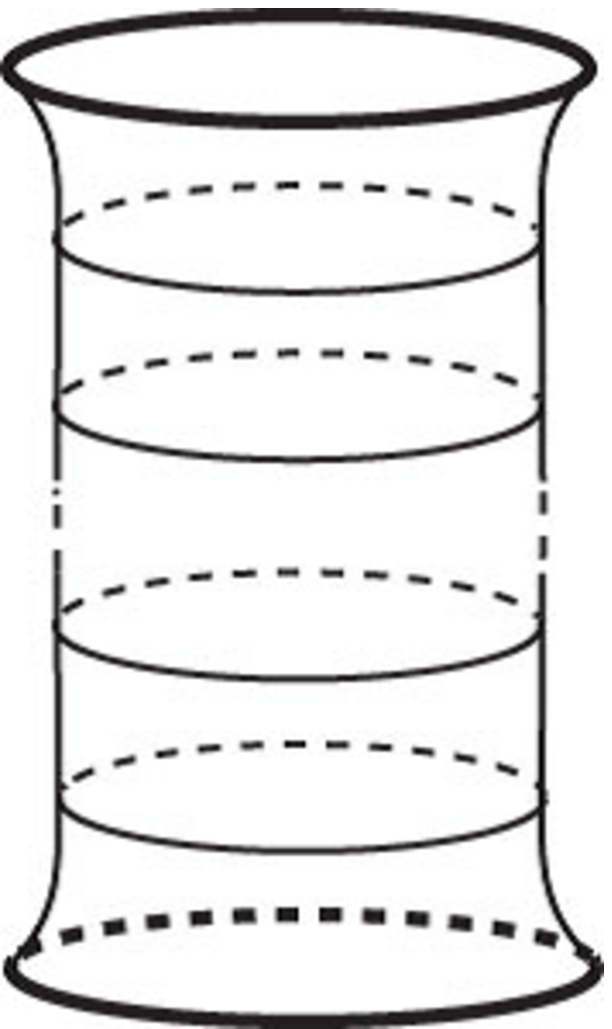}
\caption{A tower.}
\label{fig:tower}
\end{center}
\end{figure}
We require that a tower contains at least one disk region. 

\begin{lemma}
\label{lem:disks or annuli}
Let $L$ be a hyperbolic link in $S^3$ with $\bsc(S^3, L) = 1$, 
and let $P$ be a branched shadow of $(S^3,L)$ with $c(P)=1$.
Let $P'$ be a regular neighborhood of the component of $S(P)$ containing 
the vertex. 
Then there exists a branched shadow $P''$ of $(S^3, L)$ 
which is obtained from $P'$ by attaching towers to some of its boundary components.
\end{lemma}
\begin{proof}
Let $l$ be a component of $\partial P'$. 
 By Lemma  \ref{lem:non-separating torus}, $l$ separates $P$ into two components 
 $P_1$ and $P_2$. 
 Without loss of generality we can assume that 
$P_1$ is the component containing the vertex. 
It suffices to show that we can replace $P_2$ with a tower $Q$ so that the resulting 
shadowed polyhedron $P_1 \cup_l Q$ is again a branched shadow of $(S^3, L)$. 
%Apparently, if $P_2$ is already a tower (including a disk or an annulus), we 
%can choose $Q$ to be $P_2$.  
Let $\pi: S^3 \to P$ be a projection as in Lemma \ref{lem:epimorphism between the fundamental groups}. 
We set $M_i = \pi^{-1} (P_i)$ ($i=1, 2$) and $T = \pi^{-1} (l)$. 
Then at least one of $M_1$ and $M_2$ is a solid torus. 

Suppose that $T$ is parallel to $\partial \Nbd(K)$ in $E(L)$, where $K$ is a component of $L$. 
If $K$ is in $M_1$ then $M_1 \cap L = K$, and hence $P_2$ is a branched shadow of $E(L)$. 
Since $P_2$ has no vertices, it follows from 
Lemma \ref{lem:branched shadow complexities of a link and its exterior} and 
Proposition \ref{prop:graph manifolds} that $E(L)$ is a graph manifold, 
which contradicts the hyperbolicity of $L$. 
If $K$ is in $M_2$ then $M_2 \cap L = K$, hence by replacing $P_2$ with a tower,
we have a branched shadow of $(S^3, L)$. 

Now we assume that $T$ is not parallel to any component of $\partial \Nbd (L)$.
 Suppose that $M_i$ is a solid torus, where $i$ is either $1$ or $2$.

First we consider the case where the solid torus $M_i$ is knotted in $S^3$, 
that is, $M_{3-i}$ is not a solid torus. 
If $M_i \cap L$ is empty, 
we can attach a tower $Q$ to $P_{3 - i}$ along $l$ so that 
the resulting polyhedron $P_{3 - i} \cup_l Q$ is again a branched shadow of $(S^3, L)$ after giving suitable gleams 
to its regions. 
If $i=1$, then $\bsc (S^3, L) = 0$, which contradicts the hyperbolicity of $L$. 
Hence we have $i=2$, which is our assertion. 
Suppose that $M_i\cap L$ is not empty.  
If $M_i\cap L$ is not contained in a $3$-ball in $M_i$, then 
$T$ is an essential torus in $E(L)$. 
By Thurston's Hyperbolization Theorem \cite{Thu82, Mor84, Ota96, Ota98, Kap01} 
this implies $L$ is not hyperbolic,
 which is a contradiction.
If $M_i\cap L$ is contained in a $3$-ball in $M_i$, 
then either $L$ is a split link or $M_{3-i}\cap L=\emptyset$.
The former case contradicts the hyperbolicity of $L$.
In the latter case, 
we can attach a tower $Q$ to $P_i$ along $l$ so that 
the resulting polyhedron $P_i \cup_l Q$ is again a branched shadow of $(S^3, L)$ after giving suitable gleams 
to its regions. 
Note that the operation replacing $P_{3-i}$ by $Q$ corresponds to a Fox re-embedding of $M_i$, and under this re-embedding 
the knot type of $L$ does not change since $L$ is contained in a 3-ball in $M_i$. 
If $i=2$, then $\bsc (L) = 0$, which contradicts the hyperbolicity of $L$.
Hence we have $i=1$, which is our assertion.

Next we consider the case where $M_i$ is unknotted, that is, $M_{3-i}$ is also a solid torus.
Since $L$ is hyperbolic and $T$ is not parallel to any component of $\partial \Nbd (L)$,
 we have either $L\subset M_1$ or $L\subset M_2$, otherwise $T$ is an essential torus in $E(L)$ or $L$ is a split link.
If $L\subset M_2$ then, attaching a tower to $P_2$, we obtain a branched shadow
without vertices. 
Hence we have $L\subset M_1$. 
Then attaching a tower to $P_1$, we obtain a branched shadow of $(S^3, L)$, which is our assertion.
\end{proof}

%Let $L$ be a link in $S^3$. 
%We say that a knot $K$ in $S^3$ is {\it obtained from} $L$ if the exterior of $K$ is diffeomorphic to 
%a 3-manifold obtained by Dehn filling along all but one boundary tori of the exterior of $L$.  
%there exists a 3-component link $L$ in $E(K)$ 
%such that 
%$E(K \sqcup L)$ is diffeomorphic to $E(C_4)$. 

\begin{theorem}
\label{thm:vertex 1 general}
Let $L$ be a hyperbolic link in $S^3$.
Then $\bsc(S^3, L) = 1$ if and only if the exterior of $L$ is diffeomorphic to
a $3$-manifold obtained by Dehn filling the exterior of one of the 
six links $L_1$, $L_2, \ldots, L_6$ in $S^3$ along some of 
$($possibly none of$)$ boundary tori  of its exterior, 
where $L_1$, $L_2, \ldots, L_6$ are 
illustrated in Figure $\ref{fig:vertex_1_general}$.  
%\begin{figure}[htbp]
%\begin{center}
%\includegraphics[width=13cm,clip]{vertex_1_general.eps}
%\caption{}
%\label{fig:vertex_1_general}
%\end{center}
%\end{figure}
\end{theorem}
\begin{proof}
Let $P$ be a branched shadow of $(S^3 , L)$ with one vertex. 
By Lemma \ref{lem:epimorphism between the fundamental groups}, 
$P$ is simply connected. 
Let $P'$ be a regular neighborhood of the component of $S(P)$ containing the vertex. 
Then $P'$ is isomorphic to one of the eight models illustrated in 
Figures \ref{fig:vertex_1_case_1} and 
\ref{fig:vertex_1_case_2}.  
By Lemma \ref{lem:disks or annuli}, 
we may assume that $P$ is obtained from $P'$ 
by attaching towers to some of its boundary components. 
It follows from Seifert-van Kampen Theorem that 
$P$ can be simply connected if and only if 
$P'$ is isomorphic to one of the eight models except 
Figure \ref{fig:vertex_1_case_2} (i). 
Thus it suffices to find a hyperbolic link $L$ in $S^3$ 
for each of the remaining seven models such that the shadow $P'$ equipped with
``$e$" on each of the components of $\partial P'$ 
is a branched shadow of the exterior $E(L)$. 
By Theorem \ref{thm:vertex 1 connected},  these links for 
Figure \ref{fig:vertex_1_case_1} (ii), (iii), (iv) and Figure \ref{fig:vertex_1_case_2} (iv) 
are already obtained as follows: 
(Note that the links $L_3$, $L_4$, $L_6$ are those used to define the knots 
$K_1(m,n)$, $K_2(m,n)$, $K_4^\pm(m,n)$, respectively, 
though each of them is depicted in a different way.)  
\begin{enumerate}
\item
If $P'$ is isomorphic to the model shown in Figure \ref{fig:vertex_1_case_1} (ii), then $P'$ is a branched shadow 
of the exterior of $L_4$. 
\item
If $P'$ is isomorphic to the model shown in Figure \ref{fig:vertex_1_case_1} (iii), then $P'$ is a branched shadow 
of the exterior of $L_5$. 
\item
If $P'$ is isomorphic to the model shown in Figure \ref{fig:vertex_1_case_1} (iv), then $P'$ is a branched shadow 
of the exterior of $L_6$. 
\item
If $P'$ is isomorphic to the model shown in Figure \ref{fig:vertex_1_case_2} (iv), then $P'$ is a branched shadow 
of the exterior of $L_3$. 
\end{enumerate}

Suppose that $P'$ is isomorphic to the model shown in Figure \ref{fig:vertex_1_case_1} (i). 
Then $P$ is obtained from $P'$   
by attaching towers to each of the boundary components $l_1$, $l_2$ of 
$P'$, otherwise $P$ will not be simply-connected. 
Note that, in this case, $S(P)$ is not connected. 
Let $P_0$ be an almost-special polyhedron obtained by capping off  $l_1$ and $l_2$ by disks, where 
$l_1$ and $l_2$ are the boundary components of $\Nbd(S(P) ; P)$ as shown in Figure \ref{fig:vertex_1_case_1} (i).  
Then as a polyhedron $P_0$ is a spine of $S^3$ that is called the {\it abalone} shown in Figure \ref{fig:abalone}.  
See, for instance, Ikeda \cite{Ike71} and Matveev \cite{Mat03} for details of this spine.  
\begin{figure}[htbp]
\begin{center}
\includegraphics[width=10cm,clip]{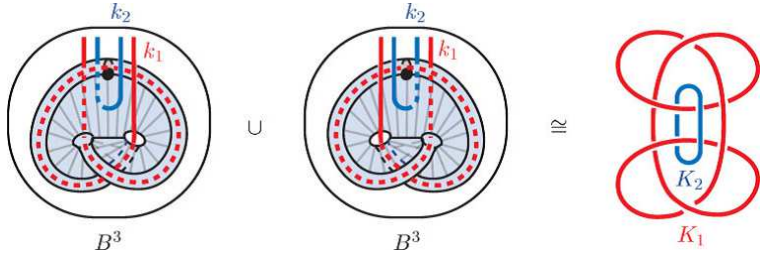}
\caption{The link in Case (i) in Figure \ref{fig:vertex_1_case_1}.}
\label{fig:abalone}
\end{center}
\end{figure}
We denote the two disk regions corresponding to $l_1$ and $l_2$ by $D_1$ and $D_2$, respectively. 
If we assume that both gleams of $D_1$ and $D_2$ are 0, 
then the 3-thickening of $P_0$ can be realized as a regular neighborhood $\Nbd(P_0; S^3)$, which is 
homeomorphic to a closed 3-ball. 
Then as in the case of Figure \ref{fig:vertex_1_case_2} (iv) of Theorem \ref{thm:vertex 1 connected}, 
the 4-manifold reconstructed from $P_0$ is $\Nbd(P_0; S^3) \times [-1,1] \cong B^4$ and 
hence $P_0$ is a branched shadow of $S^3$. 
Moreover, the boundary of $B^4$ can be identified with the union of $\Nbd(P_0;S^3)$ and its mirror image
glued along their boundaries by the canonical identity. 
The collapsing $B^4 \searrow P$ is obtained by combining two collapsings $\Nbd(P_0; S^3) \searrow P_0$ and $[0,1] \searrow \{ 0 \}$.  
Let $k_1$ and $k_2$ be the preimages of interior points of $D_1$ and $D_2$, respectively, under the collapsing $\Nbd(P_0; S^3) \searrow P_0$. 
Then on the boundary of $B^4$, the two copies of each $k_i$ give rise to a simple closed curve $K_i$ as described in 
Figure \ref{fig:abalone}. 
Then the link $K_1 \cup K_2 \subset S^3$ is $L_1$. 
This implies that 
$P'$ equipped with the color ``$e$" on each of the component of $\partial P'$  
is a branched shadow of the exterior $E(L_1)$.

We remark that Figure \ref{fig:vertex_1_case_2} (ii) and (iii) are isomorphic by the reflection. 
Suppose that $P'$ is isomorphic to the model shown in Figure \ref{fig:vertex_1_case_2} (ii). 
Then $P$ is obtained from $P'$ 
by attaching a tower to each of the boundary components $l_1$, $l_2$ of 
$P'$, otherwise $P$ will not be simply-connected. 
Again in this case, $S(P)$ is not connected. 
Let $v$ be the vertex of $P'$. 
Let $\Nbd(v; P) \subset P'$ be a regular neighborhood such that 
$P' \setminus \Nbd(v; P)$ consists of two components $P_1$ and $P_2$, 
where each of $P_1$ and $P_2$ is isomorphic to the direct product of a Y-shaped graph and an interval. 
We assume that $P_1 \cap l_1 \neq \emptyset$ (so $P_2 \cap l_1 = \emptyset$). 
Then $\Nbd(v; P) \cup P_1$ is a shadow of the 3-manifold $M_1$ shown in 
Figure \ref{fig:0-surgery_diagram} while 
$P_2$ is a shadow of the direct product $M_2$ of 
a pair of pants and an interval shown in the same figure. 
\begin{figure}[htbp]
\begin{center}
\includegraphics[width=13cm,clip]{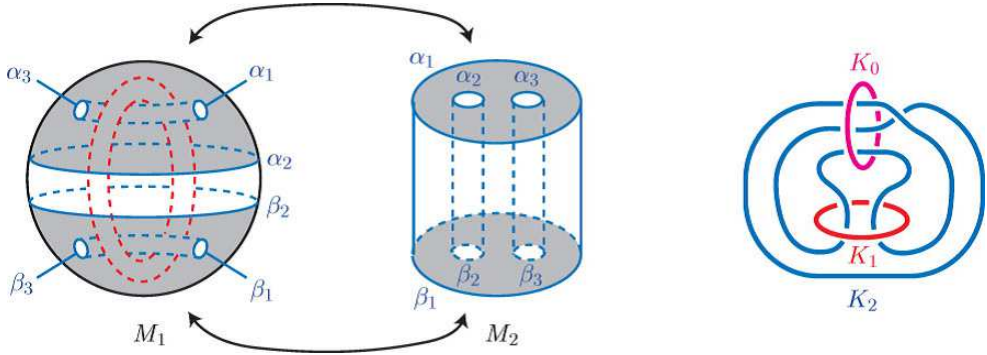}
\caption{The link in Cases (ii) and (iii) in Figure \ref{fig:vertex_1_case_2}.}
\label{fig:0-surgery_diagram}
\end{center}
\end{figure}
We may glue $M_1$ and $M_2$ along the pair of pants as indicated in the figure to get a 3-manifold 
$M$ presented by the branched shadow $P'$. 
Let $K_0 \cup K_1 \cup K_2$ be the link shown 
on the right-hand side in Figure \ref{fig:0-surgery_diagram}. 
It is easy to see that $M$ 
is the exterior of $K_1 \cup K_2$ after performing Dehn surgery along $K_0$ with the coefficient 0. 
We want to describe $M$ as the exterior of a link in $S^3$. 
To do so, we perform Dehn surgery along $K_2$ with the coefficient $1$. 
Let $K_2^*$ be the core of the solid torus corresponding to $K_2$ in the surgered manifold. 
Since $K_0 \cup K_2$ is the Hopf link, the surgered manifold is $S^3$ and 
$K_1 \cup K_2^*$ is a link whose exterior is $M$. 
By an elementary argument of Dehn surgery, we may check that the link $K_1 \cup K_2^*$ 
is $L_2$. 
%In the case of Figure \ref{fig:vertex_1_case_2} (i), we can not obtain a simply connected almost special polyhedron 
%no matter how we attach an open tower to the unique boundary of $\Nbd (S(P); P)$.  

Conversely, suppose that $L \subset S^3$ is a hyperbolic link whose exterior 
is diffeomorphic to
a $3$-manifold obtained by Dehn filling the exterior of one of the 
six links $L_1$, $L_2, \ldots, L_6$ in $S^3$ along some of 
$($possibly none of$)$ boundary tori  of its exterior. 
By Proposition \ref{prop:graph manifolds}, $\bsc(S^3, L) \geqslant 1$. 
Moreover, by the above arguments, $(S^3 , L)$ admits a branched shadow $P$ with $c(P) = 1$. 
This completes the proof. 
\end{proof}
\begin{remark}
{\rm
In Case $(3)$ of the above proof, 
we can also use the $4$-component link depicted in Figure $\ref{fig:vertex_1_example_4}$ instead of $L_6$. 
Actually, this link and $L_6$ have the diffeomorphic complements.} 
\end{remark}
%\begin{remark}
%\begin{enumerate}
%\item
%In Case (3) of the above proof, 
%we can also use the 4-component link depicted in Figure \ref{fig:vertex_1_example_4} instead of $L_6$. 
%Actually, these two knots have the diffeomorphic complements. 
%\item
%Each link $L_i$ of Theorem \ref{thm:vertex 1 general} is an hyperbolic link having 
%the volume $2 V_{\mathrm{oct}}$, where $V_{\mathrm{oct}} = 3.66 \dots$ is the volume of the ideal regular octahedron. 
%See Proposition 3.33 of Costantino-Thurston \cite{CT08}. 
%\item
%By Theorem 9.1 of Agol-Storm-Thurston \cite{AST07}, the link $L_1$ 
%is a  minimal volume hyperbolic link that contains a meridianal incompressible planar surface. 
%See also Example 3.3 of Agol \cite{Ago99}. 
%\item
%In \cite{Yos12} Yoshida proved that the complement of $L_6$ is 
%the minimal volume orientable hyperbolic 3-manifold with $4$ cusps.
%\item
%{\color{red} tunnel 数， strongly invertible か? ．．．などについて分かれば書く}
%\end{enumerate}
%\end{remark}

From Theorems \ref{thm:branched shadow complexity and crossing singulatities} and 
\ref{thm:vertex 1 general}, the following holds: 

\begin{corollary}
\label{cor:Existence of an S-map with smc 1} 
Let $L$ be a hyperbolic link in $S^3$.
Then there exists a stable map $f : (S^3, L) \to \Real^2$ without cusp points 
such that $| \mathrm{II}^2 (f) | = 1$ and $\mathrm{II}^3 (f) = \emptyset$ 
if and only if 
the exterior of $L$ is diffeomorphic to
a $3$-manifold obtained by Dehn filling the exterior of one of the 
six links $L_1$, $L_2, \ldots, L_6$ in Theorem $\ref{thm:vertex 1 general}$ 
along some of $($possibly none of$)$ boundary tori  of its exterior. 
\end{corollary}

Since the correspondence between S-maps and branched shadows 
are so explicit, we can completely determine the configuration
of the singular fibers of the S-maps from the branched shadows.
In the rest of this section, we show the configuration of singular fibers of
type $\mathrm{II}^2$ 
for the links $L_3$, $L_4$, $L_5$ and $L_6$ in 
Theorem \ref{thm:vertex 1 general}.
We omit the other links $L_1$ and $L_2$ since the configuration
seems to be more complicated.
\begin{corollary}
\label{cor:configurations of the unique singular fibers}
%Every knot $K$ in $S^3$ obtained from one of $L_3$, $L_4$, $L_5$ and $L_6$ admits a 
%stable map $f : (S^3, K) \to \Real^2$ satisfying $c(f) = 1$. 
For $i \in \{ 3,4,5, 6 \}$,  the exterior of the link $L_i$ admits an S-map $f : E(L_i) \to \Real ^2$ 
with $c(f) = 1$ such that 
the configuration of the unique singular fiber of type $\mathrm{II}^2$ 
%of the unique non-simple crossing $x$ of $f$ 
is shown in Figure $\ref{fig:vertex_1_general_singular_fibers}$. 
\begin{figure}[htbp]
\begin{center}
\includegraphics[width=14cm,clip]{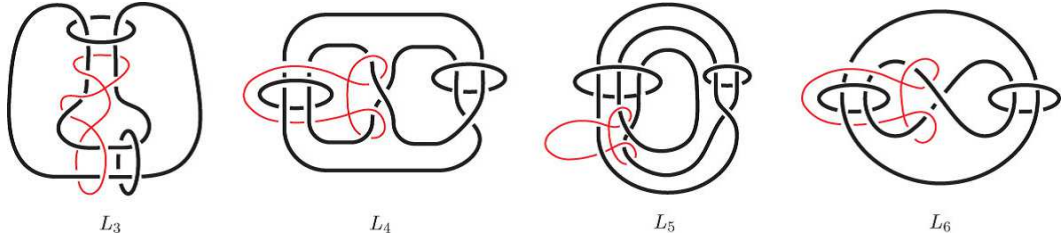}
\caption{The configuration of singular fibers in Corollary \ref{cor:configurations of the unique singular fibers}.}
\label{fig:vertex_1_general_singular_fibers}
\end{center}
\end{figure}
\end{corollary}
\begin{proof}
%The former assertion follows immediately from Theorems \ref{thm:branched shadow complexity and crossing singulatities}. 
%and \ref{thm:vertex 1 general}. 
Since the links $L_4$, $L_5$, $L_6$ admit branched shadows 
$P_{U_3}$, $P_{U_4}$, $P^*_{U_1}$, respectively, 
we get an S-map $f : E(L_i) \to \Real^2$ ($i \in \{ 4,5,6 \}$) 
with $c(f) = 1$ using these branched shadows and then we can find easily 
the configuration of its unique singular fiber of type $\mathrm{II}^2$ 
as shown in Figure \ref{fig:vertex_1_general_singular_fibers}, 
following the method introduced in Section \ref{sec:Stable maps of links}. 
Hence it remains to show the case of $L_3$. 
We recall that the exterior of $L_3$ admits a branched shadow $P'$ depicted in Figure \ref{fig:vertex_1_case_2} (iv). 
We use the same notation as in the proof of Theorem \ref{thm:vertex 1 general}, 
that is, $l_1$ and $l_2$ are the components of $\partial P'$ shown in Figure \ref{fig:vertex_1_case_2} (iv) and 
$L_3$ consists of 3 components $K$, $K_1$, $K_2$ as 
shown in Figure \ref{fig:vertex_1_case_2_torus}. 
Let $f : E(L_3) \to \Real ^2$ be an S-map with $c(f) = 1$ 
constructed from the branched shadow depicted in Figure \ref{fig:vertex_1_case_2} (iv) using the argument in 
Theorem \ref{thm:branched shadow complexity and crossing singulatities}.  
Then $P'$ can be naturally identified with $W_f$, where 
$E(L_3) \overset{q_f}{\longrightarrow} W_f  \overset{\bar{f}}{\longrightarrow} \Real^2$ is 
the Stein factorization of $f$.  
We denote the vertex of $P'$ by $v$. 
Let $x_1$, $x_2, \ldots, x_6$, $y_1$, $y_2, \ldots, y_6$ be the points on $\partial \Nbd (v ; P') \cap \partial P'$ 
shown on the left-hand side in Figure \ref{fig:L_3_singular_fiber}, and set
$\partial \Nbd (v ; P') \cap S(P') = \{ z_1, z_2, z_3, z_4 \}$ as shown in the same figure. 
\begin{figure}[htbp]
\begin{center}
\includegraphics[width=14cm,clip]{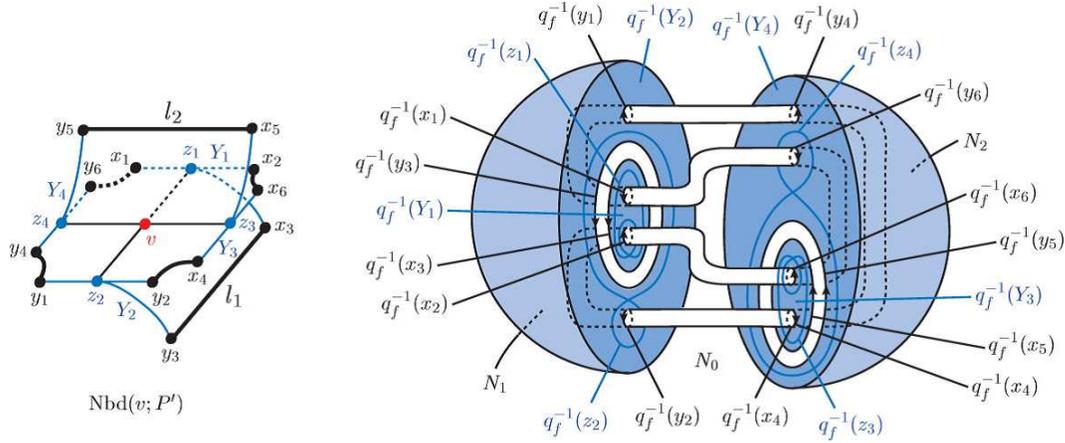}
\caption{The configuration of singular fibers for $L_3$.}
\label{fig:L_3_singular_fiber}
\end{center}
\end{figure}
Let $Y_i$ ($i \in \{ 1 , 2 , 3 , 4 \}$) be the closure of the Y-shaped component of 
$\partial \Nbd (v ; P') \setminus \partial P'$ containing $z_i$. 
Then by the argument in Theorem \ref{thm:vertex 1 connected}, 
$q_f^{-1} (x_3)$ and $q_f^{-1} (y_3)$ are longitudes of the trivial knot $K_1$, 
and $q_f^{-1} (x_5)$ and $q_f^{-1} (y_5)$ are longitudes of $K_2$, 
while all the other $x_i$'s and $y_i$'s are meridians of $K$. 
The preimage $Y_i$ ($i \in \{ 1 , 2 , 3 , 4 \}$) is a pair of pants spanned by the 
simple closed curves $q_f^{-1} (\partial Y_i)$ as shown 
on the right-hand side in Figure \ref{fig:L_3_singular_fiber}. 
We note that once we fix an orientation of $E(L_3)$ and $\Real^2$, we may give 
an orientation of each of $q_f^{-1} (x_i)$ and $q_f^{-1} (y_i)$ ($i \in \{ 1 , 2 , \ldots, 6 \}$)  
in a natural way. 
Then by the branching of the Stein factorization $P'$, 
the preimage $q_f^{-1} (z_i)$, which is an immersed figure-8 shaped curve, lies in $q_f^{-1} (Y_i)$ 
as shown in the same figure. 
The union $\bigcup_{i=1}^4 q_f^{-1} (Y_i)$ cuts $E(L_3)$ into a genus 3 handlebody $N_0$ and 
two genus 2 handlebodies $N_1$, $N_2$, where 
$N_1$ is bounded by 
$q_f^{-1} ( Y_1 )$ and $q_f^{-1} ( Y_2 )$, and 
$N_2$ is bounded by $q_f^{-1} ( Y_3 )$ and $q_f^{-1} ( Y_4 )$. 
We note here that each of $N_1$ and $N_2$ is the product $\mbox{(a pair of pants)} \times [0,1]$ corresponding to 
the preimage of one of the two components of the closure of $P' \setminus \Nbd (v ; P')$ under $q_f$. 

Now, the preimage $q_f^{-1} (v)$, which is our target, lies in $H_0$. 
As reviewed in Section \ref{subsec:Stable maps and their Stein factorizations}, 
$q_f^{-1} (\Nbd (v; P'))$ has a standard product structure 
$\mbox{(a disk with 3 holes)} \times [0,1]$ for which 
the preimages of $z_1$, $z_2$, $z_3$, $z_4$ and $v$ are illustrated on the left-hand side in Figure 
\ref{fig:L_3_singular_fiber_2} (cf. Figure \ref{fig:the stein factorization for II3}). 
\begin{figure}[htbp]
\begin{center}
\includegraphics[width=12cm,clip]{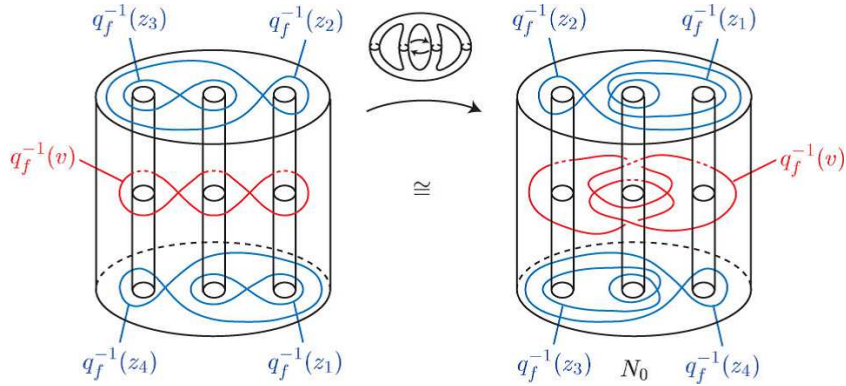}
\caption{Twisting the handlebody.}
\label{fig:L_3_singular_fiber_2}
\end{center}
\end{figure}
The right-hand side in Figure \ref{fig:L_3_singular_fiber} also gives a 
$\mbox{(a disk with 3 holes)} \times [0,1]$ of the handlebody 
$H_0$ as shown on the right-hand side in Figure 
\ref{fig:L_3_singular_fiber_2}. 
However, this product structure is not consistent with the standard one. 
There exists an orientation-preserving diffeomorphism between 
the handlebodies, which is realized by switching the two handles as described in the figure. 
Using this diffeomorphism, we can find the configuration of $q_f^{-1} (v)$ inside $N_0$. 
This completes the proof. 
\end{proof}

\begin{example}
The knot $K_2(1,1)$ is the figure-eight knot. 
By Corollary \ref{cor:configurations of the unique singular fibers} 
the exterior of $K_2(1,1)$ admits an S-map with $c(f)=1$ whose unique singular fiber of type $\mathrm{II}^2$ 
is shown in Figure \ref{fig:figure_eight_smc_1}. 
\begin{figure}[htbp]
\begin{center}
\includegraphics[width=2.5cm,clip]{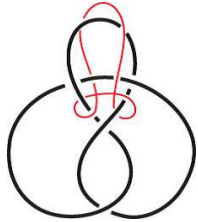}
\caption{The configuration of the singular fiber of type $\mathrm{II}^2$
in the figure-eight knot complement.}
\label{fig:figure_eight_smc_1}
\end{center}
\end{figure}

\end{example}

%\section{Big gleams bound hyperbolic volume from below}
%\label{sec:Big gleams bound hyperbolic volume from below}

\section{Stable maps and hyperbolic volume}
\label{sec:Stable maps and hyperbolic volume}

%In this section we discuss a lower bound on hyperbolic volume in terms of a combinatorial 
%information of a shadow. 
%What we will discuss below is a description of the result of 
%Futer-Kalfagianni-Purcell \cite{FKP08} in terms of shadows. 
%The contents of this section is not essentially new. 

Throughout the section, we consider a particular type of polyhedron, a {\it special polyhedron}. 
An almost-special polyhedron $P$ is said to be {\it special} 
if there is no loop without vertices in $S(P)$ and 
each region of $P$ is a disk. 
We note that in this case, $S(P)$ is connected and $P$ contains no boundary vertices. 
We call a shadow of a 3-manifold that is a special polyhedron a {\it special shadow} of $M$.  
Remark that  every closed orientable 3-manifold admits a branched, special shadow by the moves 
described in Turaev \cite{Tur94} and Costantino \cite{Cos05a}. 
See also Martelli \cite{Mar05} for interesting properties of special shadows. 

Let $M$ be a closed orientable $3$-manifold with special shadow $P$ and
$\pi:M\to P$ be the projection induced by the collapsing $W\searrow P$, where $M=\partial W$. 
Set $M_{S(P)} = \pi^{-1} (\Nbd (S(P) ; P))$. 
Costantino-Thurston \cite{CT08} showed that $M_{S(P)}$ admits a 
complete, finite volume hyperbolic structure realized by gluing $2 c(P)$ copies of a regular ideal octahedron.
Thus, in particular, we have $\vol (M_{S(P)}) = 2 c(P) V_{\mathrm{oct}}$. 
Since each region of a special polyhedron is a disk, $M$ is obtained from $M_{S(P)}$
by attaching solid tori, i.e., by Dehn fillings.
In particular, by the 6-Theorem of Agol \cite{Ago00} and Lackenby \cite{Lac00}
 and the Geometrization Theorem of Perelman \cite{Per02, Per03a, Per03b},
if all slope lengths of the Dehn fillings are more than $6$ then $M$ admits a complete finite volume hyperbolic structure.
Since the hyperbolic structure of  $M_{S(P)}$ is explicitly given by the ideal octahedra,
 the slope lengths of Dehn fillings can be calculated in terms of the combinatorial structure of 
the special polyhedron $P$ and the gleams on its regions.

Let $P$ be a shadowed, special polyhedron.
For each region $R$ of $P$, set ${\mathrm{sl}}(R)=\sqrt{(2g)^2 + k^2}$, where
$g \in \frac{1}{2} \Integer$ is the gleam on $R$ and
$k$ is an integer counting how many times 
the boundary of the closure of $R$ passes through the vertices of $P$. 
We will show later, in Lemma \ref{lem:length and weight}, 
that ${\mathrm{sl}}(R)$ is nothing but the slope length of the Dehn filling for 
the corresponding boundary torus when we obtain $M$ 
from the hyperbolic manifold $M_{S(P)}$.
We set $\mathrm{sl}(P) = \min_R {\mathrm{sl}}(R)$, where 
$R$ varies over all regions of $P$. 

\begin{proposition}
\label{lem:lower bound of volume}
Let $M$ be a closed orientable $3$-manifold. 
Let $P$ be a branched, special shadow of $M$. 
If ${\mathrm{sl}}(P) > 2 \pi$, then we have 
\begin{eqnarray*}
2 \, \mathrm{smc} (M) V_{\mathrm{oct}}  \left(  1 - \left( \frac{2 \pi}{ {\mathrm{sl}}(P) } \right)^2 \right)^{3/2}   
&\leqslant&
2 \, c (P) V_{\mathrm{oct}}  \left(  1 - \left( \frac{2 \pi}{ {\mathrm{sl}}(P) } \right)^2 \right)^{3/2}   \\
&\leqslant&  \vol (M) 
< 
2 \, \mathrm{smc} (M) V_{\mathrm{oct}}  .
\end{eqnarray*}
\end{proposition}
\begin{proof}%[Proof of Proposition $\ref{lem:lower bound of volume}$]
The first inequality is immediate from the definition. 
The second inequality follows from Theorem \ref{thm:branched shadow complexity and crossing singulatities}, 
Lemma \ref{lem:length and weight} below, and 
Futer-Kalfagianni-Purcell \cite[Theorem~1.1]{FKP08}. 
The last inequality follows from Costantino-Thurston \cite[Theorem~3.37]{CT08} and 
Theorem \ref{thm:branched shadow complexity and crossing singulatities}. 
\end{proof}

From these inequalities we have the following result that concerns 
the coincidence of shadow complexities, branched shadow complexities and stable map complexities.

\begin{theorem}
\label{thm:linear upper and lower bound of sms}
Let $M$ be a closed orientable $3$-manifold, and let $P$ be a branched, special shadow of $M$.
If $\mathrm{sl}(P) > 2\pi\sqrt{2c(P)}$, then we have
 $\mathrm{sc}(M)=\mathrm{bsc}(M)=\mathrm{smc}(M)=c(P)$.
\end{theorem}
\begin{proof}
The inequality $\mathrm{sl}(P) > 2\pi\sqrt{2c(P)}$ implies that
\[
1-\frac{1}{c(P)} < 1-2\left(\frac{2\pi}{\mathrm{sl}(P)}\right)^2
 < \left(1-\left(\frac{2\pi}{\mathrm{sl}(P)}\right)^2\right)^{3/2}.
\]
Thus we have the following inequalities:
\[
   0<c(P)-c(P)\left(1-\left(\frac{2\pi}{\mathrm{sl}(P)}\right)^2\right)^{3/2}<1.
\]
 Since $\mathrm{sc}(M)$, $\bsc(M)$ and $\smc(M)$ are all less than or equal to $c(P)$,
the above inequalities hold even if we replace $c(P)$ by  $\mathrm{sc}(M)$, $\bsc(M)$ and $\smc(M)$.
Applying this to the inequalities in Proposition~\ref{lem:lower bound of volume},
 we have
\[
   0\leqslant \smc(M) -\frac{\vol (M)}{2V_{\mathrm{oct}}}<1.
\]
By the way, it is easy to check that the inequalities in Proposition \ref{lem:lower bound of volume} hold
not only for the stable map complexity but also the shadow complexity, the branched shadow complexity
and the number of vertices of the polyhedron $P$. 
Therefore the above inequality holds even if we replace $\smc(M)$ by $\mathrm{sc}(M)$, $\bsc(M)$ and $c(P)$. 
Since $\mathrm{sc}(M)$, $\bsc(M)$, $\smc(M)$ and $c(P)$ are integers, they coincide.
 \end{proof}

Let $P$ be a special shadow of a closed orientable 3-manifold $M$. 
Costantino-Thurston \cite[Proposition~3.34]{CT08} described 
the Euclidean structure on the cusps of $M_{S(P)}$ in terms of 
{\it $\Integer_2$-gleam} and the 
numbers of the vertices through which the boundaries of the regions of $P$ runs. 
The proof of the following lemma is not essentially new. 
In fact, it is described implicitly in \cite{CT08}. 
The maximal horocusp had been observed in Costantino-Frigerio-Martelli-Petronio \cite{CFMP07}
though gleams are excluded from the discussion. 
Therefore we believe that it deserves to be clarified 
in our setting with details.
\begin{lemma}
\label{lem:length and weight}
Let $P$ be a special shadow of a closed orientable $3$-manifold $M$.
Let $R_1$, $R_2 , \ldots, R_n$ be the regions of $P$.
Set $l_i = R_i \cap \partial \Nbd (S(P); P)$, and let
$T_i$ be the component of the boundary of $M_{S(P)}$ corresponding to $l_i$
for $i \in \{ 1, 2, \ldots, n \}$. 
%Let $P$ be a special shadow of a closed orientable $3$-manifold $M$. 
%Let $R_1$, $R_2 , \ldots, R_n$ be the regions of $P$. 
%Set $l_i = R_i \cap \Nbd (S(P); P)$ $(i=1$, $2, \ldots, n)$.  
%Let $T_i$ be the component of $ \partial M_{S(P)}$ corresponding to $l_i$ $(i=1$, $2, \ldots, n)$. 
Then there exists a maximal horoball neighborhood 
$C = C_1 \cup C_2 \cup \cdots \cup C_n$ of the cusps of the interior of $M_{S(P)}$, 
where $C_i$ corresponds to $T_i$, such that 
%Then there exist disjoint horoball neighborhoods 
%$C_1$, $C_2 , \ldots , C_n$ of the cusps of $\Int M_{S(P)}$ corresponding to 
%$T_1$, $T_2 , \ldots , T_n$, respectively, such that 
$M$ is obtained from $M_{S(P)}$ by Dehn fillings along the slopes 
$s_1 \subset T_1$, $s_2 \subset T_2 , \ldots, s_n \subset T_n$ whose lengths 
with respect to $C$ are 
%with respect to $C_1$, $C_2 , \ldots , C_n$ are 
$\mathrm{sl}(R_1)$, $\mathrm{sl}(R_2) , \ldots, \mathrm{sl}(R_n)$, respectively. 
%$w_1$, $w_2 , \ldots, w_n$, respectively. 
\end{lemma}
\begin{proof}
For each vertex $v_i$ ($i \in \{ 1 , 2, \ldots, c(P) \}$), we set $P_i = \Nbd (v_i ; P)$.  
We regard  $\Nbd (S(P); P)$ as the union $\bigcup_{i=1}^{c(P)} P_i $. 
For the 3-dimensional thickening $X_i$ of $P_i$, 
we fix a collapsing $\rho_i : X_i \searrow P_i$ so that: 
\begin{itemize}
\item
for a point $y$ in $P_i \setminus S(P_i)$, $\rho_i^{-1} (y) = \{ y\} \times [-1, 1]$; 
\item
for a point $y$ in $S(P_i) \setminus \{v_i\}$, $\rho_i^{-1} (y)$ is a Y-shaped graph; and 
\item
$\rho_i^{-1} (v_i)$ is an X-shaped graph. 
\end{itemize}
Then $\partial X_i \setminus \Int \, (\rho_i^{-1} (\partial P_i))$ consists of four disks 
$d_{i, 1}$, $d_{i, 2}$, $d_{i, 3}$ and $d_{i, 4}$, 
and the closure of $\rho_i^{-1} (\partial P_i) \setminus \Nbd (\rho_i^{-1} (S(P)))$ consists of 
six squares $b_{i,1}$, $b_{i,2} , \ldots, b_{i,6}$. 
We foliate each of these squares by intervals as shown in  
Figure \ref{fig:3-dim_thickening}. 
\begin{figure}[htbp]
\begin{center}
\includegraphics[width=5cm,clip]{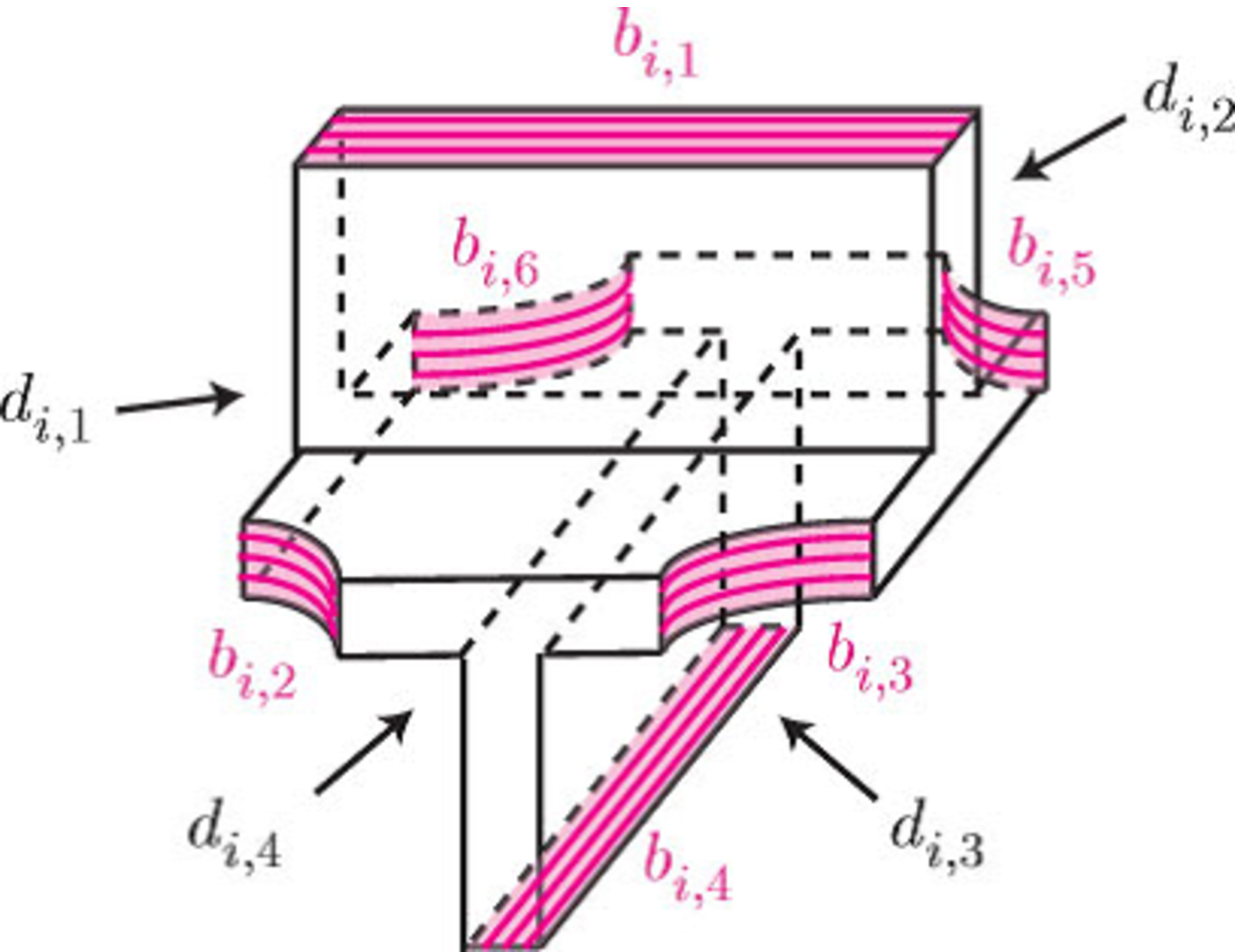}
\caption{$P_i=\Nbd(v_i;P)$.}
\label{fig:3-dim_thickening}
\end{center}
\end{figure}
Set $W_i = X_i \times [-1, 1]$. 
We note that the 4-thickening $W_0$ of $\Nbd (S(P) ; P)$, which is 
the subbundle of the determinant line bundle over $X_i$ whose fiber is 
$[-1, 1]$ (after giving an Euclidean metric over this bundle), is obtained by 
gluing the pieces $W_1$, $W_2 , \ldots, W_{c(P)}$. 
Then the boundary $\partial W$ of  the 4-thickening $W$ of $P$ 
decomposes as: 
\[
W_i \cap \partial W = X_i \times \{  -1 \} 
\bigcup_{ d_{i,j} \times \{ -1 \} } 
\left( \left( \bigsqcup_{i=1}^4 d_{i,j} \right) \times [ -1, 1 ] \right) 
\bigcup_{ d_{i,j} \times \{ 1 \} } 
X_i \times \{ 1 \} . 
\] 
See Figure \ref{fig:octahedra}. 
\begin{figure}[htbp]
\begin{center}
\includegraphics[width=14cm,clip]{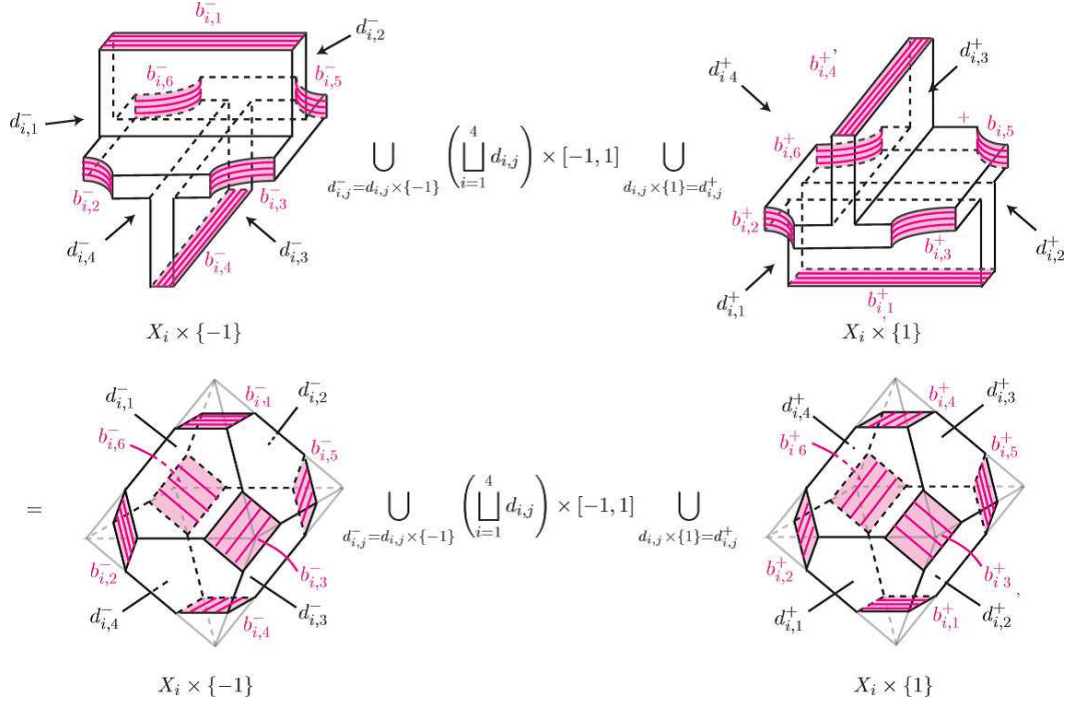}
\caption{$W_i \cap \partial W$. Topologically, the part $( \bigsqcup_{i=1}^4 d_{i,j} ) \times [ -1, 1 ]$ can be ignored 
and we may regard $W_i \cap \partial W$ as two truncated octahedra $X_i \times \{  \pm 1 \}$ glued 
as $d_{i,j}^- = d_{i,j}^+$ ($j \in \{ 1, 2, 3, 4\}$).}
\label{fig:octahedra}
\end{center}
\end{figure} 
In the figure, $d_{i, j} \times \{ \pm 1 \}$ ($j \in \{ 1,2,3,4 \}$) is denoted by $d_{i,j}^\pm$ and 
$b_{i, j} \times \{ \pm 1 \}$ ($j \in \{ 1,2, \ldots ,6\}$) is denoted by $b_{i,j}^\pm$. 
Moreover each $b_{i,j}^\pm$ inherits a foliation from $b_{i,j}$. 
 
Let $X = \bigcup_{i=1}^{c(P)}  X_i$ be the 3-dimensional thickening of $\Nbd ( S(P) ; P )$. 
Let $\rho \searrow \Nbd (S(P) ; P)$ be the collapsing given by $\rho_i$ ($i \in \{ 1 , 2, \ldots, c(P) \}$), 
and let $\pi : W_0 \searrow \Nbd (S(P) ; P)$ be the collapsing given by $\rho$ and 
the projection $[-1, 1] \to \{0\}$. 
Set $D_i = R_i \setminus \Int \, \Nbd (S(P), P)$  ($i \in \{ 1 , 2, \ldots, c(P) \}$). 
We note that $P$ is obtained from $\Nbd ( S(P) ; P )$ by attaching each disk $D_i$ along $l_i$. 
Suppose the boundary of the closure of $R_i$ runs $k_i$ times through the vertices of $P$. 
Then the preimage $\rho^{-1} (l_i)$, 
which is either an annulus or a M\"{o}bius band, 
is the union $\bigcup_{i=1}^{k_i} b_{\sigma (i) , \tau (i) }$
with $\sigma (i) \in \{ 1$, $2, \ldots , c(P) \}$ and
$\tau (i) \in \{ 1$, $2, \ldots , 6 \}$. 
If $\rho^{-1} (l_i)$ is an annulus, it lies in the solid torus 
$\pi^{-1} (l_i)$ as shown in Figure \ref{fig:annulus_framing}. 
\begin{figure}[htbp]
\begin{center}
\includegraphics[width=13cm,clip]{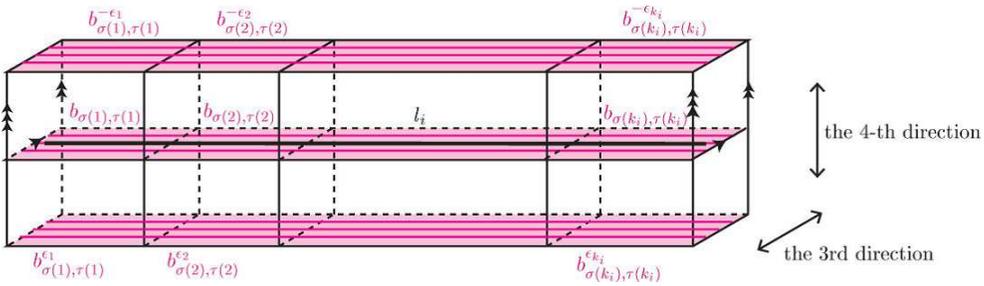}
\caption{The solid torus $\pi^{-1} (l_i)$ in the case where $\rho^{-1} (l_i)$ is an annulus. In this figure, $\epsilon_i = \pm$.}
\label{fig:annulus_framing}
\end{center}
\end{figure} 
Then the annulus $\bigcup_{i=1}^{k_i} b_{\sigma (i) , \tau (i) }$ is parallel to 
$\bigcup_{i=1}^{k_i} b_{\sigma (i) , \tau (i) }^{\epsilon_i}$ in the solid torus $\pi^{-1} (l_i)$. 
Hence by the definition of gleams, 
attaching the disk $D_i $ to $P'$ along $l_i$ and putting the gleam $0$ on the resulting region $R_i$ correspond to 
performing Dehn filling on $X$ 
along a leaf of the foliation on $\bigcup_{i=1}^{k_i} b_{\sigma (i) , \tau (i) }^{\epsilon_i}$, 
which lies in the torus $T_i = \partial \pi^{-1} (l_i) \subset \partial X$. 
Moreover,  putting the gleam $g_i \in \Integer$ on $R_i$ in this process corresponds to 
performing Dehn filling on $X$ 
along the slope $s_i$ obtained from the leaf by $g_i$-th power of the Dehn twist about the meridian of the solid torus $\pi^{-1} (l_i)$. 
Note that this slope $s_i$ can be expressed by $g_i [\mu_i] + [\lambda_i] \in H_1 (T_i ; \Integer)$, 
where $\mu_i$ is a meridian of the solid torus $\pi^{-1} (l_i)$ 
and $\lambda_i$ is a leaf of the foliation on $\bigcup_{i=1}^{k_i} b_{\sigma (i) , \tau (i) }^{\epsilon_i}$.
If $\rho^{-1} (l_i)$ is a M\"{o}bius band, it lies in the solid torus 
$\pi^{-1} (l_i)$ as shown in Figure \ref{fig:mobius_band_framing}. 
\begin{figure}[htbp]
\begin{center}
\includegraphics[width=13cm,clip]{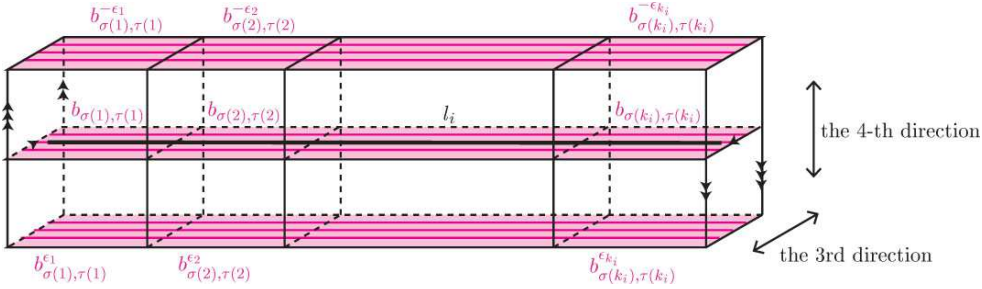}
\caption{The solid torus $\pi^{-1} (l_i)$ in the case where $\rho^{-1} (l_i)$ is a M\"obius band.}
\label{fig:mobius_band_framing}
\end{center}
\end{figure} 
In this case, the M\"{o}bius band  $\bigcup_{i=1}^{k_i} b_{\sigma (i) , \tau (i) }$ is no longer 
parallel to the boundary of the solid torus $\pi^{-1} (l_i)$. 
Let $a_{\sigma (1) , \tau (1)}^{\pm}$ be a foliated square shown in Figure \ref{fig:twisted_band}. 
\begin{figure}[htbp]
\begin{center}
\includegraphics[width=10cm,clip]{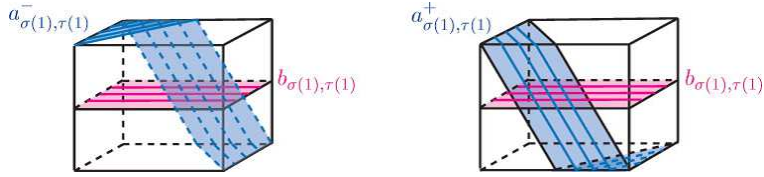}
\caption{The foliation when the gleam is a half-integer.}
\label{fig:twisted_band}
\end{center}
\end{figure} 
Then the union $a_{\sigma (1) , \tau (1)}^\pm \cup (\bigcup_{i=2}^{k_i} b_{\sigma (i) , \tau (i) } )$ 
is an annulus that is isotopic to $\bigcup_{i=1}^{k_i} b_{\sigma (i) , \tau (i) }^{\epsilon_i}$ after 
a $\pm 1/2$-twist. 
Hence again by the definition of gleams, 
attaching the disk $D_i $ to $P'$ along $l_i$ and putting the gleam $\pm 1/2$ on the resulting region $R_i$ correspond to 
performing Dehn filling on $X$ along a leaf of the foliation on 
$a_{\sigma (1) , \tau (1)}^\pm \cup (\bigcup_{i=2}^{k_i} b_{\sigma (i) , \tau (i) } )$.  
Moreover, putting the gleam $g_i \in ( \frac{1}{2}\Integer ) \setminus \Integer$ on $R_i$ 
in this process corresponds to 
performing Dehn filling on $X$ 
along the slope $s_i$ obtained from a leaf of $a_{\sigma (1) , \tau (1)}^{+} \cup (\bigcup_{i=2}^{k_i} b_{\sigma (i) , \tau (i) } )$ 
by $(g_i - 1/2)$-th power of Dehn twist about the meridian of the solid torus $\pi^{-1} (l_i)$. 
Again, note that this slope $s_i$ can be written as $(g_i-1/2) [\mu _i] + [\lambda_i]  \in H_1 (T_i ; \Integer)$, 
where $\mu_i$ 
is a meridian of the solid torus $\pi^{-1} (l_i)$ 
and $\lambda_i$ is a leaf of the foliation on $a_{\sigma (1) , \tau (1)} ^+ \cup (\bigcup_{i=2}^{k_i} b_{\sigma (i) , \tau (i) } )$.

As is explained in \cite{CT08}, we can equip the interior of $X$ with a complete, finite volume hyperbolic structure 
by regarding $X$ as $2 c(P)$ regular ideal octahedra glued by isometries. 
The decomposition of the interior of $X$ into these ideal octahedra exactly corresponds to 
the decomposition described in Figure \ref{fig:octahedra}. 
We recall that the regular ideal octahedron has a maximal horocusp section consisting of six Euclidean unit squares. 
This maximal horocusp gives rise to a maximal horoball neighborhood $C$ of the interior of $X$. 
By the above argument, the length of the slope $s_i$ with respect to $C$ coincides with $\mathrm{sl}  (R_i)$. 
\end{proof}

%\begin{theorem}
%\label{thm:lower bound of volume}
%Let $P \subset W$ be a special shadow of a closed orientable $3$-manifold $M$. 
%Let $R_1$, $R_2 , \ldots, R_n$ be the regions of $P$ and 
%$w_1$, $w_2 , \ldots , w_n$ be the weights of them. 
%Set $w_{\mathrm{min}} = \min_{1 \leqslant i \leqslant n} \{ w_i \}$. 
%If $w_{\mathrm{min}} > 2 \pi$, then we have 
%\[
%2 c (P) \left(  1 - \left( \frac{2 \pi}{w_{\mathrm{min}}} \right)^2 \right)^{3/2}  V_{\mathrm{oct}} 
%\leqslant  
%\vol (M) 
%\leqslant 
%2 \mathrm{sc} (M) V_{\mathrm{oct}} . 
%\]
%\end{theorem}

\section*{Acknowledgments} 
The authors wish to express their gratitude to 
Francesco Costantino, Bruno Martelli and Osamu Saeki 
for very helpful suggestions and comments. %during the preparation of the paper.
This work was carried out while the second-named author was visiting
Universit\`a di Pisa as a
JSPS Postdoctoral Fellow for Reserch Abroad.	
He is grateful to the university and its staffs for
the warm hospitality.

\end{document}